\documentclass[a4paper,11pt]{amsart}
\usepackage[left=2.7cm,right=2.7cm,top=3.5cm,bottom=3cm]{geometry}

\usepackage{amsthm,amssymb,amsmath,amsfonts,mathrsfs,amscd}
\usepackage[latin1]{inputenc}
\usepackage[all]{xy}
\usepackage{latexsym}
\usepackage{longtable}
\usepackage{calligra}

\theoremstyle{plain}
\newtheorem{theorem}{Theorem}[section]
\newtheorem{proposition}[theorem]{Proposition}
\newtheorem{lemma}[theorem]{Lemma}
\newtheorem{corollary}[theorem]{Corollary}

\theoremstyle{definition}

\theoremstyle{remark}
\newtheorem{remark}[theorem]{Remark}

\newtheorem{example}[theorem]{Example}

\newcommand{\Q}{\mathbb Q}

\newcommand{\Z}{\mathbb Z}
\newcommand{\R}{\mathbb R}
\newcommand{\C}{\mathbb C}

\DeclareMathOperator{\Spec}{Spec}
\DeclareMathOperator{\Sp}{Sp}

\DeclareMathOperator{\End}{End}

\DeclareMathOperator{\Hom}{Hom}

\DeclareMathOperator{\Gal}{Gal}
\DeclareMathOperator{\GL}{GL}
\DeclareMathOperator{\PGL}{PGL}
\DeclareMathOperator{\SL}{SL}

\DeclareMathOperator{\M}{M}
\DeclareMathOperator{\CH}{CH}

\DeclareMathOperator{\Sym}{Sym}
\DeclareMathOperator{\SFD}{SFD}
\DeclareMathOperator{\Spf}{Spf}
\DeclareMathOperator{\Lie}{Lie}
\DeclareMathOperator{\AJ}{AJ}

\newcommand{\tr}{\mathrm{tr}}
\newcommand{\ord}{\mathrm{ord}}

\newcommand{\norm}{\mathrm{nr}}

\newcommand{\unr}{\mathrm{unr}}
\newcommand{\rig}{\mathrm{rig}}
\newcommand{\univ}{\mathrm{univ}}

\newcommand{\Hp}{\mathcal{H}_p}
\newcommand{\hatHp}{\hat{\mathcal{H}}_p}
\newcommand{\Nilp}{\mathrm{Nilp}}
\newcommand{\proNilp}{\mathrm{pro\text{-}Nilp}}

\newcommand{\can}{\mathrm{can}}
\newcommand{\Liecal}{\mathcal{L}ie}
\newcommand{\HH}{\mathcal{H}}
\newcommand{\cl}{\mathrm{cl}}

\newfont{\cyr}{wncyr10 scaled 1100}
\newfont{\cyrr}{wncyr9 scaled 1000}

\usepackage[usenames]{color}
\definecolor{Indigo}{rgb}{0.2,0.1,0.7}
\definecolor{Violet}{rgb}{0.5,0.1,0.7}
\definecolor{White}{rgb}{1,1,1}
\definecolor{Green}{rgb}{0.1,0.9,0.2}

\newcommand{\longmono}{\mbox{\;$\lhook\joinrel\longrightarrow$\;}}

\newcommand{\longepi}{\mbox{\;$\relbar\joinrel\twoheadrightarrow$\;}}

\newcommand{\smallmat}[4]{\bigl(\begin{smallmatrix}#1&#2\\#3&#4\end{smallmatrix}\bigr)}

\include{thebibliography}
\begin{document}

\title[Shimura-Mass operator]{A $p$-adic Shimura-Maass operator on Mumford curves}
\today
\author{Matteo Longo}

\thanks{}

\begin{abstract} We study a $p$-adic Shimura-Maass operator in the context of Mumford curves 
defined by C. Franc in \cite{Franc}. We prove that this operator arises from a splitting of the Hodge filtration, 
thus answering a question in \cite{Franc}. We also study the relation of this operator with 
generalized Heegner cycles, in the spirit of \cite{BDP}, \cite{Brooks}, \cite{Kritz} and \cite{AI}. 
\end{abstract}

\address{Dipartimento di Matematica, Universit\`a di Padova, Via Trieste 63, 35121 Padova, Italy}
\email{mlongo@math.unipd.it}

\subjclass[2010]{}


\maketitle

\tableofcontents

\section{Introduction}

The main purpose of this paper is to study in the context of Mumford curves a $p$-adic variant of the Shimura-Maass operator, and relate it to generalized Heegner cycles. 

The real analytic Shimura-Maass operator is defined by the formula 
\begin{equation}\label{intro SM}
\delta_k\left(f(z)\right)=\frac{1}{2\pi i}\left(\frac{\partial}{\partial z}+\frac{k}{z-\bar{z}}\right)f(z)\end{equation}
where $z$ is a variable in the complex upper half plane $\mathcal{H}$, $f(z)$ is a real analytic modular form of weight $k$, and $z\mapsto \bar{z}$ denotes the complex conjugation; here $\delta_k\left(f(z)\right)$ is a real analytic modular form of weight $k+2$. 
The relevance of this operator arises in studyin algebraicity properties of Eisenstein series and $L$-functions: see Shimura \cite{Shimura-ArithProp}, Hida \cite[Chapter 10]{HidaElementary}. One of the main results in \cite{Shimura-ArithProp} is the following. Let 
\[\delta_k^r=\delta_{k+2(r-1)}\circ\delta_{k+2(r-2)}\circ\dots\circ\delta_k\] for any $r\geq 1$, and let $K$ be an imaginary quadratic field. Then there exists $\Omega_K\in\C^\times$ such that for every CM point $z\in K\cap\mathcal{H}$, every congruence subgroup $\Gamma\subseteq\SL_2(\Z)$, every integer $k\geq 0$, $r\geq 1$, and every modular form of weigh $k$ and level $\Gamma$ with algebraic Fourier coefficients we have 
\begin{equation}\label{intro rationality}
\frac{\delta_k^r(f)(z)}{\Omega_K^{(k+2r}}\in\bar{\Q}.\end{equation}

Katz described in \cite{KatzIHES} the Shimura-Maass operator in more abstract terms by means of the Gauss-Manin connection (see also \cite{KO}).  
More precisely. let $N\geq1$ be an integer, 
$X_1(N)$ the modular form of level $\Gamma_1(N)$ over $\Q$, and let $\pi:\mathcal{E}\rightarrow X_1(N)$ be the universal elliptic curve. Consider the relative de Rham cohomology sheaf \[\mathcal{L}_1=\R^1\pi_*\left(0\rightarrow\mathcal{O}_\mathcal{E}\rightarrow\Omega^1_{\mathcal{E}/X_1(N)}\right)\] on $X_1(N)$, and define $\mathcal{L}_r=\mathrm{Sym}^r(\mathcal{L}_1)$. 
Let $\underline\omega=\pi_*\left(\Omega^1_{\mathcal{E}/X_1(N)}\right)$. 
The sheaf $\underline{\omega}$ is invertible and we have 
the Hodge filtration 
\begin{equation}\label{intro Hodge}
0\rightarrow\underline\omega\rightarrow\mathcal{L}_1\rightarrow\underline{\omega}^{-1}\rightarrow 0.\end{equation} Once we are given a splitting $\Psi$ of the Hodge filtration \eqref{intro Hodge}, one may define by means of the Gauss-Manin connection and the Kodaira-Spencer map an operator 
$\Theta^{(r)}_\Psi\colon\underline{\omega}^r\rightarrow\underline\omega^{r+2}$, for any even integer $r\geq 2$.  
In particular, considering the associated real analytic sheaves, which we denote by a superscript $^\mathrm{ra}$, the Hodge exact sequence admits a splitting 
\begin{equation}\label{intro real analytic splitting}
\Psi_\infty:\mathcal{L}_1^\mathrm{ra}\simeq\underline{\omega}_\mathrm{ra}\oplus\underline{\bar{\omega}}_\mathrm{ra},\end{equation} 
where $\underline{\bar{\omega}}_\mathrm{ra}$ is obtained from $\underline{\omega}$ by applying the complex conjugation.
The Shimura-Maass operator can then be described as the map $\Theta^{(r)}_{\Psi_\infty}$ 
on real analytic differentials by appealing to the general procedure alluded to above applied to the real analytic splitting \eqref{intro real analytic splitting}.  For details on this construction, the reader is referred to 
\cite[\S1.8]{KatzCM} and \cite[\S 1.2]{BDP}; for the case of Siegel modular forms, see  \cite[\S4]{Harris} while for the case of Shimura curves see \cite[\S 3]{Brooks}, \cite[\S 2]{Mori}.


As hinted from the above discussion, Katz description of the Shimura-Maass operator rests on the fact that the real analytic Hodge sequences \eqref{intro Hodge} splits. In \cite[\S 1.11]{KatzCM}, Katz introduces a $p$-adic analogue of this splitting. Suppose that $p\nmid N$ is a prime number, and let $X_1^\ord(N)$ denote the ordinary locus of the modular curve, viewed as a rigid analytic scheme over $\Q_p$. Let $\mathcal{F}^\rig$ be the rigid analytic sheaf associated with a sheaf $\mathcal{F}$ on $X_1(N)$. 
Then $\mathcal{L}_1^\rig$ splits over $X_1^\ord(N)$ as the direct sum  
\[\Psi_p:\mathcal{L}_1^\rig\simeq\underline\omega^\rig\oplus\mathcal{L}_1^\mathrm{Frob}\] 
(where $\mathcal{L}_1^\mathrm{Frob}$ has the property that the Frobenius endomorphism acts on this sheaf invertibly). This allows to define a differential operator $\Theta_{\Psi_p}^{(r)}$, which can be seen as a $p$-adic analogue of the Shimura-Maass operator; this operator can be also described in terms of Atkin-Serre derivative. At CM points the splittings $\Psi_\infty$ and $\Psi_p$ coincide, and therefore one deduces by comparison rationality results for the values of $\Theta_{\Psi_p}^{(r)}$ at CM points from {intro rationality}. For details, see \cite[Proposition 1,12]{BDP}. The $p$-adic Shimura-Maass operator is then used in \cite{KatzCM} and \cite{BDP} to construct $p$-adic $L$-functions and study their properties.   

We now fix an integer $N$, a prime $p\nmid N$, and a quadratic imaginary field in which $p$ is inert. In this context, Kritz introduced in \cite{Kritz} for modular forms of level $N$ a new $p$-adic Shimura-Maass operator by using perfectoid techniques, and define $p$-adic $L$-functions by means of this operator, thus removing the crucial assumption that $p$ is split in $K$, but keeping the assumption that $p$ is a prime of good reduction for the modular curve. Moreover, Andreatta-Iovita \cite{AI} introduced still an other $p$-adic Shimura-Maass operator in \cite{AI}, and obtained results analogue to \cite{BDP}, thus extending their 
work to the non-split case. 

On the other hand, Franc in his thesis \cite{Franc} proposed still an other $p$-adic Shimura-Maass operator for primes $p$ which are inert in $K$, in the following context. Let $N\geq1$ be an integer, $K/\Q$ a quadratic imaginary field, $p\nmid N$ a prime number which is inert in $K$, and let $Np=N^+\cdot N^-p$ be a factorization of $Np$ into coprime integers such that $N^+$ is divisible only by primes which are split in $K$, and $N^-p$ is a square-free product of an even number of primes factors which are inert in $K$. Let $\mathcal{B}$ be the indefinite quaternion of discriminant $N^-$, $\mathcal{R}$ an Eichler order of $X$ of level $N^+$, and $\mathcal{C}$ the Shimura curve attached to $(\mathcal{B},\mathcal{R})$. The rigid analytic curve $X^\rig$ over $\Q_p$ is then a Mumford curve, namely $X^\ord(\C_p)$ is isomorphic to the rigid analytic quotient of the $p$-adic upper half plane $\mathcal{H}_p(\Q_p)=\C_p-\Q_p$ by an arithmetic subgroup $\Gamma\subseteq\SL_2(\Q_p)$.
Franc defines in this context a $p$-adic Shimura-Maass operator $\delta_{p,k}$ 
by mimicking the definition \eqref{intro SM} and formally replacing the variable $z\in\mathcal{H}$ with the $p$-adic variable $z\in\mathcal{H}_p(\hat{\Q}_p^\unr)=\hat{\Q}_p^\unr-\Q_p$, and replacing the complex conjugation with the Frobenius map (here $\hat{\Q}_p^\unr$ is the completion of the maximal unramified extension of $\Q_p$). Following the arguments of \cite{Shimura-ArithProp}, Franc proves statement analogue to \eqref{intro rationality} (see \cite[Theorem 5.1.5]{Franc}). 

In \cite[\S 6.1.3]{Franc}, Franc asks for a construction of his $p$-adic Shimura-Maass operator by means of a (non-rigid analytic) splitting $\Psi_p$ of the Hodge filtration, similar to what happens over $X_1(N)$ (in the real analytic case \cite{Shimura-ArithProp}) and $X_1^\ord(N)$ (in the $p$-adic rigid analytic case \cite{KatzCM}). The first result of this paper is to provide such a splitting $\Psi_p$, and define the associated $p$-adic Shimura-Maass operator. 
In particular, we show that our splitting $\Psi_p$ coincides at CM points with the Hodge splitting $\Psi_\infty$, and therefore, as in \cite{KatzCM}, we reprove the main results of \cite{Franc} by the comparison of the two Shimura-Mass operators. We also derive a relation between our $p$-adic Shimura-Maass operator and generalized Heegner cycles in the context of Mumford curves, which can be vied as an analogue of \cite[Proposition 3.24]{BDP}. In the remaining part of the introduction we describe more precisely the results of this paper.  

Instead of the curve $X$ attached to the Eichler order $\mathcal{R}$, we follow \cite{Brooks} and consider a covering $\mathcal{C}\rightarrow X$ where $\mathcal{C}$ is a geometrically connected curve defined over $\Q$ corresponding to a $\Gamma_1(N^+)$-level structure subgroup of $\hat{\mathcal{R}}^\times$, where $\hat{\mathcal{R}}=\mathcal{R}\otimes_\Z\hat{\Z}$ is the profinite completion of $\mathcal{R}$. The advantage of using $\mathcal{C}$ is that $\mathcal{C}$ is the solution of a moduli problem, and we have a universal false elliptic curve $\pi:\mathcal{A}\rightarrow\mathcal{C}$ (see \S\ref{Brooks moduli problem}). Following \cite{Hashimoto}, \cite{Mori}, \cite{Brooks}, we define a quaternionic projector $e$, acting on the relative de Rham cohomology of $\pi:\mathcal{A}\rightarrow\mathcal{C}$, and define the sheaf
\[\mathcal{L}_1=e\cdot\mathcal{H}^1_\mathrm{dR}(\mathcal{A}/\mathcal{C})\]
and the line bundle 
\[\underline{\omega}=e\cdot \pi_*\left(\Omega_{\mathcal{A}/\mathcal{C}}^1\right).\]
We have a corresponding Hodge filtration 
\[0\longrightarrow \underline{\omega}\longrightarrow\mathcal{L}_1\longrightarrow\underline{\omega}^{-1}\longrightarrow 0.\]
The rigid analytic curve $\mathcal{C}^\rig$ associated with $\mathcal{C}$ admits a
$p$-adic uniformization 
\[\mathcal{C}^\rig(\C_p)\simeq\Gamma\backslash\mathcal{H}_p(\C_p)\]
for a suitable subgroup $\Gamma\subseteq\SL_2(\Q_p)$. Modular forms on $\mathcal{C}^\rig$ are then $\Gamma$-invariant sections of $\mathcal{H}_p$,  and therefore, to define a $p$-adic Shimura-Maass operator on $\mathcal{C}$ one is naturally led to consider the analogue problem for $\mathcal{H}_p$. 

Let $\mathcal{C}^0$ denote the $\C_p$-vector space of continuous (for the standard $p$-adic topology on both spaces) $\C_p$-valued functions on ${\mathcal{H}_p(\hat{\Q}_{p}^\unr)}$, and let 
$\mathcal{A}$ denote the $\hat{\Q}_p^\unr$-vector space of rigid analytic global sections of $\mathcal{H}_p(\hat{\Q}_p^\unr)$. We have a map of $\hat{\Q}_p^\unr$-vector spaces 
$r:\mathcal{A}\rightarrow\mathcal{C}^0$ and, following \cite{Franc}, 
we denote $\mathcal{A}^*$ the image of the morphism of $\mathcal{A}$-algebras 
$\mathcal{A}[X,Y]\rightarrow\mathcal{C}^0$ defined by sending $X$ to the 
function $z\mapsto 1/(z-\sigma(z))$ and $Y$ to the function 
$z\mapsto\sigma(z)$, where $\sigma:\hat{\Q}_p^\unr\rightarrow\hat{\Q}_p^\unr$ 
is the Frobenius automorphism (note that the function $z\mapsto z-\sigma(z)$ is invertible on $\mathcal{H}_p(\hat{\Q}_p^\unr)$). Denote $\hat{\mathcal{H}}_p$ the formal $\Z_p$-scheme whose generic fiber is $\mathcal{H}_p$, let $\hat{\mathcal{H}}_p^\unr$ be its base change to $\hat{\Q}_p^\unr$ and  
let $\mathcal{G}\rightarrow\hat{\mathcal{H}}_p^\unr$ be the universal $\SFD$-module. 
Denote $\underline\omega_{\mathcal{G}}=e_\mathcal{G}^*(\Omega^1_{\mathcal{G}/\hat{\mathcal{H}}_p^\unr})$,
where $e_\mathcal{G}:\hatHp^\unr\rightarrow\mathcal{G}$ is the zero-section, 
and let 
$\Liecal _{\mathcal{G}^\vee}$
be the Lie algebra of the Cartier dual $\mathcal{G}^\vee$ of $\mathcal{G}$.
Then $\underline\omega_{\mathcal{G}}$ and $\mathcal Lie_{\mathcal{G}^\vee}$ are locally free $\mathcal{O}_{\hatHp^\unr}$-modules, 
dual to each other and we have the Hodge-Tate exact sequence 
of $\mathcal{O}_{\hat{\mathcal{H}}_p^\unr}$-modules 
\[0\longrightarrow \underline\omega_{\mathcal{G}} \longrightarrow
\mathcal{H}^{1 }_\mathrm{dR}(\mathcal{G}/\hatHp^\unr)\longrightarrow 
\Liecal _{\mathcal{G}^\vee}\longrightarrow0.\]
Set $\mathcal{L}_\mathcal{G}^0=e\cdot\mathcal{H}^1_\mathrm{dR}(\mathcal{G}/\hat{\mathcal{H}_p^\unr})$ and $\omega_\mathcal{G}^0=e\cdot\underline\omega_\mathcal{G}$. 
Define $\Lambda_\mathcal{G}=H^0(\mathcal{H}_p^\unr,\mathcal{L}_\mathcal{G}^0)$, 
$\Lambda_{\mathcal{G}}^\ast=\Lambda_{\mathcal{G}}\otimes_\mathcal{A}\mathcal{A}^\ast$, $w_\mathcal{G}=H^0(\mathcal{H}_p^\unr,\underline\omega_\mathcal{G}^0)$, ${w}_{\mathcal{G}}^\ast=w_{\mathcal{G}}\otimes_\mathcal{A}\mathcal{A}^\ast$. 
We have then an injective map of $\mathcal{A}^*$-algebras 
\begin{equation}\label{intro split}
{w}_{\mathcal{G}}^\ast\longmono \Lambda_{\mathcal{G}}^\ast.\end{equation}

\begin{theorem}\label{intro theorem1}
The injection \eqref{intro split} of $\mathcal{A}^*$-algebras admits a canonical splitting 
$\Psi_p^*: \Lambda_{\mathcal{G}}^\ast\rightarrow {w}_{\mathcal{G}}^\ast$. 
\end{theorem}

This is the main result of this paper, Theorem \ref{splitting theorem}. 
We may then attach to $\Psi_p^*$ a $p$-adic Shimura-Maass operator $\Theta_p^{(r)}$. 
We have the following two corollaries. 

\begin{corollary}\label{coro intro1} The $p$-adic Shimura-Maass operator $\delta_{p,k}$ defined by Franc in \cite{Franc} coincides with the $p$-adic Shimura-Maass operator $\Theta_p^{(r)}$ defined 
by means of the splitting in Theorem $\ref{intro theorem1}$.  
\end{corollary} 

The main tool which is used to prove Theorem \ref{intro theorem1} and 
Corollary \ref{coro intro1} is Drinfel'd interpretation of $\hat{\mathcal{H}}_p^\unr$ 
as moduli space of special formal modules with quaternionic multiplication; 
following \cite{Tei-univ}, we call these objects $\SFD$-modules.  
We study the relative de Rham cohomology of the universal $\SFD$-module 
$\mathcal{G}\rightarrow\hat{\mathcal{H}}_p^\unr$ by means of techniques from \cite{Tei-univ}, \cite{Faltings} and \cite{IS}. The upshot of our analysis is an explicit description of the Gauss-Manin connection and the Kodaira-Spencer isomorphism 
for $\mathcal{G}\rightarrow\hat{\mathcal{H}}_p^\unr$, once we apply to the relevant sheaves the projector $e$.  This detailed study is contained in Section \ref{Drinfelderia}, which we believe is of independent interest and is the technical hart of the paper. 

For the next corollary, define as in the real analytic case  
\[\delta_{p,k}^r=\delta_{p,k+2(r-1)}\circ\delta_{p,k+2(r-2)}\circ\dots\circ\delta_k.\] Moreover, 
we fix an embedding $\bar{\Q}\hookrightarrow\bar{\Q}_p$, and we say that $\xi\in\bar{\Q}_p$ belongs to $\bar{\Q}$ if $\xi$ belongs to the image of this embedding. 

\begin{corollary}\label{coro intro2} Let $f$ be a modular form on $\mathcal{C}$. 
Then there exists $t_p\in\C_p^\times$, independent of $f$, 
such that for every CM point $z\in K\cap\mathcal{H}_p(\hat{\Q}_p^\unr)$ we have  
$\frac{\delta_{p,k}^r(f)(z)}{t_p^j}\in\bar{\Q}$.  
\end{corollary} 

As remarked above, this is the main result of Franc thesis \cite{Franc}, which he proves via an explicit approach following Shimura. Instead, we derive this result in Theorem \ref{comparison theorem} from a comparison between the values at CM points of 
our $p$-adic Shimura-Maass operator $\Theta_p^{(r)}$ and the 
real analytic Shimura-Maass operator $\Theta_\infty^{(r)}$. 

We explain now the connection with generalized Heegner cycles. These cycles were introduced in \cite{BDP} with the aim of studying certain anticyclotomic $p$-adic $L$-functions. Generalized Heegner cycles have been also studied in the context of Shimura curves with good reduction at $p$ by \cite{Brooks}, and in the context of Mumford curves in \cite{Masdeu}, \cite{LP2}. In this paper we introduce still an other variant of Generalized Heegner cycles. Fix a false elliptic curve $A_0$ with CM by $\mathcal{O}_K$. To any isogeny $\varphi:A_0\rightarrow A$, where $A$ is a false elliptic curve, 
we construct a cycle $\Upsilon_\varphi$ in the Chow group $\CH^m(\mathcal{A}\times A_0)$ of the Chow motive $\mathcal{A}\times A_0$, 
where $m=n/2$ with $n=k-2$. The work of Brooks \cite{Brooks} gives us a projector 
$\epsilon$ in the ring of correspondences of $X_m=\mathcal{A}^m\times A_0^m$, 
which defines the motive $\mathcal{D}=(X^m,\epsilon)$.  
The generalised Heegner cycle $\Delta_\varphi$ in the image of $\Upsilon_\varphi$ 
in $\CH^m(\mathcal{D})$ via this projector. 
Let $M_k(\Gamma)$ be the $\C_p$-vector space of rigid analytic quaternionic modular forms of weight $k$ and level $\Gamma$; elements of $M_k(\Gamma)$ are functions from $\mathcal{H}_p(\C_p)=\C_p-\Q_p$ to $\C_p$ which transform under the action of $\Gamma$ by the automorphic factor of weight $k$. 
We construct a $p$-adic Abel-Jacobi map 
\[\AJ_p:\CH^m(\mathcal{D})\longrightarrow \left(M_k(\Gamma)\otimes\mathrm{Sym}^meH^1_\mathrm{dR}(A_0)\right)^\vee\] where $^\vee$ denotes $\C_p$-linear dual. 
If follows from our work that $\mathcal{L}_\mathcal{G}^0=e\mathcal{H}^1_\mathrm{dR}(\mathcal{G}/\hat{\mathcal{H}}_p^\unr)$ is equipped with two canonical sections $\omega_\can$ and $\eta_\can$, such that $\omega_\can$ is a generator of the 
invertible sheaf $\omega_\mathcal{G}^0$. 
Let $\omega_f\in w_\mathcal{G}^0$ be the $\Gamma$-invariant differential form  associated with $f\in M_k(\Gamma)$, and let $F_f$ its Coleman primitive satisfying $\nabla(F_f)=\omega_f$, where $\nabla$ is the Gauss-Manin connection. 
Denote 
$\langle,\rangle$ is the Poincar\'e pairing on $\mathrm{Sym}^neH^1_\mathrm{dR}(A_z)$, 
where $A_z$ is the fiber of $\mathcal{A}$ at $z$.  
Define the function 
\[H(z)=\langle F_{f}(z),\omega_\mathrm{can}^n(z)\rangle.\]

\begin{theorem}\label{intro theorem2} Let $\varphi:A_0\rightarrow A$ be an isogeny and $z_A$ the fiber of 
$A$ of $\mathcal{A}\rightarrow\mathcal{C}$. Then for each integer $j=n/2,\dots,n$ we have 
\[\delta_{p,k}^{n-j}(H_n)(z_A)=\AJ_p(\Delta_\varphi)(\omega_f\otimes\omega^j_\can\eta^{n-j}_\can).\]
\end{theorem}

Theorem \ref{intro theorem2} relates the Shimura-Maass operator with generalised Heegner cycles, and corresponds to Corollary \ref{coro9.6}.

We finally make a remark on $p$-adic $L$-functions. It would be interesting to use our $p$-adic Shimura-Maass operator to construct $p$-adic $L$-functions interpolating special values of the complex $L$-function of $f$ twisted by Hecke characters as in \cite{BDP}, \cite{Brooks}, \cite{Kritz}, \cite{AI}. We would like to come back to this problem in a future work.

\section{Algebraic de Rham cohomology of Shimura curves}

Throughout this section, let $k\geq 2$ be an even integer and 
$N\geq 1$ an integer. 
Fix an imaginary quadratic 
field $K/\Q$ of discriminant $D_K$ prime to $N$ 
and factor $N=N^+\cdot N^-$ by requiring that 
all primes dividing $N^+$ (respectively $N^-$) split in $K$ (respectively, 
are inert in $K$). 
Assume that $N^-$ is a square-free of an odd number of primes, and let $p\nmid N$ be a prime number 
which is inert in $K$ (thus $N^-p$ is a square-free of an even number of primes). Let $f\in S_k(\Gamma_0(Np))$ be a weight $k$ newform of level $\Gamma_0(Np)$. 
Fix also embeddings $\bar{\Q}\hookrightarrow\C$ and $\bar{\Q}\hookrightarrow\bar{\Q}_\ell$ for each prime number $\ell$. 

\subsection{Quaternion algebras}\label{quaternion algebras}
Fix a indefinite quaternion algebra $\mathcal{B}/\Q$ of discriminant $N^-p$,
a maximal order $\mathcal{R}_\mathrm{max}\subseteq\mathcal{B}$ and 
an Eichler order $\mathcal{R}\subseteq \mathcal{R}_\mathrm{max}$ 
of level $N^+$.
Fix isormorphisms $\iota_\ell:\mathcal{B}_\ell=\mathcal{B}\otimes_\Q\Q_\ell\simeq\M_2(\Q_\ell)$ 
for each prime $\ell\nmid N^-p$ such that $\iota_\ell(\mathcal{R}_\mathrm{max}\otimes_\Z\Z_\ell)$ is the subgroup $\M_2(\Z_\ell)$ and moreover for each prime $\ell\mid N^+$ we require 
that $\iota_\ell(\mathcal{R}\otimes_\Z\Z_\ell)$ is the subgroup consisting of matrices 
which are upper triangular modulo $N^+$. Define $V_1(N^+)$ to be the subgroup of $\hat{\mathcal{R}}^\times_\mathrm{max}$ to 
consisting of elements $x$ such that $\iota_\ell(x)\equiv\smallmat{*}{*}{0}{1}\pmod{N^+}$ (as a general notation, for any 
$\Z$-algebra $A$, let $\hat{A}=A\otimes_\Z\hat{\Z}$, where $\hat{\Z}$ is the pro finite completion of $\Z$). 

We need to fix a convenient basis for the $\Q$-algebra $\mathcal{B}$, 
called \emph{Hashimoto model}. Denote $M=\Q(\sqrt{p_o})$ 
the splitting field of the quadratic polynomial $X^2-p_0$, 
where $p_0$ if an auxiliary prime number fixed as in \cite[\S 1.1]{Mori} and \cite[\S 2.1]{Brooks}, 
such that: 
\begin{enumerate}
\item for all primes $\ell$ we have 
$(p_0,pN^-)_\ell=-1$ if and only if $\ell\mid pN^-$, where 
$(a,b)_\ell$ denotes the Hilbert symbol,  
\item all primes $\ell\mid N^+$ are split in the real quadratic field $M=\Q(\sqrt{p_0})$, 
where $\sqrt{p_0}$ is a square root of $p_0$ in $\bar\Q$. 
\end{enumerate} 
The choice of $p_0$ fixes a $\Q$-basis of $\mathcal{B}$ 
as in \cite[\S2]{Hashimoto} 
given by $\{\mathbf{1},\mathbf{i},\mathbf{j},\mathbf{k}\}$ 
with $\mathbf{i}^2=-pN^-$, $\mathbf{j}^2=p_0$, $\mathbf{k}=\mathbf{ij}=-\mathbf{ji}$, and $\mathbf{1}$ the unit of $\mathcal{B}$; of course, if $x\in \Q$ we will often just write $x$ for $x\cdot\mathbf{1}$. 

\subsection{Moduli problem} \label{Brooks moduli problem}A \emph{false elliptic curve} $A$ over a scheme $S$ is a
an abelian scheme $A\rightarrow S$ of relative dimension $2$ 
equipped with an embedding $\iota_A:\mathcal{R}_\mathrm{max}\hookrightarrow\End_S(A)$.
An \emph{isogeny of false elliptic curves} is an isogeny which commutes 
with the action of $\mathcal{R}_\mathrm{max}$. A \emph{full level $N^+$-structure}
on $A$ is an isomorphism of group schemes $\alpha_A:A[N^+]\simeq(\mathcal{R}_\mathrm{max}\otimes_\Z(\Z/N^+\Z))_S$, where for any group $G$ we denote $G_S$ 
the constant group scheme $G$ over $S$. Note that $\hat{\mathcal{R}}^\times_\mathrm{max}$ 
acts from the right on $A[N^+]$ and $\mathcal{R}_\mathrm{max}\otimes_\Z(\Z/N^+\Z)$ through 
the canonical projection $\hat{\mathcal{R}}_\mathrm{max}^\times\rightarrow \prod_{\ell\mid{N^+}}(\mathcal{R}_\mathrm{max}\otimes_\Z\Z_\ell)^\times$. 
A \emph{level structure of $V_1(N^+)$-type} is an equivalence class of 
full level $N^+$-structures under the (right) action of $V_1(N^+)$.

The moduli problem 
which associates to any $\Z[1/Np]$-scheme $S$ 
the set of isomorphism classes of false elliptic curves equipped with
a $V_1(N^+)$-level structure 
is representable by 
a smooth proper scheme $\mathcal{C}$ defined over $\Spec(\Z[1/Np])$
(\cite[Theorem 2.2]{Brooks}). Let $\pi:\mathcal{A}\twoheadrightarrow\mathcal{C}$ be the 
universal false elliptic curve. For any $\Z[1/Np]$-algebra $R$, 
let $\pi_R:\mathcal{A}_R\twoheadrightarrow\mathcal{C}_R$ be the base change of $\pi$
to $R$. 

\subsection{Algebraic de Rham cohomology}\label{algebraic de Rham} We review some preliminaries 
on the algebraic de Rham cohomology of Shimura curves, including the Gauss-Manin connection 
and the Kodaira-Spencer map, referring for details 
to \cite{KO}, and especially to 
\cite[\S 2.1]{Mori} for the case under consideration of Shimura curves.

We first recall some general notation. For any morphism of schemes $\phi:X\rightarrow S$, denote 
$(\Omega^\bullet_{X/S},d^\bullet_{X/S})$, or simply $\Omega^\bullet_{X/S}$ understanding 
the differentials $d^\bullet_{X/S}$, 
the complex of sheaves of relative differential forms for the morphism $\phi$. For a sheaf $\mathcal{F}$ of $\mathcal{O}_X$-modules over a scheme $X$, we denote $\mathcal{F}^\vee$ its 
$\mathcal{O}_X$-linear dual and, for a positive integer $k$, we put 
$\mathcal{F}^{\otimes k}=\mathcal{F}\otimes_{\mathcal{O}_X}\dots\otimes_{\mathcal{O}_X}\mathcal{F}$ ($k$ factors). If $\mathcal{F}$ is 
invertible, we denote $\mathcal{F}^{-1}$ its inverse, and in this case for $k$ a negative integer, $\mathcal{F}^{\otimes{k}}$ denotes $(\mathcal{F}^{-1})^{\otimes{k}}$ as usual. 

Fix a field $F$ of characteristic zero. The relative de Rham cohomology bundle for the morphism $\mathcal{A}_F\overset{\pi_F}\rightarrow\mathcal{C}_F$ is defined by 
\[\mathcal{H}^q_\mathrm{dR}(\mathcal{A}_F/\mathcal{C}_F)=\mathbb{R}^q\pi_{F*}\left(\Omega^\bullet_{\mathcal{A}_F/\mathcal{C}_F}\right).\]  

We first recall the construction of the Gauss-Manin connection. 
We have a canonical short exact sequence of locally free sheaves
\begin{equation}\label{ex seq diff} 
0\longrightarrow \pi_F^*\left(\Omega^1_{\mathcal{C}_F/F}\right)\longrightarrow
\Omega^1_{\mathcal{A}_F/\mathcal{C}_F}\longrightarrow\Omega^1_{\mathcal{A}_F/F}\longrightarrow 0 
\end{equation}
(the exactness is because $\pi_F$ is smooth). 
This exact sequence induces maps 
\[\Omega^{\bullet-i}_{\mathcal{A}_F/F}\otimes_{\mathcal{O}_{\mathcal{A}_F}}\pi_{F}^*\left(\Omega^i_{\mathcal{C}_F/F}\right)\overset{\psi^i_{\mathcal{A}_F/F}}
\longrightarrow\Omega^\bullet_{\mathcal{A}_F/F}\]
for each integer $i$, 
defining a filtration $F^i\Omega^\bullet_{\mathcal{A}_F/F}=\mathrm{Im}(\psi^i_{\mathcal{A}_F/F})$ on $\Omega^\bullet_{\mathcal{A}_F/F}$ 
with associated graded objects
\[\mathrm{gr}^i\left(\Omega^\bullet_{\mathcal{A}_F/F}\right)=\Omega^{\bullet-i}_{\mathcal{A}_F/F}\otimes_{\mathcal{O}_{\mathcal{A}_F}}\pi_{F}^*\left(\Omega^i_{\mathcal{C}_F/F}\right).\] 
Let $E_i^{p,q}$ 
denote the spectral sequence associated 
with this filtration. The $E_1^{p,q}$ terms are then given by 
$E_1^{p,q}=\R^{p+q}\pi_{F*}\left(\mathrm{gr}^p\left(\Omega^\bullet_{\mathcal{A}_F/F}\right)\right)$. Since $\Omega^p_{\mathcal{C}_F/F}$ is locally free, 
and the differentials in the complex $\pi_{F}^*\left(\Omega_{\mathcal{C}_F/F}^p\right)\otimes_{\mathcal{O}_{\mathcal{A}_F}}\Omega^\bullet_{\mathcal{A}_F/\mathcal{C}_F}$ 
are $\pi^{-1}(\mathcal{O}_{\mathcal{C}_F})$-linear, one can show that 
\[E_1^{p,q}\simeq\Omega^p_{\mathcal{C}_F/F}\otimes_{\mathcal{O}_{\mathcal{C}_F}}\mathcal{H}^q_\mathrm{dR}(\mathcal{A}_F/\mathcal{C}_F)\] 
(\cite[(7)]{KO}).  
The Gauss-Manin connection 
\[\nabla:\mathcal{H}^i_\mathrm{dR}(\mathcal{A}_F/\mathcal{C}_F)\longrightarrow
\Omega^1_{\mathcal{C}_F/F}\otimes_{\mathcal{O}_{\mathcal{C}_F}}\mathcal{H}^i_\mathrm{dR}(\mathcal{A}_F/\mathcal{C}_F)\]
is then defined as the differential $d_1^{0,i}:E_{1}^{0,i}\rightarrow E_1^{1,i}$ 
in this spectral sequence. 

We now recall various descriptions of the Kodaira-Spencer map. 
It is defined to be the boundary map 
\[\mathrm{KS}_{\mathcal{A}_F/\mathcal{C}_F}: \pi_{F*}\left(\Omega^1_{\mathcal{A}_F/\mathcal{C}_F}\right)\longrightarrow 
\R^1\pi_{F*}\left(\pi_F^*\left(\Omega^1_{\mathcal{C}_F/F}\right)\right)\] in the long exact sequence 
of derived functors obtained from \eqref{ex seq diff}. It can also be 
reconstructed from the Gauss-Manin connection as follows. 
Let $\pi_F^\vee:\mathcal{A}_F^\vee\rightarrow\mathcal{C}_F$ 
denote the dual abelian variety. 
By a result of Buzzard (\cite[Section1]{BuzzardIntegral}, see also \cite[page 4183]{Brooks}), it is known that the abelian surface $\mathcal{A}_F$
is equipped with a canonical principal polarization $\iota_{\mathcal{A}_F}:\mathcal{A}_F\simeq\mathcal{A}_F^\vee$ over $\mathcal{C}_F$,  
which we use to identify $\mathcal{A}_F^\vee$ and $\mathcal{A}_F$ 
in the following without explicit mentioning it; we recall that this polarization 
is characterized by the fact that the associated Rosati involution in $\End_{\mathcal{C}_F}(\mathcal{A}_F)$ 
restricted to the image of $\mathcal{R}_\mathrm{max}$ via the 
map $\mathcal{R}_\mathrm{max}\hookrightarrow\End_{\mathcal{C}_F}(\mathcal{A}_F)$ 
coincides with the involution $x\mapsto x^\dagger$ of $\mathcal{R}_\mathrm{max}$, defined by $x^\dagger=i^{-1}\bar{b}i$ (as usual, 
if $x=a+bi+cj+dk$, then $\bar{x}=a-bi-cj-dk$).
Using the principal polarization and the isomorphism between $\R^1\pi_{F*}(\Omega^1_{\mathcal{A}_F/F})$ and the tangent bundle 
of $\mathcal{A}^\vee_F$, 
the Hodge exact sequence can be written as 
\begin{equation}\label{Hodge}
0\longrightarrow \pi_{F*}\left(\Omega^1_{\mathcal{A}_F/\mathcal{C}_F}\right)\longrightarrow\mathcal{H}^1_\mathrm{dR}(\mathcal{A}_F/\mathcal{C}_F)
\longrightarrow
\left(\pi_{F*}\left(\Omega^1_{\mathcal{A}_F/\mathcal{C}_F}\right)\right)^\vee
\longrightarrow0\end{equation}
(\emph{cf.} \cite[(2.2)]{Mori}, \cite[\S 2.6]{Brooks}). 
The Kodaira-Spencer map 
can be defined using the Gauss-Manin connection as the composition 
\[
\begin{split}\mathrm{KS}_{\mathcal{A}_F/\mathcal{C}_F}:
\pi_{F*}\left(\Omega^1_{\mathcal{A}_F/\mathcal{C}_F}\right)\overset{\eqref{Hodge}}\longmono&\mathcal{H}^1_\mathrm{dR}(\mathcal{A}_F/\mathcal{C}_F) \overset\nabla\longrightarrow
\mathcal{H}^i_\mathrm{dR}(\mathcal{A}_F/\mathcal{C}_F)\otimes_{\mathcal{O}_{\mathcal{C}_F}}
\Omega^1_{\mathcal{C}_F/F}\overset{\eqref{Hodge}}\longepi \\
&\overset{\eqref{Hodge}}\longepi \left(\pi_{F*}\left(\Omega^1_{\mathcal{A}_F/\mathcal{C}_F}\right)\right)^\vee
\otimes_{\mathcal{O}_{\mathcal{C}_F}}
\Omega^1_{\mathcal{C}_F/F}\end{split}
\] in which the first and the last map come from the Hodge exact sequence 
\eqref{Hodge}. Therefore the Kodaira-Spencer map can also 
be seen as a map of $\mathcal{O}_{\mathcal{C}_F}$-modules, denoted again with the same symbol,  
\[\mathrm{KS}_{\mathcal{A}_F/\mathcal{C}_F}:
\pi_{F*}\left(\Omega^1_{\mathcal{A}_F/\mathcal{C}_F}\right)^{\otimes2}\longrightarrow \Omega^1_{\mathcal{C}_F/F}.\]

\subsection{Idempotents and line bundles} 
Let \[e=\frac{1}{2}\left(\mathbf{1}\otimes1+\frac{1}{p_0}\mathbf{j}\otimes\sqrt{p_0}\right)\in
\mathcal{R}_M=\mathcal{R}_\mathrm{max}\otimes_{\Z}\mathcal{O}_M[1/(2p_0)]\] 
be the idempotent 
in \cite[(1.10)]{Mori}, \cite[\S2.1]{Brooks}, where $\mathcal{O}_M$ is the ring of 
integers of $M$. We have 
an isomorphism $\iota_M:B\otimes_\Q{M}\simeq\M_2(M)$.

Suppose we have an embedding
 $M\hookrightarrow F$, allowing us to identify $M$ with a subfield of $F$; in the cases we are interested in, either $F\subseteq\bar{\Q}$ (and then we require that $F$ contains $M$),  
 or $F=\C$ (and then we view $M\hookrightarrow\C$ via the fixed embedding $\bar{\Q}\hookrightarrow\C$) or $F\subseteq\bar{\Q}_p$ (and then we require that $F$ contains the image of $M$ via the fixed embedding $\bar{\Q}\hookrightarrow\bar{\Q}_p$). 

Since $M$ is contained in $F$, we have an action of $\mathcal{R}_M$ 
on the sheaves $\pi_{F*}\left(\Omega_{\mathcal{A}_F/\mathcal{C}_F}^1\right)$ 
and $\mathcal{H}^1_\mathrm{dR}(\mathcal{A}_F/\mathcal{C}_F)$, 
and we may 
therefore define the invertible sheaf of 
$\mathcal{O}_{\mathcal{C}_F}$-modules
\begin{equation}\label{def omega}
\underline{\omega}_F=e\cdot \pi_{F*}\left(\Omega_{\mathcal{A}_F/\mathcal{C}_F}^1\right)\end{equation} and the sheaf of 
$\mathcal{O}_{\mathcal{C}_F}$-modules
\begin{equation}\label{def L} 
\mathcal{L}_F=e\cdot \mathcal{H}^1_\mathrm{dR}(\mathcal{A}_F/\mathcal{C}_F).\end{equation} Using that $e$ is fixed by the Rosati involution, 
the Hodge exact sequence \eqref{Hodge} becomes 
\[0\longrightarrow \underline{\omega}_F\longrightarrow\mathcal{L}_{F}\longrightarrow
\underline{\omega}_F^{-1}\longrightarrow0\]
(see \cite[\S 2.6]{Brooks} for details). 
For any integer $n\geq 1$, define 
\begin{equation}\label{def L_n}\mathcal{L}_{F,n}=\Sym^n(\mathcal{L}_{F}).\end{equation}
The Gauss-Manin connection is compatible with the quaternionic 
action (\cite[Proposition 2.2]{Mori}). Therefore, restricting to $\mathcal{L}_{F,1}$ and using the Leibniz rule 
(see for example \cite[\S3.2]{Brooks}), 
the Gauss-Manin connection defines a connection
\[
\nabla_n:\mathcal{L}_{F,n}
\longrightarrow\mathcal{L}_{F,n}\otimes
\Omega^1_{\mathcal{C}_F/F}.\]
By \cite[Theorem 2.5]{Mori}, restricting the Kodaira-Spencer map to 
$\underline{\omega}_F^{\otimes{2}}$ 
gives an isomorphism 
\[\mathrm{KS}_F:\underline{\omega}_F^{\otimes{2}}\simeq\Omega^1_{\mathcal{C}_F/F}.\]
We may then 
define a map $\tilde\nabla_n:\mathcal{L}_{F,n}\rightarrow\mathcal{L}_{F,n+2}$ 
by the composition 
\begin{equation}
\label{nabla n}
\xymatrix{
\tilde{\nabla}_n:\mathcal{L}_{F,n}\ar[r]^-{\nabla_n}
&\mathcal{L}_{F,n}\otimes_{\mathcal{O}_{\mathcal{C}_L}}
\Omega^1_{\mathcal{C}_F/F}\ar[r]^-{\mathrm{id}\otimes{\mathrm{KS}_F^{-1}}}&\mathcal{L}_{F,n}\otimes_{\mathcal{O}_{\mathcal{C}_L}}\underline{\omega}_F^{\otimes{2}}\ar[r]&\mathcal{L}_{F,n}\otimes_{\mathcal{O}_{\mathcal{C}_L}}\mathcal{L}_{F,2}
\ar[r]&\mathcal{L}_{F,n+2}}
\end{equation}
where the last map is the product map in the symmetric algebras. 

\subsection{Algebraic modular forms} \label{modular forms}
For any $F$-algebra $R$, we define the $R$-algebra 
\[S_k^\mathrm{alg}\left(V_1(N^+),R)\right)=H^0(\mathcal{C}_R,\underline{\omega}_R^{\otimes{k}})\] 
of \emph{algebraic modular forms} of weight $k$ and level $V_1(N^+)$ over $R$. 
One can show (\cite[\S3.1]{Brooks}) that the $R$-algebra 
$S_k^\mathrm{alg}(V_1(N^+),R)$ can be alternatively described in modular terms. 
Let $R'$ be an $R$ algebra. A \emph{test triple} over $R'$ is a triplet $(A',t',\omega')$ consisting of a false elliptic curve $A'/R'$, 
a $V_1(N^+)$-level structure $t'$ and a global section $\omega'$ of $\underline{\omega}_{A'/R'}$. An isomorphism of test triples $(A',t',\omega')$ and $(A'',t'',\omega'')$ 
is an isomorphism of false elliptic curves $\phi:A'\rightarrow A''$ such that 
$\phi(t')=t''$ and $\phi^*(\omega'')=\omega'$. A \emph{test pair} over $R'$ 
is a pair $(A',t')$ obtained from a test triplet by forgetting the datum of the global section. 
Then one can identify global sections of 
$\underline{\omega}_R^{\otimes{k}}$ with:
\begin{enumerate} 
\item A rule $F$ which assigns, to each $R$-algebra $R'$ and each 
isomorphism class of test triplets $(A',t',\omega')$ 
over $R'$, an element $F(A',t',\omega')\in R'$,   
subject to the \emph{base change axiom} (for all maps of $R$-algebras $\phi:R'\rightarrow R''$, 
we have  
$F(A',t',\phi^*(\omega''))=F(A'',\phi(t'),\omega'')$, where $A'$ is the base change of $A''$ 
via $\phi$)  
and \emph{the weight $k$ condition} ($F(A',t',\lambda\omega')=\lambda^{-k}F(A',t',\omega')$ for any 
$\lambda\in(R')^\times$) (\cite[Definition 3.2]{Brooks}).
\item A rule $F$ which assigns to each $R$-algebra $R'$ and 
each isomorphism class of test pairs $(A',t')$ over $R'$,  
a translation invariant section $F(A',t')\in \underline{\omega}^{\otimes{k}}_{A'/R'}$ subject to the \emph{base change axiom} 
(for all maps of $R$-algebras $\phi:R'\rightarrow R''$, 
we have 
$F(A',t')=\phi^*(F(A'',\phi(t'))$, 
where $A'$ is the base change of $A''$ via $\phi$)  
(\cite[Definition 3.3]{Brooks})).\end{enumerate}  
Let us make the relations between these definition more explicit 
(\cite[page 4193]{Brooks}).
Given a global section $f\in H^0(\mathcal{C}_R,\underline{\omega}^{\otimes{k}}_R)$, 
we get a function as in (2) above associating to each test pair $(A',t')$ over $R'$ the point $x_{(A',t')}\in\mathcal{C}_R(R')$, 
and taking the value of 
$f$ at $x_{(A',t')}$; 
if $F$ is as in (2), we get a function $G$ on test 
triples $(A',t',\omega')$ over $R'$ as in (1) by the formula $F(A',t')=G(A',t',\omega')\omega^{\otimes{k}}$ where $\omega\in\underline{\omega}_{A'/R'}$ 
is the choice of any translation invariant global section. 

\section{Special values of $L$-series} 
In this section we review the work of Brooks \cite{Brooks} expressing 
special values of certain $L$-functions of modular forms 
in terms of CM-values of the Shimura-Maass operator applied to the modular form in question. 

\subsection{The real analytic Shimura-Maass operator} 
We denote $(X,\mathcal{O}_X)\leadsto (X^\mathrm{an},\mathcal{O}_X^\mathrm{an})$ 
the analytification functor which takes a scheme of finite type over $\C$ 
to its associated complex analytic space (\cite[\S2]{SerreGAGA}). 
For each sheaf $\mathcal{F}$ of $\mathcal{O}_X$-modules on $X$, we also denote $\mathcal{F}^\mathrm{an}$ 
the analytification of $\mathcal{F}$, and for each morphism 
$\varphi:\mathcal{F}\rightarrow\mathcal{G}$ of $\mathcal{O}_X$-modules, we let $\varphi^\mathrm{an}:\mathcal{F}^\mathrm{an}\rightarrow\mathcal{G}^\mathrm{an}$ the corresponding morphism of analytic sheaves 
(\cite[\S3]{SerreGAGA}).
If $(X,\mathcal{O}_X)$ is an analytic space, we denote $\mathcal{O}_X^\mathrm{r\text{-}an}$ the ring of real analytic functions on $X$; 
this is a sheaf of $\mathcal{O}_X$-modules, and for any sheaf $\mathcal{F}$ of $\mathcal{O}_X$-modules, we let $\mathcal{F}^\mathrm{r\text{-}an}=\mathcal{F}\otimes_{\mathcal{O}_X}\mathcal{O}_X^\mathrm{r\text{-}an}$; when $\mathcal{F}=\mathcal{F}^\mathrm{an}$, 
we simplify the notation by writing $\mathcal{F}^\mathrm{r\text{-}an}$ instead of $(\mathcal{F}^\mathrm{an})^\mathrm{r\text{-}an}$. 

Since $\mathcal{C}_\C$ is proper and smooth over $\C$, the analytification functor $\mathcal{F}\leadsto\mathcal{F}^\mathrm{an}$ 
induces 
an equivalence of categories between the category of coherent sheaves 
$\mathcal{C}_\C$ and the category of analytic coherent sheaves of 
$\mathcal{O}_{\mathcal{C}_\C}^\mathrm{an}$-modules (\cite[Th. 2, 3]{SerreGAGA}). 
Also, the analytic sheaf obtained from the 
sheaf of algebraic de Rham cohomology $\mathcal{H}^i_\mathrm{dR}(\mathcal{A}_\C/\mathcal{C}_\C)$ 
coincides with the derived functor $\R^1\pi_{\C*}^\mathrm{an}\left(\Omega^{1}_{\mathcal{A}_\C^\mathrm{an}/\mathcal{C}_\C^\mathrm{an}}\right)$ in the category of analytic sheaves over $\mathcal{C}_\C^\mathrm{an}$ (\cite[Theorem 1]{SerreGAGA}). 

Hodge theory gives a splitting 
\[\mathcal{H}^1_\mathrm{dR}(\mathcal{A}_\C/\mathcal{C}_\C)^\mathrm{r\text{-}an}\longrightarrow\left(\pi_{\C*}\left(\Omega^1_{\mathcal{A}_\C/\mathcal{C}_\C}\right)\right)^\mathrm{r\text{-}an}\]
of the corresponding Hodge exact sequence of real analytic sheaves 
obtained from \eqref{Hodge}. 
Since this splitting is the identity on the image 
of $\left(\pi_{\C*}\left(\Omega^1_{\mathcal{A}_\C/\mathcal{C}_\C}\right)\right)^\mathrm{r\text{-}an}$ in $\mathcal{H}^1_\mathrm{dR}(\mathcal{A}_\C/\mathcal{C}_\C)^\mathrm{r\text{-}an}$, it gives rise to a map
$\Psi_\infty:\mathcal{L}_{\C,1}^\mathrm{r\text{-}an}\rightarrow\underline\omega_{\C}^\mathrm{r\text{-}an}$ (\emph{cf.} \cite[Proposition 2.8]{Mori}). We may then consider the induced maps
$\Psi_{\infty,n}:\mathcal{L}_{\C,n}^\mathrm{r\text{-}an}\rightarrow
(\underline\omega_{\C}^{\otimes{n}})^\mathrm{r\text{-}an}$ for any integer $n\geq 1$.
Further, the map $\tilde\nabla_n$ gives rise to a map 
$\tilde\nabla_n^\mathrm{r\text{-}an}:\mathcal{L}_{\C,n}^\mathrm{r\text{-}an}
\rightarrow\mathcal{L}_{\C,n+2}^\mathrm{r\text{-}an}$ of real analytic sheaves. 
The composition 
\[
\xymatrix{\Theta_{\infty,n}: (\underline\omega_{\C}^{\otimes{n}})^\mathrm{r\text{-}an}\ar[r]&
\mathcal{L}_{\C,n}^\mathrm{r\text{-}an}\ar[r]^-{{\tilde\nabla}_n^\mathrm{r\text{-}an}}&\mathcal{L}_{\C,n+2}^\mathrm{r\text{-}an}\ar[r]^-{\Psi_{\infty,n}}&(\underline\omega_{\C}^{\otimes{n+2}})^\mathrm{r\text{-}an}
}
\]
is the real-analytic Shimura-Maas operator. 

The effect of $\Theta_{\infty,n}$ on modular forms is described in \cite[Proposition 3.4]{Brooks} and \cite[Proposition 2.9]{Mori}. Denote $\Gamma=\Gamma_1(N^+)$ the subgroup of 
$\mathcal{B}^\times\cap V_1(N^+)$ consisting of elements of norm equal to $1$. 
Fix an isomorphism $\mathcal{B}\otimes_\Q\R\simeq\M_2(\R)$ and denote
 $\Gamma_\infty$ the image of $\Gamma$ in $\GL_2(\R)$. 
Let $S_k(\Gamma_\infty)$ denote the $\C$-vector space of \emph{holomorphic 
modular forms} of weight $k$ and level $\Gamma_\infty$ 
consisting of those holomorphic 
functions on $\mathcal{H}_\infty$, the complex upper half plane, such that 
$f(\gamma(z))=j(\gamma,z)^kf(z)$ for all $\gamma\in\Gamma_\infty$; here
$\Gamma_\infty$ acts on $\mathcal{H}_\infty$ by fractional linear transformations 
via the map $\mathcal{B}\hookrightarrow\mathcal{B}\otimes_\Q\R\simeq\M_2(\R)$. 
We have (\emph{cf}. \cite[\S2.7]{Brooks})
\[S_k(\Gamma_\infty)\simeq H^0\left(\mathcal{C}_\C^\mathrm{an},(\underline{\omega}_\C^{\otimes{k}})^\mathrm{an}\right).\] 
Define the space $S_k^\mathrm{r\text{-}an}(\Gamma_\infty)$ of real analytic modular forms of level $\Gamma_\infty$ and 
weight $k$ to be the $\C$-vector space of real analytic functions $f:\mathcal{H}_\infty\rightarrow\C$ 
such that $f(\gamma(z))=j(\gamma,z)^kf(z)$ for all $\gamma\in\Gamma_\infty$. 
One then has 
\begin{equation}\label{real analytic modular forms}
S_k^{\mathrm{r\text{-}an}}(\Gamma_\infty)\simeq H^0\left(\mathcal{C}_\C^{\mathrm{an}},(\underline{\omega}_\C^{\otimes{k}})^\mathrm{r\text{-}an}\right).\end{equation} 
The operator $\Theta_{\infty,k}$ gives then rise to a map
$\delta_{\infty,k}:S_k(\Gamma_\infty)\rightarrow S_{k+2}^{\mathrm{r\text{-}an}}(\Gamma_\infty)$ 
and we have \[\delta_{\infty,k}\left(f(z)\right)=\frac{1}{2\pi i}\left(\frac{d}{dz}+\frac{k}{z+\bar{z}}\right)f(z).\]

\subsection{CM points and triples}\label{CM points} Fix an embedding $\bar\Q\hookrightarrow\C$. 
For any embedding $\varphi:K\hookrightarrow\mathcal{B}$ there 
exists a unique $\tau\in\mathcal{H}$ such that $\iota_\infty(\varphi(K^\times))(\tau)=\tau$. 
The additive map $K\hookrightarrow \C$ defined by 
$\alpha\mapsto j(\iota_\infty(\varphi(\alpha)),\tau)$ gives an embedding $K\hookrightarrow\C$; we 
say that $\varphi$ is \emph{normalized} if $\alpha\mapsto j(\iota_\infty(\varphi(\alpha)),\tau)$ is the identity (with respect to our fixed embedding $\bar{\Q}\hookrightarrow\C$). 

We say 
that $\tau\in\mathcal{H}$ is a \emph{CM point} if there exists 
an embedding
$\varphi:K\hookrightarrow\mathcal{B}$ 
which has $\tau$ as fixed point as above, and that a CM point 
$\tau$ 
is \emph{normalized} if $\varphi$ is normalized.    
Finally, we say that a CM point 
$\tau\in\mathcal{H}$ is a \emph{Heegner point}
if $\varphi(\mathcal{O}_K)\subseteq\mathcal{R}$ 
(\cite[\S2.4 and page 4188]{Brooks}). 

Fix a CM point $\tau$ corresponding to an embedding 
$\varphi:K\hookrightarrow\mathcal{B}$. 
Let $\mathfrak{a}$ be 
an integral ideal of $\mathcal{O}_K$, and define the 
$\mathcal{R}_\mathrm{max}$-ideal 
$\mathfrak{a}_\mathcal{B}=\mathcal{R}_\mathrm{max}\cdot\varphi(\mathfrak{a})$. This ideal is principal, generated by an element $\alpha=\alpha_\mathfrak{a}\in\mathcal{B}$. Right multiplication by $\alpha$ gives 
an isogeny $A_\tau\rightarrow A_{\alpha^{-1}\tau}$, whose 
kernel is $A_\tau[\mathfrak{a}]$. Let
$\Gamma_\mathrm{max}$ be the subgroup of $\mathcal{R}_\mathrm{max}^\times$ consisting of elements of norm equal to $1$. The image of $\alpha\tau$ by the canonical projection map $\rho_\mathrm{max}:\mathcal{H}\rightarrow\Gamma_\mathrm{max}\backslash\mathcal{H}$ 
does not depend on the choice of the representative $\alpha$, 
and therefore one may write $A_{\mathfrak{a}\star\tau}$ for the 
corresponding abelian surface. Shimura's reciprocity law states that $\rho_\mathrm{max}(\tau)$ 
is defined over the Hilbert class field $H$ of $K$, and that $\rho_\mathrm{max}(\tau)^{(\mathfrak{a}^{-1},H/K)}=\rho_\mathrm{max}(\mathfrak{a}\star\tau)$, where $(\mathfrak{a}^{-1},H/K)$ denotes the Artin symbol.

Fix a primitive $N^+$-root of unity $\zeta$. 
Fix a normalized Heegner point $\tau$, and fix a point
$P_\tau\in A_\tau[N^+]$ of exact order $N^+$ such that $e\cdot P=P$. Let $(A_\tau,P_\tau)$ denote the point on $\mathcal{C}(F)$
corresponding to the level structure 
$\mu_{N^+}\times\mu_{N^+}\simeq\Z/N^+\Z\times\Z/N^+\Z\rightarrow A_\tau[N^+]$ which takes $(1,0)\in\Z/N^+\Z\times\Z/N^+\Z$ to $P_\tau$. A \emph{CM triple} is an isomorphism 
class of triples $(A_\tau,P_\tau,\omega_\tau)$ with 
$(A_\tau,P_\tau)$ as above and $\omega_\tau\in e\cdot\Omega_{A_\tau/F}$ non vanishing. 

There is an action of $\mathrm{Cl}(\mathcal{O}_K)$ on 
the set of CM triples, given by 
\[\mathfrak{a}\star(A_\tau,P_\tau,\pi^*(\omega))=
(A_\tau/A_\tau[\mathfrak{a}],\pi(P_\tau),\omega)\] where 
$\pi:A_\tau\twoheadrightarrow A_\tau/A_\tau[\mathfrak{a}]$ is the canonical projection. 

\subsection{Special value formulas} 

Fix a CM triple $(A,P,\omega)=(A_\tau,P_\tau,\omega_\tau)$ 
with $\omega$ defined over $H$, the Hilbert class field of $K$; recall that $A$ is also defined over $H$, while in general $P$ is 
only defined over 
a field $L$ as in \S\ref{Brooks moduli problem}. 

The complex structure $J_\tau$ on $\M_2(\R)$ defines a differential 
form $\omega_\C=J_\tau^*(2\pi idz_1)$, and let $\Omega_\infty\in\C$ be define 
by $\omega=\Omega_\infty\cdot\omega_\C$; clearly, different choices of $\omega$ correspond to changing $\Omega$ by a multiple in $H$. 

We now let $f$ be a modular form of weight $k$, level 
$\Gamma_1(N^+)\cap\Gamma_0(N^-)$, 
and character $\varepsilon_f$,  
and let $f^\mathrm{JL}$ be the modular form on the 
Shimura curve $\mathcal{C}_\C$ associated with $f$ by the Jacquet-Langlands correspondence. We can normalise the choice of $f^\mathrm{JL}$ so that 
the ration $\langle f,f\rangle/\langle f^\mathrm{JL},f^\mathrm{JL}\rangle$ belongs to $K$ 
(\cite[\S2.7 and page 4232]{Brooks}). 

Let $\Sigma^{(2)}$ be the set of Hecke characters $\chi$ of $K$ of infinite type 
$(\ell_1,\ell_2)$ with $\ell_1\geq k$ and $\ell_2\leq 0$. We say that $\chi\in\Sigma^{(2)}$ 
is \emph{central critical} if $\ell_1+\ell_2=k$, so that the infinite type of $\chi$ 
is $(k+j,-j)$ for some integer $j\geq0$. Denote $\Sigma_\mathrm{cc}^{(s)}$ 
the subset of $\Sigma^{(2)}$ consisting of central critical characters. 

For each positive integer $j$, let $\delta_{\infty,k}^j:S_k^{\mathrm{r\text{-}an}}(\Gamma_\infty)\rightarrow D_{k+2j}^\mathrm{r\text{-}an}(\Gamma_\infty) $ denote the 
$j$-th iterate of the Shimura-Mass operator defined by 
\[\delta_{\infty,k}^j=\delta_{\infty,k+2(j-1)}\circ\dots\circ\delta_{\infty,k+2}\circ\delta_{\infty,k}.\]

For any Hecke character, one may consider the $L$-function 
$L(f,\chi^{-1},s)$, and for $\chi\in\Sigma^{(2)}$ central critical 
define the algebraic part $L_\mathrm{alg}(f,\chi^{-1})$
of its special value at $s=0$ as in \cite[Proposition 8.7]{Brooks}.
By \cite[Proposition 8.7]{Brooks}, if $\chi\in\Sigma^{(2)}_\mathrm{cc}$ then 
$L_\mathrm{alg}(f,\chi^{-1})\in\bar{\Q}$, and we have 
\[L_\mathrm{alg}(f,\chi^{-1})=\left(\sum_{\mathfrak{a}\in\mathrm{Cl}(\mathcal{O}_K)}
\chi_j^{-1}(\mathfrak{a})\cdot\delta_{\infty,k}^j(f^{\mathrm{JL}})(\mathfrak{a}\star(A,t,\omega))\right)^2 
\] where $\chi_j=\chi\cdot\norm^{-j}$ and $\norm$ is the norm map on ideals of $\mathcal{O}_K$. In this formula we view the real analytic modular 
$\delta_{\infty,k}(f^{\mathrm{JL}})$ 
as a function on test triplets, as in \cite[Proposition 8.5]{Brooks} 
via \eqref{real analytic modular forms} 
(see also the discussion in \cite[page 1094]{BDP} in the $\GL_2$ case).

For each ideal class 
$\mathfrak{a}$ in $\mathrm{Cl}(\mathcal{O}_K)$, let
$\alpha_\mathfrak{a}$ be the corresponding element in $\mathcal{B}$, 
as in \S\ref{CM points}. 
Then using the dictionary between 
real analytic forms as functions on $\mathcal{H}$ or functions on 
test triples, and recalling that $A=A_\tau$ for a normalized Heegner point $\tau$, we have 
\[L_\mathrm{alg}(f,\chi^{-1})=\left(\Omega_\infty^{k+2j}\cdot\sum_{\mathfrak{a}\in\mathrm{Cl}(\mathcal{O}_K)}
\chi_j^{-1}(\mathfrak{a})\cdot\delta_{\infty,k}^j(f^{\mathrm{JL}})(\alpha_\mathfrak{a}\cdot\tau)\right)^2.
\]  In this formula we view $\delta_{\infty,k}^j(f^{\mathrm{JL}})$ as a function on $\mathcal{H}$.

\section{The Shimura-Maass operator on the $p$-adic upper half plane}\label{Drinfelderia}
In this Section we define a $p$-adic Shimura-Maass operator in the context 
of Drinfel'd upper half plane. These results will be used in the next section 
to define a $p$-adic Shimura-Maass operator on Shimura curves, whose values at CM points will be compared with their complex analogue. 
As in the complex case, we will see that this operator plays a special role 
in defining $p$-adic $L$-functions. 

Let $\hatHp$ denote Drinfel'd $p$-adic upper half plane; this is a $\Z_p$-formal scheme, and we denote $\mathcal{H}_p$ its generic fiber, which is a $\Q_p$-rigid space (\cite[Chapitre I]{BC}).

\subsection{Drinfel'd Theorem}
Denote $D$ the unique division quaternion algebra over $\Q_p$, and let $\mathcal{O}_D$ be its 
maximal order. The field $\Q_{p^2}$ can be embedded in $D$, 
and in the following we will 
see it as a maximal commutative subfield of $D$ without 
explicitly mentioning it. Let $\sigma$ denote 
the absolute Frobenius automorphism of $\Gal(\Q_p^\unr/\Q_p)$. If fix an element $\Pi\in\mathcal{O}_D$ such that $\Pi^2=p$
and $\Pi x=\sigma(x)\Pi$ for $x\in\Q_{p^2}$, then $D=\Q_{p^2}[\Pi]$. 
We will denote $x\mapsto\bar{x}$ the restriction of $\sigma$ to $\Gal(\Q_{p^2}/\Q_p)$. 

For any $\Z_p$-algebra $B$, a \emph{formal $\mathcal{O}_D$-module over $B$} 
is a commutative $2$-dimensional formal group $G$ over $B$ equipped with an embedding 
$\iota_G:\mathcal{O}_D\hookrightarrow\End(G)$. A formal $\mathcal{O}_D$-module is said to 
be \emph{special} if for each geometric point $P$ of $\Spec(B/pB)$, 
the representation of $\mathcal{O}_D/\Pi\mathcal{O}_D$ on the tangent space $\Lie(G_P)$ 
of $G_P=G\times k_P$ is the 
sum of two distinct characters of $\mathcal{O}_D/\Pi\mathcal{O}_D$, where $k_P$ is the residue field of $P$; 
see \cite[Definition 1]{Tei-univ} for more details on this definition. By an \emph{$\SFD$-module over $B$}, 
we mean a special formal $\mathcal{O}_D$-module over $B$. If 
$G$ is a $\SFD$-module over $B\in\Nilp$, we denote 
 $\mathbb{M}(G)$ the (covariant) Cartier-Dieudonn\'e module of $G$ (\cite[Chapitre II, \S 1]{BC}); we also 
 denote $\mathbb{F}_G$ and $\mathbb{V}_G$ (or simply $\mathbb{F}$ and 
 $\mathbb{V}$ when there is no confusion)
the Frobenius and Verschiebung endomorphisms of $\mathbb{M}(G)$. If $B$ is $\Z_{p^2}$-algebra, and $G$ a formal $\mathcal{O}_D$-module, 
then we may define \[\Lie^0(G)=\{m\in\Lie(G):\iota_G(a)=am, a\in\Z_{p^2}\},\] 
\[\Lie^1(G)=\{m\in\Lie(G):\iota_G(a)=\bar{a}m, a\in\Z_{p^2}\}\] and, since $G$ is special, both $\Lie^0(G)$ and $\Lie^1(G)$ are free $B$-modules of rank $1$, 
(recall that $\bar{x}=\sigma(x)$, so $x\mapsto\bar{x}$ is the non-trivial automorphism of 
$\Gal(\Q_{p^2}/\Q_p)$).
Moreover, 
$\mathbb{M}(G)$ is also equipped with a graduation 
$\mathbb{M}(G)=\mathbb{M}^0(G)\oplus\mathbb{M}^1(G)$ where 
\[\mathbb{M}^0(G)=\{m\in \mathbb{M}(G):\iota_G(a)=am, a\in\Z_{p^2}\},\] 
\[\mathbb{M}^1(G)=\{m\in \mathbb{M}(G):\iota_G(a)=\bar{a}m, a\in\Z_{p^2}\}.\]

Fix a $\SFD$-module 
$\Phi=G\times G$ over $\bar{\mathbb{F}}_p$,  
where $G$ is the reduction modulo $p$ of a Lubin-Tate formal group 
$\hat{\mathcal{E}}$ of height $2$ over $\hat{\Z}_p^\unr$, the completion of the 
valuation ring of the maximal unramifed extension $\Z_p^\unr$ of $\Z_p$; so  
$\hat{\mathcal{E}}$ is the formal group of a 
supersingular elliptic curve $\mathcal{E}$ over $\Z^\unr_p$ (see \cite[Definition 9 and Remark 27]{Tei-univ}).  The Dieudonn\'e module $\mathbb{M}(\Phi)$ of $\Phi$
is the $\hat{\Z}_p^\unr[\mathbb{F},\mathbb{V}]$-module with $\mathbb{V}$-basis $g^0$ and $g^1$, 
satisfying the relations $\mathbb{F}(g^0)=\mathbb{V}(g^0)$ and 
$\mathbb{F}(g^1)=\mathbb{V}(g^1)$. The quaternionic order $\mathcal{O}_D$ acts via the rules $\Pi(g^0)=\mathbb{V}(g^1)$, $\Pi(g^1)=\mathbb{V}(g^0)$ and 
$a(g^0)=ag^0$, $a(g^1)=\bar{a}g^1$ for $a\in\Z_{p^2}$ (viewed inside $\mathcal{O}_D$ by the fixed embedding $\Q_{p^2}\hookrightarrow D$), where $a\mapsto\bar{a}$ is the non-trivial automorphism of $\Gal(\Q_{p^2}/\Q_p)$. 
By \cite[Corollary 30]{Tei-univ}, $\eta^0(\Phi)$ is generated over $\Z_p$ by $[g^0,0]$ and 
$[\mathbb{V}(g^1),0]$, and  
$\eta^1(\Phi)$ is generated over $\Z_p$ by $[g^1,0]$ and 
$[\mathbb{V}(g^0),0]$.

Let $\Nilp$ denote the category of $\Z_p$-algebras in 
which $p$ is nilpotent.
Denote $\mathrm{SFD}$ 
the functor on $\mathrm{Nilp}$ which associates 
to each $B\in\mathrm{Nilp}$ the 
set $\SFD(B)$ of isomorphism classes of triples $(\psi,G,\rho)$ where 
\begin{enumerate}
\item $\psi:\bar{\mathbb{F}}_p\rightarrow B/pB$ is an homomorphism, 
\item $G$ is a $\mathrm{SFD}$-module over $B$ of height $4$, 
\item $\rho:\psi_*\Phi\rightarrow G_{B/pB}= G\otimes_B{B/pB}$ is a quasi-isogeny of height $0$, called \emph{rigidification}. 
\end{enumerate}
See \cite[page 663]{Tei-univ} or \cite[Chapitre II (8.3)]{BC} for more details on the definition of the functor $\SFD$. 

Drinfel'd shows in \cite{Dr} that the 
functor $\SFD$ is represented by the $\Z_p$-formal scheme 
\[\hatHp^\unr=\hatHp\hat\otimes_{\Z_p}\hat{\Z}_p^\unr\]
(see \cite[Theorem 28]{Tei-univ}, \cite[Chapitre II (8.4)]{BC}). Note that $\hatHp^\unr$, considered as  
$\hat{\Z}_p^\unr$-formal scheme, represents the restriction $\overline{\SFD}$ 
of $\SFD$ to the category $\overline{\Nilp}$ of $\hat{\Z}_p^\unr$-algebras in which $p$ is nilpotent (\emph{cf.} \cite[Chapitre II, \S 8]{BC}). Unless otherwise stated, we will see 
$\hatHp^\unr$ as a $\hat{\Z}_p^\unr$-formal scheme. 

For later use, we review some of the steps involved in the proof of Drinfel'd Theorem. 
The crucial step is the interpretation of 
the $\Z_p$-formal scheme $\hatHp$ as the solution of a moduli problem. 
For $B\in \Nilp$, a \emph{compatible data on $S=\Spf(B)$} 
consists of a quadruplet $(\eta,T,u,\rho)$ where 
\begin{enumerate}
\item $\eta=\eta^0\oplus\eta^1$ is a sheaf of flat $\Z/2\Z$-graded $\Z_p[\Pi]$-modules on $S$, 
\item $T=T^0\oplus T^1$ is a $\Z/2\Z$-graded sheaf of $\mathcal{O}_S[\Pi]$-modules 
with $T^i$ invertible, 
\item $u:\eta\rightarrow T$ is a homogeneous degree zero map 
such that $u\otimes 1:\eta\otimes_{\Z_p}\mathcal{O}_S\rightarrow T$ is surjective, 
\item $\rho:(\Q_p^2)_S\rightarrow \eta_0\otimes_{\Z_p}\Q_p$ is a $\Q_p$-linear isomorphism, 
\end{enumerate}
which satisfy natural compatibilities, denoted $\mathbf{(C1)}$, $\mathbf{(C2)}$, $\mathbf{(C3)}$ in \cite[page 652]{Tei-univ}, to which we refer for details. 
The first step in Drinfel'd work is to show that the $\Z_p$-formal sheme
$\hat{\mathcal{H}}_p$ represents the functor which associates to each $B\in\Nilp$ the 
set of admissible quadruplets over $B$. 
To each compatible data $\mathcal{D}=(\eta,T,u,\rho)$ on $S$ one associates 
a $S$-valued point $\Psi:S\rightarrow\hatHp$ 
of $\hat{\mathcal{H}}_p$, as explained in \cite[pages 652-655]{Tei-univ}. 
The second step to prove the representability of $\overline{\SFD}$ is to 
associate with any $B\in\overline{\Nilp}$ and 
$X=(\psi,G,\rho)\in\SFD(B)$ 
a quadruplet $(\eta_X,T_X,u_X,\rho_X)$ 
which corresponds to an $S=\Spf(B)$-valued point on 
$\hatHp\hat\otimes_{\Z_p}\hat{\Z}_p^\unr$. 
If $X=(\psi,G,\rho)\in\SFD(B)$ is given as above, the quadruplet $(\eta_X,T_X,u_X,\rho_X$) can be explicitly constructed as follows:
\begin{itemize}
\item $T_X=T(G)=\mathbb{M}(G)/\mathbb{V}\mathbb{M}(G)$ the tangent space to $G$ 
at the origin, equipped with its graduation defined previously;  
\item Define $\mathbb{N}(G)=\mathbb{M}(G)\times\mathbb{M}(G)/\sim$ where 
$(\mathbb{V}(x),0)\sim(0,\Pi(x))$; we denote $[x,y]$ the class in $\mathbb{N}$ 
represented by the pair $(x,y)$. Let $\lambda:\mathbb{N}(G)\rightarrow\mathbb{M}(G)$ be the map $\lambda([x,y])=\Pi(x)-\mathbb{V}(y)$. There is a map $\mathbb{N}(G)\rightarrow \mathbb{M}(G)/\mathbb{V}\mathbb{M}(G)$ induced by the projection onto the first component. Further, 
one easily shows that there exists a unique map $L:\mathbb{M}(G)\rightarrow\mathbb{N}(G)$ satisfying the relation $\lambda\circ L=\mathbb{F}$. Let 
$\phi:\mathbb{N}(G)\rightarrow\mathbb{N}(G)$ be defined by $\phi([x,y])=L(x)+[y,0]$. Then $\eta_X=\eta(G)=\mathbb{N}^{\phi=\mathrm{Id}}$. The graduation 
of $\mathbb{M}(G)$ defines a graduation $\eta(G)=\eta^0(G)\oplus\eta^1(G)$. 
\item $u_X=u(G)$ is induced by the projection map $\mathbb{N}(G)\rightarrow \mathbb{M}(G)/\mathbb{V}\mathbb{M}(G)$.
\item Fix an isomorphism $\eta^0(\Phi)\simeq\Z_p\oplus\Z_p$. The quasi-isogeny $\rho$ induces a map 
\[\rho_X=\rho(G):\Q_p\oplus\Q_p\simeq\eta^0(\psi_*(\Phi))\otimes_{\Z_p}\Q_p\rightarrow\eta^0(\Phi)\otimes_{\Z_p}\Q_p.\]  
\end{itemize}

We finally discuss rigid analytic parameters (\cite{Tei-univ}). With an abuse of notation, let $\SFD$ be the functor from the category
$\proNilp$ of projective limits of objects in $\Nilp$ associated with $\SFD$. 
In \cite[Def. 10]{Tei-univ}, Teitelbaum introduces 
a function 
\begin{equation}\label{rigid analytic parameter} 
z_0:\SFD(\hat{\Z}_p^\unr)\longrightarrow\mathcal{H}_p(\hat{\Q}_p^\unr)\end{equation}
such that the map $X=(\psi,G,\rho)\mapsto(z_0(X),\psi)$ gives a bijection
between $\SFD(\hat{\Z}_p^\unr)$ and 
$(\hatHp\hat\otimes_{\Z_p}\hat{\Z}_p^\unr)(\hat{\Z}_p^\unr)$, which we identify with 
the set $\mathcal{H}_p(\hat{\Q}_p^\unr)\times\Hom(\hat{\Z}_p^\unr,
\hat{\Z}_p^\unr)$.  
We call the map $X\mapsto z_0(X)$ a \emph{rigid analytic parameter on $\SFD$}. 
If we let $\overline{\proNilp}$ the category of projective
limits of objects in $\overline{\Nilp}$, and we still denote $\overline{\SFD}$ 
the restriction of $\SFD$ to $\overline{\proNilp}$, 
this implies that the map $X=(\psi,G,\rho)\mapsto z_0(X)$ 
gives a bijection between $\overline{\SFD}(\hat{\Z}_p^\unr)$ 
and $\Hp(\hat{\Q}_p^\unr)$. 
By \cite[Thm. 45]{Tei-univ},
for each $z\in\mathcal{H}_p(\hat{\Q}_p^\unr)$, there exists triple $X=(\psi,G,\rho)$ in $\overline{\SFD}(\hat{\Z}_p^\unr)$ such that $z_0(X)=z$. 

\subsection{Filtered $\phi$-modules} Let $F$ be an unramified field extension of $\Q_p$. 
For an integer $a$, a \emph{$\sigma^{a}$-isocrystal $\mathbb{E}$ over $F$} 
is a pair $\mathbb{E}=(V,\phi)$ consisting of a finite dimensional $F$-vector space $V$ with a $\sigma^a$-linear isomorphism $\phi$ 
(\emph{i.e.} the isomorphism $\phi:V\rightarrow V$ satisfies the relation 
$\phi(xv)=\sigma(x)^a\cdot v$ for $x\in F$ and $v\in V$; see \cite[Chapter VI, \S1]{Zink}).  
If $a=1$, $\sigma$-isocrystals are also called \emph{$F$-isocrystals} 
(here $F$ stands for Frobenius, do not confuse with our fixed $p$-adic field $F$) of 
\emph{$\phi$-modules}, in which case the $\sigma$-linear isomorphism $\phi$ is called \emph{Frobenius} (in the following we will use both terminologies of $\phi$-modules and $F$-isocrystals). 

A \emph{filtered $F$-isocrystal}, or a \emph{filtered $\phi$-module} is a $\phi$-module $(V,\phi)$ equipped 
with an exhaustive and separate filtration $F^\bullet V$.  

If $G$ is a $p$-divisible formal group over $\bar{\mathbb{F}}_p$, one can define 
its first crystalline cohomology cohomology group 
as in \cite{Groth}, \cite{BO}, \cite[D\'efinition 2.5.7]{BBM}, in terms of the \emph{crystalline Dieudonn\'e functor} 
(among many other references, see for example 
\cite{Illusie}, \cite{Chambert-Loir}, \cite{deJong} for self-contained expositions). In the following we will denote 
$H^1_\mathrm{cris}(G)$ the global sections of the crystalline Dieudonn\'e functor (defined 
as in \cite[Th\'eor\`eme 4.2.8.1]{BBM}) tensored 
over $\hat{\Z}_p^\unr$ with $\hat{\Q}_p^\unr$.  
By construction, $H^1_\mathrm{cris}(G)$ 
is then an $F$-isocrystal. Moreover, the canonical isomorphism between $H^1_\mathrm{cris}(G)$ 
and the first de Rham cohomology group $H^1_\mathrm{dR}(G)$ of $G$ equips $H^1_\mathrm{cris}(G)$ with a canonical filtration (arising from the Hodge filtration in the de Rham cohomology), making $H^1_\mathrm{dR}(G)$ a filtered $F$-isocrystal; see 
\cite{Oda}. 

Let $G$ be a $\SFD$-module over $\bar{\mathbb{F}}_p$. 
Then the $F$-isocrystal $H^1_\mathrm{cris}(G)$ is a four-dimensional $\hat{\Q}_p^\unr$-vector space, equipped 
with its $\sigma$-linear Frobenius $\phi_{\mathrm{cris}}(G)$. 
It 
is also equipped with a $D$-module structure 
$j_G:D\hookrightarrow\End_{\hat\Q_p^\unr}\left(H^1_\mathrm{cris}(G)\right)$ which commutes with $\phi_{\mathrm{cris}}(G)$, 
and a $\Q_p$-algebra 
embedding $i_G:\M_2(\Q_p)\hookrightarrow\End_{\hat\Q_p^\unr}\left(H^1_\mathrm{cris}(G)\right)$ induced by the isomorphism $\End_{\mathcal{O}_D}(G)\simeq\M_2(\Q_p)$, 
which commutes with the $D$-action. 
Define $\phi'_\mathrm{cris}(G)=j_G(\Pi)^{-1}\phi_\mathrm{cris}(G)$ and put 
\[V_\mathrm{cris}(G)=H^1_\mathrm{cris}(G)^{\phi'_\mathrm{cris}(G)=\mathrm{Id}}.\] 
Denote $\phi_{V_\mathrm{cris}(G)}=j_G(\Pi)_{|{V_\mathrm{cris}(G)}}$ the restriction of $j_G(\Pi)$ to ${V_\mathrm{cris}(G)}$. Moreover, 
denote
 \[(\eta^i(G)\otimes_{\Z_p}\Q_p)^\vee=\Hom_{\Q_p}(\eta^i(G)\otimes_{\Z_p}\Q_p,\Q_p)\] 
 the $\Q_p$-linear dual of 
$\eta^i(G)\otimes_{\Z_p}\Q_p$. 

The following lemma is crucial in what follows, and identifies $V_\mathrm{cris}(G)$ 
with $(\eta^i(G)\otimes_{\Z_p}\Q_p)^\vee$, from which one deduces a complete description 
of the filtered $F$-isocrystal $H^1_\mathrm{cris}(G)$. It appears in a slightly different version in the proof of \cite[Lemma 5.10]{IS}. Since we did not find an 
reference for this fact in the form we need it, we add a complete proof. 

\begin{lemma}\label{prop1}
There is a canonical isomorphism 
$V_\mathrm{cris}(G)\simeq\left(\eta(G)\otimes_{\Z_p}\Q_p\right)^\vee$ of $\Q_p$-vector spaces. 
Moreover, $H^1_\mathrm{cris}(G)=V_\mathrm{cris}(G)\otimes_{\Q_p}\hat\Q_p^\unr$, 
where the right hand side is equipped with the structure of $\hat{\Q}_p^\unr$-vector space 
given by $x\cdot (v\otimes\alpha)=v\otimes(\sigma(x)\alpha)$ for $v\in V_\mathrm{cris}(G)$, 
$x,\alpha\in\hat{\Q}_p^\unr$. Finally, 
under this isomorphism the Frobenius $\phi_\mathrm{cris}(G)$ 
corresponds to $\phi_{V_\mathrm{cris}(G)}\otimes\sigma$.
\end{lemma}

\begin{proof} The $F$-isocrystal $H^1_\mathrm{cris}(G)$ 
is canonically isomorphic to the 
contravariant Dieudonn\'e module of $G$ with $p$ inverted, 
and with $\hat{\Q}_p^\unr$-action twisted by the Frobenius 
automorphism 
$\sigma$ of $\hat{\Q}_p^\unr$, equipped with the canonical Frobenius
of the contravariant Dieudonn\`e module 
(see \cite[4.2.14]{BBM}). More precisely, 
denote $\mathbb{D}(G)=\Hom_{\hat{\Q}_p^\unr}(\mathbb{M}(G)[1/p],\hat{\Q}_p^\unr)$ the $\hat{\Q}_p^\unr$-linear dual  
of the covariant Dieudonn\'e module 
$\mathbb{M}(G)$ of $G$ with $p$ inverted, and let $\mathbb{D}(G)^\sigma=\mathbb{D}(G)\otimes_{\hat{\Q}_p^\unr,\sigma}\hat{\Q}_p^\unr$, 
where the tensor product is taken with respect to the Frobenius endomorphism 
$\sigma$ of $\hat{\Q}_p^\unr$. Then 
as $\hat{\Q}_p^\unr$-vector spaces, we have $H^1_\mathrm{cris}(G)
\simeq\mathbb{D}(G)^\sigma$. Under this isomorphism the Frobenius 
$\phi_\mathrm{cris}(G)$ is given by the map $\varphi\mapsto\sigma\circ\varphi\circ\mathbb{V}_{G}$ for $\varphi\in\mathbb{D}(G)$. 

Now, by \cite[Lemme (5.12)]{BC}, we have an isomorphism of 
$\sigma^{-1}$-isocrystals
\begin{equation}\label{BC lemma}
\left(\mathbb{M}^i(G)[1/p],\mathbb{V}_G\Pi^{-1}\right)\simeq\left(\eta^i(G)\otimes_{{\Z}_p}\hat{\Q}_p^\unr,\sigma^{-1}\right)\end{equation} 
for each index $i=0,1$ (where the action of $\sigma^{-1}$ on 
$\eta^i(G)\otimes_{{\Z}_p}\hat{\Q}_p^\unr$ is on the second factor only). 
We may therefore compute $V_{\mathrm{cris}}(G)$ in terms of 
the isocrystal $\left(\eta^i(G)\otimes_{{\Z}_p}\hat{\Q}_p^\unr,\sigma^{-1}\right)$. 

As above, define 
$\mathbb{D}^i(G)=\Hom_{\hat{\Q}_p^\unr}(\mathbb{M}^i(G)[1/p],\hat{\Q}_p^\unr)$ ($\hat{\Q}_p^\unr$-linear dual) 
and let $\mathbb{D}^i(G)^\sigma$ denote the base change 
$\mathbb{D}^i(G)\otimes_{\hat{\Q}_p^\unr,\sigma}\hat{\Q}_p^\unr$ via $\sigma$. Since $\mathbb{M}(G)=\mathbb{M}^0(G)\oplus\mathbb{M}^1(G)$, 
we have $\mathbb{D}(G)^\sigma=\mathbb{D}^0(G)^\sigma\oplus\mathbb{D}^1(G)^\sigma$, and 
we may write any element $\varphi\in\mathbb{D}(G)^\sigma$  
as a pair $(\varphi_0,\varphi_1)$ with 
$\varphi_i\in \mathbb{D}^i(G)^\sigma$, $i=0,1$.   
By definition, an element $\varphi=(\varphi_0,\varphi_1)\in\mathbb{D}(G)^\sigma$ belongs to $V_\mathrm{cris}(G)$ 
if and only if $\varphi_i(V_G\Pi^{-1}(m_i))$ is equal to $\sigma^{-1}(\varphi_i(m_i))$ for all $m_i\in\mathbb{M}^i(G)[1/p]$, and for all $i=0,1$. 
Using \eqref{BC lemma}, identify $\varphi_i$ with 
a $\hat{\Q}_p^\unr$-linear homomorphism $\varphi_i:\eta^i(G)\otimes_{\Z_p}\hat{\Q}_p^\unr\rightarrow\hat{\Q}_p^\unr$, 
denoted with a slight abuse of notation with the same symbol; then  
the above equation describing $V_\mathrm{cris}(G)$ becomes 
$\varphi_i(n\otimes \sigma^{-1}(x))=\sigma^{-1}\varphi_i(n\otimes x)$ for all
$n\in \eta^i(G)$ and all $x\in\hat{\Q}_p^\unr$, or equivalently, 
since $\varphi_i$ is $\hat{\Q}_p^\unr$-linear,  
$\varphi_i(n\otimes1)=\sigma^{-1}\varphi_i(n\otimes1)$ for all 
$n\in\eta^i(G)$, and we conclude that $\varphi_i(n\otimes1)\in\Q_p$ for all
$n\in\eta^i(G)$. So $\varphi_i$ is the $\hat{\Q}_p^\unr$-linear extension of a $\Q_p$-linear homomorphism $\eta^i(G)\otimes_{\Z_p}\Q_p\rightarrow\Q_p$.
Since $\eta(G)=\eta^0(G)\oplus\eta^1(G)$, 
we then conclude that
$V_\mathrm{cris}(G)\simeq\left(\eta(G)\otimes_{\Z_p}\Q_p\right)^\vee$ as $\Q_p$-vector spaces. If $n_1,\dots,n_4$ is a $\Q_p$-basis of $\eta(G)\otimes_{\Z_p}\Q_p$, then 
$dn_1,\dots,dn_4$ defined by $dn_i(n_j)=\delta_{i,j}$ (as usual, $\delta_{i,j}=1$ if $i=j$ and $0$ otherwise) is a basis of $(\eta(G)\otimes_{\Z_p}\Q_p)^\vee$ and, by $\hat{\Q}_p^\unr$-linear extension, also of $(\eta(G)\otimes_{\Z_p}\hat{\Q}_p^\unr)^\vee$.
If we now base change the $\hat{\Q}_p^\unr$-vector space  
$(\eta(G)\otimes_{\Z_p}\hat{\Q}_p^\unr)^\vee$ via $\sigma$, we see that 
$dn_1,\dots,dn_4$ is still a $\hat{\Q}_p^\unr$-basis, and we have 
$(x\cdot dn_i)(n_j)=\sigma(x)\delta_{i,j}$ for all $x\in\hat{\Q}_p^\unr$. 
Using the above description of $H^1_\mathrm{cris}(G)$ in terms of $\mathbb{D}(G)^\sigma$, and the description of $V_\mathrm{cris}(G)$ in terms of $\eta(G)$, we have an isomorphism of $\hat{\Q}_p^\unr$-vector spaces, 
\[H^1_\mathrm{cris}(G)\simeq \left(V_\mathrm{cris}(G)\otimes_{\Q_p}\hat{\Q}_p^\unr\right)^\sigma,\] 
where the upper index $\sigma$ on the right hand side means 
that the structure of $\hat{\Q}_p^\unr$-vector space is twisted by $\sigma$ as explained above. Moreover, the $\sigma^{-1}$-linear isomorphism
$V_G\Pi^{-1}$ of $\mathbb{M}(G)[1/p]$ corresponds to 
the $\sigma^{-1}$-linear isomorphism 
$\sigma^{-1}$ of $\eta(G)\otimes_{\Z_p}\hat{\Q}_p^\unr$ (acting on the second component only), and therefore the isomorphism 
$\varphi\mapsto\sigma\circ\varphi\circ V_G$ of $\mathbb{M}(G)[1/p]^\vee$ ($\hat{\Q}_p^\unr$-linear dual) corresponds to the 
isomorphism $\Pi\otimes\sigma$ of $\left((\eta(G)\otimes_{\Z_p}{\Q}_p)^\vee\otimes_{\Q_p}\hat{\Q}_p^\unr\right)^\sigma$ given by $dn_i\otimes x\mapsto (dn_i\circ\Pi)\otimes\sigma(x)$ 
where $(dn_i\circ\Pi)(n)=n_i(\Pi n)$, which corresponds to $\phi_{V_\mathrm{cris}(G)}\otimes\sigma$ by definition of $\phi_{V_\mathrm{cris}(G)}$. 
\end{proof}

\subsection{Filtered convergent $F$-isocrystals on $\hatHp^\unr$} To describe the relative de Rham cohomology of the $p$-adic upper half plane, we first need some preliminaries on the notion of filtered convergent $F$-isocrystals introduced 
in \cite{IS}. 

We first recall some preliminaries. Let $F\subseteq\Q_p^\unr$ be an unramified extension of $\Q_p$, with valuation ring $\mathcal{O}_F$. 
If $(X,\mathcal{O}_X)$ is a $p$-adic $\mathcal{O}_F$-formal scheme, 
we denote $(X^\rig,\mathcal{O}_X^\rig)$ the associated $F$-rigid analytic space (or its generic fiber), and 
if $\mathcal{F}$ is a sheaf of $\mathcal{O}_X$-modules, 
we denote $\mathcal{F}^\rig$ its associated sheaf of $\mathcal{O}_X^\rig$-modules
(\cite[\S 7.4]{Bosch}, \cite[\S 1]{BoschCoherent}).
We say that $X$ is a 
\emph{$p$-adic $\mathcal{O}_F$-formal scheme} 
if $X$ is a $\mathcal{O}_F$-formal scheme which is locally of finite type. 
We will always assume in the following that $X$ is analytically smooth, 
so that $X^\rig$ is smooth. 

An \emph{enlargement} of $X$ is a pair $(T,z_T)$ consisting of a flat $p$-adic $\mathcal{O}_F$-formal scheme $T$ and a morphism of $\mathcal{O}_F$-formal schemes $z_T:T_0\rightarrow X$ where for each $\mathcal{O}_F$-formal scheme $T$ we denote $T_0$ the reduced closed subscheme of the closed subscheme $T_1$ of $T$ defined by the PD ideal $p\mathcal{O}_T$. 

A \emph{convergent isocrystal on $X$} (\emph{cf.} \cite[Definition 3.1]{IS}) is a rule $\mathcal{E}$ 
which assigns to each enlargement $(T,z_T)$ of $X$ a coherent $\mathcal{O}_T\otimes_{\mathcal{O}_F}F$-module $\mathcal{E}_T$ 
such that for any morphism $g:T'\rightarrow T$ of $\mathcal{O}_F$-formal schemes with $g_0:T_0'\rightarrow T_0$ satisfying $z_{T'}=z_T\circ g_0$ (where $g_0$ is induced from $g$), there 
is an isomorphism of $\mathcal{O}_{T'}\otimes_{\mathcal{O}_F}F$-modules 
$\theta_g:g^*(\mathcal{E}_T)\simeq\mathcal{E}_{T'}$ satisfying the cocycle condition.
The $\mathcal{O}_T\otimes_{\mathcal{O}_F}F$-module $\mathcal{E}_T$  
also seen as rigid analytic $\mathcal{O}_T^\rig$-module on the $F$-rigid analytic space 
$T^\rig$ (\cite[Remark (1.5)]{Ogus}); 
we distinguish the notation and write $\mathcal{E}_T^\rig$ to emphasise this viewpoint. 
 If $\mathcal{E}$ is a convergent isocrystal over $X$, for each enlargement $(T,z_T)$ which is analytically smooth over $\mathcal{O}_F$ we have an integrable connection 
\begin{equation}\label{rigid connection}
\nabla_T^\rig:\mathcal{E}_T^\rig\longrightarrow\mathcal{E}_T^\rig\otimes_{\mathcal{O}_T^\rig}\Omega^1_{X^\rig}.\end{equation}
 
A \emph{convergent $F$-isocrystal on $X$} (\emph{cf.} \cite[Definition 3.2]{IS}) is a convergent 
isocrystal $\mathcal{E}$ on $X$ equipped with an isomorphism of convergent isocrystals $\phi_\mathcal{E}:F^*\mathcal{E}\simeq\mathcal{E}$, where $F$
is the absolute Frobenius of $X_0$. 

A \emph{filtered convergent $F$-isocrystal on $X$} (\emph{cf.} \cite[Definition 3.3]{IS}) is a $F$-isocrystal 
$(\mathcal{E},\phi_\mathcal{E})$ 
such that $\mathcal{E}^\rig_X$ is equipped with an exhaustive and separated decreasing filtration $F^\bullet\mathcal{E}^\rig_X$
of coherent $\mathcal{O}_X^\rig$-submodules such that $\nabla_X^\rig(\mathcal{F}^i\mathcal{E}^\rig_X)$ is contained in $\mathcal{F}^{i-1}\mathcal{E}^\rig_X\otimes_{\mathcal{O}_X^\rig}\Omega^1_{X^\rig}$ for all $i$. 

We present two explicit examples of filtered convergent $F$-isocrystals. A third example will be 
discussed in \S\ref{universal SFD subsection}. 

\begin{example}\label{example1} The first example (\emph{cf.} \cite[Example 3.4(a)]{IS}) is the identity object of the additive tensor category of filtered isocrystals on $X$. This is the convergent isocrystal $\mathcal{E}(\mathcal{O}_X)$ on $X$ given by the rule $(T,z_T)\leadsto\mathcal{O}_T\otimes_{\mathcal{O}_F}F$ equipped with the canonical Frobenius
and the filtration given by $F^i\mathcal{O}_{X}^\rig=\mathcal{O}_X^\rig$ for $i\leq 0$ and $F^i\mathcal{O}_{X}^\rig=0$ for $i>0$ (in \emph{loc. cit.} this filtered convergent $F$-isocrystal is simply denoted $\mathcal{O}_X$).   \end{example}

\begin{example}\label{example2} Our second example (\emph{cf.} \cite[pages 345-346]{IS}) is the filtered convergent $F$-isocrystal $\mathcal{E}(V)$ attached to a representation $\rho:\GL_2\times\GL_2\rightarrow\GL(V)$, where $V$ is a finite dimensional $\Q_p$-rational representation, and $\GL_2$ is the algebraic group of invertible matrix over $\Q_p$. 
First, for a given such representation $\rho:\GL_2\times\GL_2\rightarrow\GL(V)$, let $\rho_1$ and $\rho_2$ denote the restrictions of $\rho$ to the first and second $\GL_2$-factor, respectively. 
As convergent isocrystal, $\mathcal{E}(V)=V\otimes_{\Q_p}\mathcal{E}(\mathcal{O}_{\hatHp^\unr})$. The Frobenius, making it a convergent $F$-isocrystal, is defined by $\phi_V\otimes\phi_{\mathcal{E}(\mathcal{O}_{\hatHp^\unr})}$, where $\phi_V=\rho_2\left(\smallmat 0p10\right)$ 
(so, the Frobenius depends on $\rho_2$ only ).  
To define the filtration, making it a filtered convergent $F$-isocrystal, 
we first recall some preliminaries. First, the filtration only depends on $\rho_1:\GL_2\rightarrow\GL(V)$, and therefore it is enough to define 
the filtration attached to a given representation $\rho:\GL_2\rightarrow\GL(V)$. For this, let $\mathcal{P}_n$ be the $\Q_p$-vector space of polynomials in one variable $X$ of degree at most $n$, 
equipped with a right action of $\GL_2$ by $P(X)\cdot A=(cX+d)^nP\left(\frac{aX+b}{cX+d}\right)$ for $A=\smallmat abcd$ and $P(X)\in\mathcal{P}_n$. Then put $V_n=\mathcal{P}_n^\vee$ ($\Q_p$-linear dual), equipped with the left action of $\GL_2$ by 
\[(A\cdot\varphi)(P(x))=\varphi(P(X)\cdot A).\]  
Recall that any representation $\rho:\GL_2\rightarrow\GL(V)$ can be 
written as a direct sum of a sum of representations of the form 
$V_1^{\otimes{m}}\otimes(V_1^\vee)^{\otimes{n}}$, where 
$n$, $m$ are non-negative integers.
To define the filtration on $\mathcal{E}(V)$ 
it is then enough to defined it for $V=V_1$. We have a map 
$\mathrm{ev}_{X-z}:V_1\otimes_{\Q_p}\mathcal{O}_{\mathcal{H}_p^\unr}\rightarrow\mathcal{O}_{\mathcal{H}_p^\unr}$ of sheaves 
on $\mathcal{H}_p^\unr$ defined by 
\[\varphi\otimes{f}\longmapsto[z\mapsto(\varphi(X)-z\varphi(1))f(z)].\] We put 
$F^0\mathcal{E}(V_1)=\mathcal{E}(V_1)$, $F^1\mathcal{E}(V_1)=\ker(\mathrm{ev}_{X-z})$ and $F^2\mathcal{E}(V_1)=0$. 
This defines the filtered convergent $F$-isocrystal $\mathcal{E}(V_1)$ attached to $V_1$, 
and therefore for any representation $\rho:\GL_2\times\GL_2\rightarrow\GL(V)$ we obtain a filtered convergent $F$-isocrystal $\mathcal{E}(V)$.\end{example}

\subsection{The filtered convergent $F$-isocrystal of the universal $\SFD$-module}
\label{universal SFD subsection} 
The third example of filtered convergent $F$-isocrystal 
arises from relative de Rham cohomology of the universal $\SFD$-module. Since it is more articulated 
that the previous ones, we prefer to keep it in a separate subsection. We follow \cite{Faltings}, \cite{IS}.

Let $(\lambda_\mathcal{G},\mathcal{G},\rho_\mathcal{G})$ be the universal triple, 
arising from the representability of the functor $\SFD$ by $\hatHp^\unr$; 
denote $\lambda:\mathcal{G}\rightarrow\hatHp^\unr$ be universal map.  
Let $\mathcal{G}^\vee$ be the Cartier dual of $\mathcal{G}$ (\cite[Chapitre III, \S5]{Fontaine}), which is equipped with a canonical map 
$\lambda^\vee:\mathcal{G}^\vee\rightarrow\hat{\mathcal{H}}_p^\unr$. 
In this setting one may define a convergent $F$-isocrystal
\[\mathcal{E}(\mathcal{G})=\R^1\lambda_*(\mathcal{O}_{\mathcal{G}/\hat{\Q}_p^\unr})\] 
interpolating 
crystalline cohomology sheaves 
(\cite[Theorems (3.1), (3.7)]{Ogus}): for each enlargement $(T,z_T)$, the value 
$\mathcal{E(G)}_T$ of $\mathcal{E(G)}$ 
at $T$ is defined to be the crystalline 
cohomology sheaf of coherent 
$\mathcal{O}_T\otimes_{\mathcal{O}_F}F$-modules 
$\mathbb{R}^qf_{T_1*}\mathcal{O}_{\mathcal{G}_{T_1/T}}\otimes_{\mathcal{O}_F}F$. 
The notation adopted here is standard, following \cite[\S 3]{Ogus}: $f_{T_1}:\mathcal{G}\times_{\hatHp^\unr}T_1\rightarrow T_1$ is the canonical projection where  
we use $z_T:T_1\rightarrow\hatHp^\unr$ to form the fiber product $\mathcal{G}\times_ZT_1$, and  
$\mathbb{R}^qf_{T_1*}\mathcal{O}_{\mathcal{G}_{T_1/T}}$ is the cristalline cohomology sheaf on the formal scheme $T$ 
(note that $T_1\hookrightarrow T$ is defined by the PD ideal $p\mathcal{O}_T$ and $f_{T_1}$ is smooth and proper); see \cite{BO0}, \cite{BO}.  
Since $T$ is noetherian, these are coherent sheaves of $\mathcal{O}_T\otimes_{\mathcal{O}_F}F$-modules, and therefore $\mathcal{E(G)}^\rig_T$ are coherent 
$\mathcal{O}_T^\rig$-modules. 

The coherent $\mathcal{O}_{\mathcal{H}_p^\unr}=\mathcal{O}_{\hatHp^\unr}^\rig$-module 
$\mathcal{E(G)}^\rig_{\hatHp^\unr}$ is 
equipped as in \eqref{rigid connection} with a connection $\nabla_{\hatHp^\unr}^\rig$. 
The coherent $\mathcal{O}_{\mathcal{H}_p^\unr}$-module 
 $\mathcal{E(G)}^\rig_{\hatHp^\unr}$ 
is canonically isomorphic to the 
relative rigid de Rham cohomology sheaf 
\[\mathcal{H}_\mathrm{dR}^{1,\rig}(\mathcal{G})=\mathcal{H}^1_\mathrm{dR}(\mathcal{G}^\rig/\mathcal{H}_p^\unr)=\R^1\lambda_*^\rig\left(\Omega^\bullet_{\mathcal{G}^\rig/\mathcal{H}_p^\unr}\right).\] Moreover, the connection $\nabla_{\hatHp^\unr}^\rig$
corresponds to the Gauss-Manin connection 
\[\nabla_{\mathcal{G}}^\rig:\mathcal{H}^{1,\rig}_\mathrm{dR}(\mathcal{G})\longrightarrow\Omega^1_{\mathcal{H}_p^\unr/\hat{\Q}_p^\unr}\otimes_{\mathcal{O}_{\mathcal{H}_p^\unr}}
\mathcal{H}^{1,\rig}_\mathrm{dR}(\mathcal{G})\]
whose construction
in this context follows  
\cite{KO}, and is the analogue of the construction we outlined in 
\S\ref{algebraic de Rham}; see \cite[Example 3.4(c)]{IS}, \cite[Theorem (3.10)]{Ogus}. 
The Hodge filtration on the de Rham cohomology 
$\mathcal{H}^{1,\rig}_\mathrm{dR}(\mathcal{G})$ makes then 
$\mathcal{E(G)}$ a filtered convergent $F$-isocrystal. 

The filtration on $\mathcal{E(G)}$ arising from the Hodge filtration on the de Rham cohomology can be described more explicitly. 
Denote
$\mathcal{H}^1_\mathrm{dR}(\mathcal{G}/\hatHp^\unr)$ the dual of the Lie algebra of the universal vectorial 
extension of $\mathcal{G}$, equipped with 
its structure of convergent $F$-isocrystal (\cite[Chapter IV, \S2]{Messing}, \cite[\S\S 1,9,11)]{MM}). 
By \cite[\S3.3]{BBM}, we have an isomorphism of convergent $F$-isocrystals 
\[\mathcal{E(G)}\simeq\mathcal{H}^1_\mathrm{dR}(\mathcal{G}/\hatHp^\unr).\] 
The Hodge-Tate filtration on $\mathcal{H}^{1,\rig}_\mathrm{dR}(\mathcal{G})\simeq
\mathcal{H}^1_\mathrm{dR}(\mathcal{G}/\hatHp^\unr)^\rig_{\mathcal{H}_p^\unr}$ 
can be described in explicit terms as follows. 
Denote $\underline\omega_{\mathcal{G}}=e_\mathcal{G}^*(\Omega^1_{\mathcal{G}/\hat{\mathcal{H}}_p^\unr})$,
where $e_\mathcal{G}:\hatHp^\unr\rightarrow\mathcal{G}$ is the zero-section, 
and let 
$\Liecal _{\mathcal{G}^\vee}$
be the Lie algebra of the Cartier dual $\mathcal{G}^\vee$ of $\mathcal{G}$.
Then $\underline\omega_{\mathcal{G}}$ and $\mathcal{L}_{\mathcal{G}^\vee}$ are locally free $\mathcal{O}_{\hatHp^\unr}$-modules, 
dual to each other (\cite[\S3.3]{BBM}).
We have the Hodge-Tate exact sequence 
of $\mathcal{O}_{\hat{\mathcal{H}}_p^\unr}$-modules 
\[0\longrightarrow \underline\omega_{\mathcal{G}} \longrightarrow
\mathcal{H}^{1 }_\mathrm{dR}(\mathcal{G}/\hatHp^\unr)\longrightarrow 
\Liecal _{\mathcal{G}^\vee}\longrightarrow0\]
(see \cite[Corollaire 3.3.5]{BBM}, and recall that $\SFD$-modules are 
in particular $p$-divisible group schemes, 
by \cite[Ch. II, \S7]{BC}). 
Since the functor $\mathcal{F}\leadsto\mathcal{F}^\rig$ is exact by \cite[Prop. 1.4]{BoschCoherent}, 
we obtain an exact sequence 
\begin{equation}\label{Hodge SFD}0\longrightarrow \underline\omega_{\mathcal{G}}^\rig \longrightarrow
\mathcal{H}^{1,\rig}_\mathrm{dR}(\mathcal{G})\longrightarrow 
\Liecal_{\mathcal{G}^\vee}^\rig\longrightarrow0.\end{equation}
The Hodge-Tate filtration \eqref{Hodge SFD} coincides with the Hodge filtration on the de Rham cohomology groups. 

One the one hand, the exact sequence \eqref{Hodge SFD} defines the filtration on the 
convergent $F$-isocrystal 
$\mathcal{E(G)}\simeq\mathcal{H}^1_\mathrm{dR}(\mathcal{G}/\hatHp^\unr)$. On the other hand, recall that 
$V_\mathrm{cris}(\Phi)\otimes_{\Q_p}\mathcal{E}(\mathcal{O}_{\hatHp^\unr})$ 
has a structure of filtered convergent $F$-isocrystal on $\hatHp^\unr$, with 
Frobenius given by 
$\phi_{V_\mathrm{cris}(\Phi)}\otimes\phi_{\hatHp^\unr}$, and filtration 
induced as described before by the restriction 
$i_{|\GL_2(\Q_p)}:\GL_2(\Q_p)\rightarrow\GL(V)$ of $i$ to $\GL_2(\Q_p)$ (Example \ref{example2}). 
A result of Faltings \cite[\S5]{Faltings} (see also the discussion in \cite[Lemma 5.10]{IS}) shows that, 
as filtered convergent $F$-isocrystals, we have
\begin{equation}\label{IS}
\mathcal{H}^1_\mathrm{dR}(\mathcal{G}/\hatHp^\unr)\simeq V_\mathrm{cris}(\Phi)\otimes_{\Q_p}\mathcal{E}(\mathcal{O}_{\hatHp^\unr}).
\end{equation}

The isomorphism of filtered convergent $F$-isocrystals \eqref{IS} 
can be reformulated as follows (\cite[Lemma 5.10]{IS}). Let 
$\rho:\GL_2\times\GL_2\rightarrow\GL(\M_2)$ be the representation defined by 
$\rho_1(A)(B)=AB$ and $\rho_2(A)B=B\bar{A}$ where if $A=\smallmat abcd$ 
then $\bar{A}=\smallmat d{-b}{-c}a$. Note that $\mathcal{E}(\M_2)$ is pure of weight $1$. The isomorphism \eqref{IS} can be rewritten in a more compact way 
as the isomorphism of filtered convergent $F$-isocrystals  
\begin{equation}\label{IS2}
\mathcal{H}^1_\mathrm{dR}(\mathcal{G}/\hatHp^\unr)\simeq\mathcal{E}(\M_2).\end{equation}
The Kodaira-Spencer map in this setting is the composition 
\[
\begin{split}\mathrm{KS}_{\mathcal{G}}^\rig:
\underline\omega_{\mathcal{G}}^\rig\overset{\eqref{Hodge SFD}}\longmono&\mathcal{H}^{1,\rig}_\mathrm{dR}(\mathcal{G}) 
\overset{\nabla_\mathcal{G}^\rig}\longrightarrow
\mathcal{H}^{1,\rig}_\mathrm{dR}(\mathcal{G})\otimes_{\mathcal{O}_{\mathcal{H}_p^\unr}}
\Omega^1_{\mathcal{H}_p^\unr/\hat{\Q}_p^\unr}
\overset{\eqref{Hodge SFD}}\longrightarrow\Liecal_{\mathcal{G}^\vee}^\rig
\otimes_{\mathcal{O}_{\mathcal{H}_p^\unr}}
\Omega^1_{\mathcal{H}_p^\unr/\hat{\Q}_p^\unr}\end{split}
\] in which the first and the last map come from the Hodge exact sequence 
\eqref{Hodge SFD}. 
%
%
Recalling the duality between $\underline\omega_\mathcal{G}$ and $\Liecal_{\mathcal{G}}$, we therefore obtain, as in the algebraic case, a symmetric map of $\mathcal{O}_{\mathcal{H}_p^\unr}$-modules, again denoted again with the same symbol, 
$\mathrm{KS}_{\mathcal{G}}^\rig:
\underline\omega_\mathcal{G}^\rig\otimes
\underline\omega_\mathcal{G^\vee}^\rig
\rightarrow \Omega^1_{\mathcal{H}_p^\unr/\hat{\Q}_p^\unr}$ (tensor product over ${\mathcal{O}_{\mathcal{H}_p^\unr}}$). 
By fixing a formal polarization $\iota_\mathcal{G}:\mathcal{G}\simeq\mathcal{G}^\vee$ of $\mathcal{G}$ (\cite[Chapitre III, Lemma 4.4]{BC}),  
we obtain isomorphism 
$\underline\omega_\mathcal{G}\simeq\omega_\mathcal{G^\vee}$ of $\mathcal{O}_{\mathcal{H}_p^\unr}$-modules, and  
the Kodaira-Spencer map takes the form 
\[\mathrm{KS}_{\mathcal{G}}^\rig:
(\underline\omega_\mathcal{G}^\rig)^{\otimes{2}}
\longrightarrow \Omega^1_{\mathcal{H}_p^\unr/\hat{\Q}_p^\unr}\]
where the tensor product is again over $\mathcal{O}_{\mathcal{H}_p^\unr}$. 

The Kodaira-Spencer map can also also be expressed in a different form, 
(\emph{cf.} \cite{Harris}, \cite{Mori} in the complex setting).  
Denote $\mathcal{H}_{1,\mathrm{dR}}(\mathcal{G}/\mathcal{H}_p^\unr)$ the universal vectorial extension of 
$\mathcal{G}$, which is equipped with a structure of convergent filtered $F$-isocrystal as before; see \cite[\S 5]{Faltings}. Put 
$\mathcal{H}_{1,\mathrm{dR}}^\rig(\mathcal{G})=\mathcal{H}_{1,\mathrm{dR}}(\mathcal{G}/\mathcal{H}_p^\unr)_{\hat{\mathcal{H}}_p^\unr}^\rig$.
By definition, the universal vectorial extensions of $\mathcal{G}$ and $\mathcal{G}^\vee$ are dual to each other, since each extension of 
one each gives rise to an extension of the other by duality. 
We therefore obtain a $\mathcal{O}_{\mathcal{H}_p^\rig}$-bilinear 
skew-symmetric map 
\[\mathcal{H}^{1,\rig}_{\mathrm{dR}}(\mathcal{G})\times\mathcal{H}_{1,\mathrm{dR}}^\rig(\mathcal{G})\longrightarrow\mathcal{O}_{\mathcal{H}_p^\unr}.\] 
The principal polarization $\iota_\mathcal{G}:\mathcal{G}\simeq\mathcal{G}^\vee$ identifies canonically 
$\mathcal{H}^{1,\rig}_{\mathrm{dR}}(\mathcal{G})$ 
and $\mathcal{H}_{1,\mathrm{dR}}^\rig(\mathcal{G})$, 
and we therefore obtain a pairing 
\[\langle,\rangle_\mathrm{dR}^\rig:\mathcal{H}^{1,\rig}_{\mathrm{dR}}(\mathcal{G})\times
\mathcal{H}^{1,\rig}_{\mathrm{dR}}(\mathcal{G})
\longrightarrow\mathcal{O}_{\mathcal{H}_p^\unr}\]
satisfying $\langle dx,y\rangle_\mathcal{G}^\rig=\langle x,d^\dagger y\rangle_\mathcal{G}^\rig$ for all $x,y$ sections in $\mathcal{H}^{1,\rig}_{\mathrm{dR}}(\mathcal{G})$ 
and all $d\in D$ (because $\iota_\mathcal{G}(dy)=d^\dagger\iota_\mathcal{G}(y)$)   
which we call \emph{rigid polarization pairing}. 
We may therefore construct  
a map \[\rho: \mathcal{H}^{1,\rig}_\mathrm{dR}(\mathcal{G})\longrightarrow\left(\mathcal{H}^{1,\rig}_\mathrm{dR}(\mathcal{G})\right)^\vee\longrightarrow (\underline\omega_\mathcal{G}^\rig)^\vee\]
where the first map takes a section $s$ to the map defined for a section 
$t$ by $t\mapsto \langle s,t\rangle_\mathrm{dR}^\rig$ 
and the second map is induced by duality from the inclusion $\underline\omega_\mathcal{G}^\rig\hookrightarrow\mathcal{H}^{1,\rig}_\mathrm{dR}(\mathcal{G})$. Fix now a section $s\in H^0(U,(\Omega^1_{\mathcal{H}_p^\unr/\hat{\Q}_p^\unr})^\vee)$ over some affinoid $U$. 
Then we may compose the maps to get 
\begin{equation}\label{KSdual}\begin{split}\rho_s:H^0\left(U,\underline\omega_\mathcal{G}^\rig\right)\longrightarrow &H^0\left(U,\mathcal{H}^{1,\rig}_\mathrm{dR}(\mathcal{G})\right)\overset{\nabla_\mathcal{G}^\rig} \longrightarrow
H^0\left(U,\mathcal{H}^{1,\rig}_\mathrm{dR}(\mathcal{G})
\otimes\Omega^1_{\mathcal{H}_p^\unr/\hat{\Q}_p^\unr}\right)
\overset{1\otimes s}\longrightarrow\\
&\overset{1\otimes s}\longrightarrow H^0\left(U,\mathcal{H}^{1,\rig}_\mathrm{dR}(\mathcal{G})\right)\overset\rho \longrightarrow
H^0\left(U,(\underline\omega_\mathcal{G}^\rig)^\vee\right).\end{split}
\end{equation}
The association $s\mapsto\rho_s$ defines then a map of sheaves 
\[(\mathrm{KS}_\mathcal{G}^\rig)^\vee:(\Omega^1_{\mathcal{H}_p^\unr/\hat{\Q}_p^\unr})^\vee\longrightarrow
\Hom_{\mathcal{O}_{\mathcal{H}_p^\unr}}\left(\underline\omega_\mathcal{G}^\rig,(\underline\omega_\mathcal{G}^\rig)^\vee\right).\] By construction, the dual of this map  
is the Kodaira-Spencer map, under the canonical identification between  
$\Hom(\underline\omega_\mathcal{G}^\rig,(\underline\omega_\mathcal{G}^\rig)^\vee)^\vee$ and $(\underline\omega_\mathcal{G}^\rig)^{\otimes2}$. 

\subsection{Universal rigid data} \label{section universal rigid data}

The aim of this subsection is to use the results of \cite{Tei-univ} to better 
describe the Hodge filtration \eqref{Hodge SFD}. For this, we need to recall the universal rigid data introduced 
in \cite{Tei-univ}. 

Let $V_0$ and $V_1$ be constant sheaves of one-dimensional $\Q_p$-vector 
spaces on the $\Q_p$-rigid analytic space $\mathcal{H}_p$ with basis $t_0$ and $t_1$ respectively. Define 
two invertible sheaves $T_0^\univ$ and $T_1^\univ$ on $\mathcal{H}_p$ 
by 
$T_i^\univ=\mathcal{O}_{\mathcal{H}_p}\otimes V_i$  
for $i=0,1$, where 
$\mathcal{O}_{{\mathcal{H}_p}}$ is the structural sheaf of rigid analytic functions on
$\mathcal{H}_p$. Define $T^\univ=T_0^\univ\oplus T_1^\univ$. 
For $i=0,1$, let $\eta_i^\univ$ be the constant sheaf of 
two-dimensional $\Q_p$-vector spaces on $\mathcal{H}_p$ with basis $e_{i,0}$ and $e_{i,1}$.  
One fixes 
\begin{equation}\label{universal eta}
\eta^\univ_i=\eta^i(\Phi)\otimes_{\Z_p}\Q_p\end{equation}
as in \cite[page 664]{Tei-univ}.
Define $u_0^\univ:\eta_0^\univ\rightarrow T_0^\univ$ by 
$u_0^\univ(e_{0,0})=zt_0$ and $u_0^\univ(e_{0,1})=t_0$, 
and $u_1^\univ:\eta_1^\univ\rightarrow T_1^\univ$ by 
$u_1^\univ(e_{1,0})=(p/z)t_1$ and $u_1^\univ(e_{1,1})=t_1$, 
where $z$ denotes the standard coordinate function on $\mathcal{H}_p$. 
Define $\eta^\univ=\eta_0^\univ\oplus\eta_1^\univ$ and similarly define 
$u^\univ=u_0^\univ\oplus u_1^\univ$.
We write $\rho^\univ:(\Q_p)_{\mathcal{H}_p}^2\simeq\eta_0^\univ$ for the 
isomorphism determined by the choice of the basis $\{e_{0,0}, e_{0,1}\}$. 
For $\gamma=\smallmat abcd\in\M_2(\Q_p)$ and $i=0,1$, define 
endomorphisms $\phi_i^{(\gamma)}$ in $\End_{\mathcal{O}_{\mathcal{H}_p}}(T_i^\univ)$ 
by 
\begin{equation}\label{moebius}
\phi_0^{(\gamma)}\left(f(z)\otimes t_0\right)=(cz+d)f(\gamma(z))\otimes t_0, 
\text{ and }
\phi_1^{(\gamma)}\left(f(z)\otimes t_1\right)=(a+b/z)f(\gamma(z))\otimes t_1
\end{equation}
for any $f\in\mathcal{O}_{\mathcal{H}_p}(U)$, and any affinoid $U\subseteq\mathcal{H}_p$. 
Define an action of $\SL_2(\Q_p)$ on $\eta_{i}^\univ$ for $i=0,1$ in such ta way that $u_i^\univ$ is equivariant with respect to these actions, namely,
for $\gamma=\smallmat abcd$ in $\GL_2(\Q_p)$, put $\gamma^*\binom{x_{0,0}}{x_{0,1}}=\smallmat abcd\binom{x_{0,0}}{x_{0,1}}$ and $\gamma^*\binom{x_{1,0}}{x_{1,1}}=\smallmat a{b/p}{pc}d\binom{x_{1,0}}{x_{1,1}}$. 
Let $\Z_p[\Pi]$ act on $T^\univ$ 
by $\Pi t_0=(p/z)t_1$ and $\Pi t_1=zt_0$. We let $\Z_p[\Pi]$ act on $\eta^\univ$ 
in such a way that $u^\univ$ commutes with this action.  We call the quadruplet  
\[{\mathcal{D}}^\univ=(\eta^\univ,T^\univ,u^\univ,\rho^\univ)\] the \emph{universal rigid data}. 

Passing to the associated normed sheaves (\cite[
Definition 6]{Tei-univ}), we obtain from 
${\mathcal{D}}^\univ$ a quadruplet $\hat{\mathcal{D}}^\univ=(\hat{\eta}^\univ,\hat{T}^\univ,\hat{u}^\univ,\hat{\rho}^\univ)$ 
on $\hatHp$, corresponding to a $\hatHp$-valued point, 
which is universal in the following sense: for each 
$B\in \Nilp$ and each $\Psi:S=\Spec(B)\rightarrow\hatHp$ corresponding 
to a quadruplet $(\eta,T,u,\rho)$, we have 
\begin{equation}\label{universal quadruplet} 
(\eta,T,u,\rho)=(\Psi^{-1}\hat\eta^\univ,\Psi^*\hat{T}^\univ,\Psi^{-1}\hat{u}^\univ,\Psi^{-1}\hat{\rho}^\univ).\end{equation} See \cite[Cor. 18 and Thm. 19]{Tei-univ} for more precise and complete statements. 
We call $\hat{\mathcal{D}}^\univ$ the \emph{universal formal data}, 
and we denote the quadruplet on the RHS of \eqref{universal quadruplet} by $\hat{\mathcal{D}}^\univ_\Psi$ to simplify the notation. 

The universal $\SFD$-module $\mathcal{G}$ over $\hatHp^\unr$ can be recovered
from a universal rigid data $\mathcal{D}^\univ$. 
Pulling back via the projection $\pi_{\hatHp}:\hatHp^\unr\rightarrow\hatHp$, we obtain a quadruplet 
\[\hat{\mathcal{D}}^\unr=(\hat\eta^\unr,\hat{T}^\unr,\hat{u}^\unr,\hat{\rho}^\unr)=(\pi^{-1}_{\hatHp}\hat\eta^\univ,\pi_{\hatHp}^*\hat{T}^\univ,\pi^{-1}_{\hatHp}\hat{u}^\univ,\pi^{-1}_{\hatHp}\hat\rho^\univ)\] on $\hatHp^\unr$. Comparing
\eqref{universal quadruplet} with the universal property satisfied by
$\mathcal{G}$, we see that the quadruplet $(\eta_\mathcal{G},T_\mathcal{G},u_\mathcal{G},\rho_\mathcal{G})$ 
associated to $\mathcal{G}$ coincides with the quadruplet 
$\hat{\mathcal{D}}^\unr$. In particular, the associated quadruplet 
$(\eta_\mathcal{G}^\rig,T_\mathcal{G}^\rig,u_\mathcal{G}^\rig,\rho_\mathcal{G}^\rig)$ on the rigid $\hat{\Q}_p^\unr$-rigid analytic 
space $\mathcal{G}^\rig$ is the quadruplet 
\[\mathcal{D}^\unr=(\eta^\unr,{T}^\unr,{u}^\unr,{\rho}^\unr)=
(\pi^{-1}_{\mathcal{H}_p^\unr}\eta^\univ,\pi_{\mathcal{H}_p^\unr}^*{T}^\univ,
\pi^{-1}_{\mathcal{H}_p^\unr}{u}^\univ,\pi^{-1}_{\mathcal{H}_p^\unr}\rho^\univ)\]
obtained from the quadruplet $\mathcal{D}^\univ$, where 
$\pi_{\mathcal{H}_p^\unr}:\mathcal{H}_p^\unr\rightarrow\mathcal{H}_p$ is the 
canonical projection. 

Let $(T^\unr)^\vee$ denote the $\mathcal{O}_{\mathcal{H}_p^\unr}$-dual of 
$T^\unr$, and, as above, denote $(\eta^\unr\otimes_{\Z_p}\Q_p)^\vee$ 
the $\Q_p$-linear dual of $\eta^\unr\otimes_{\Z_p}\Q_p$. 
From the surjective map $u^\unr:\eta^\unr\otimes_{\Z_p}\mathcal{O}_{\mathcal{H}_p^\unr}\twoheadrightarrow T^\unr$ induced by $u^\univ$ 
we obtain an injective map 
\[\tau:(T^\unr)^\vee\longmono (\eta^\unr\otimes_{\Z_p}\Q_p)^\vee\otimes_{\Q_p}\mathcal{O}_{\mathcal{H}_p^\unr} 
\]

\begin{proposition}\label{identifications}
We have canonical isomorphisms 
\[(T^\unr)^\vee\simeq\omega_\mathcal{G}^\rig\]
\[(\eta^\unr\otimes_{\Z_p}\Q_p)^\vee\otimes_{\Q_p}\mathcal{O}_{\mathcal{H}_p^\unr} 
\simeq \mathcal{H}^{1,\rig}_\mathrm{dR}(\mathcal{G}).\]
under which the map $\tau$ corresponds to the canonical map 
in \eqref{Hodge SFD}.
\end{proposition}

\begin{proof}
The first statement follows from the canonical isomorphism 
between $T_\mathcal{G}=\Liecal_{\mathcal{G}}$ 
and $T^\unr$, while the second 
follows from Proposition \ref{prop1} combined with \eqref{IS}. 
For the statement about $\tau$, note that for each $\SFD$-module $G$ 
over $\bar{\mathbb{F}}_p$, 
the map $u_G$ corresponds under the identification  
between $\eta(G)\otimes_{\Z_p}\hat{\Q}_p^\unr$ 
and $\mathbb{M}(G)\otimes_{\hat{\Z}_p^\unr}\hat{\Q}_p^\unr$ to 
the canonical projection $\mathbb{M}(G)/\mathbb{V}_G\mathbb{M}(G)\rightarrow T_G$, where $T_G$ is the tangent space of $G$ at the origin. 
\end{proof} 

\subsection{The action of the idempotent $e$}\label{idempotent section}
Fix an isomorphism $\Q_p(\sqrt{p_0})\simeq\Q_{p^2}$. By means 
of this isomorphism, and the fixed embedding $\Q_{p^2}\hookrightarrow D$,  we may identify elements $a+b\sqrt{p_0}$ in $\Q_p(\sqrt{p_0})$ (where $a,b\in\Q_p$) with elements of $D$ in what follows without explicitly mentioning it. 

\begin{lemma}\label{action on eta}
$e\cdot \left(\eta(\Phi)\otimes_{\Z_p}\Z_{p^2}\right)=\eta^0(\Phi)\otimes_{\Z_p}\Z_{p^2}$ and $e\cdot \left(T(\Phi)\otimes_{\Z_p}\Z_{p^2}\right)= T^0(\Phi)\otimes_{\Z_p}\Z_{p^2}$. 
\end{lemma}

\begin{proof} The action of $\mathcal{O}_D$ on 
$\eta(\Phi)$ 
is induced by duality from the action on $\mathbb{M}(\Phi)$, so  
any element $a\in\Z_{p^2}\hookrightarrow\mathcal{O}_D$ acts on $\eta^0(\Phi)$ by 
multiplication by $a$ and on $\eta^1(\Phi)$ by multiplication by $\bar{a}$. 
On the other hand, the action of $1\otimes a$ on $\eta(\Phi)\otimes_{\Q_p}\Q_{p^2}$ is given by multiplication by $a$. 
An immediate calculation shows then that the action of $e$ is just  the projection $\eta(\Phi)\rightarrow\eta^0(\Phi)$. The argument for $T(\Phi)$ is similar. \end{proof}

Write $\eta_0^\unr=\pi^{-1}_{\mathcal{H}_p^\unr}\eta^\univ_0$, 
$T_0^\unr=\pi_{\mathcal{H}_p^\unr}^*{T}_0^\univ$, 
$u_0^\unr=\pi^{-1}_{\mathcal{H}_p^\unr}{u}_0^\univ$.

\begin{proposition}\label{idempotent} 
$e\cdot \eta^\unr=\eta_0^\unr$ and $e\cdot T^\unr=T_0^\unr$. 
\end{proposition}

\begin{proof} This is clear from Lemma \ref{action on eta} and \eqref{universal eta}. 
\end{proof}

For $i=0,1$, the sheaf $T_0^\unr$ is a free $\mathcal{O}_{\mathcal{H}_p^\unr}$-module 
of rank $1$, so it is invertible; denote $(T_0^\unr)^\vee$ its $\mathcal{O}_{\mathcal{H}_p^\unr}$-dual. Taking duals we get a map 
$du_0:(T^\unr_0)^\vee\rightarrow(\eta_0\otimes_{\Z_p}\Q_p)^\vee\otimes_{\Q_p}\mathcal{O}_{\mathcal{H}_p^\unr}$ (where 
the RHS denotes $\Q_p$-duals as above). 
We set up the following notation: 
\begin{itemize}
\item $\underline{\omega}_{\mathcal{G}}^0=e\cdot\underline\omega_\mathcal{G}^\rig$; 
\item $\mathcal{L}^0_{\mathcal{G}}=e\cdot\mathcal{H}^{1,\rig}_\mathrm{dR}(\mathcal{G})$. 
\end{itemize}
Applying the idempotent $e$ and using Propositions \ref{identifications}  and \ref{idempotent} 
we then obtain a diagram with exact rows in which the vertical arrows are isomorphisms: 
\begin{equation}\label{diagram 2} 
\xymatrix{
0\ar[r]&
(T^\unr_0)^\vee\ar[r]^-{du_0}\ar[d]^-\simeq&
(\eta_0\otimes_{\Z_p}\Q_p)^\vee\otimes_{\Q_p}\mathcal{O}_{\mathcal{H}_p^\unr}\ar[d]^-\simeq\\
0\ar[r]&
\underline{\omega}_{\mathcal{G}}^0\ar[r]&
\mathcal{L}_{\mathcal{G}}^0}
\end{equation}

\subsection{Differential calculus on the $p$-adic upper half plane}\label{section differential operators} 

We now set up the following notation. Recall that the map $u_0$ takes 
$x_{0,0}e_{0,0}+x_{0,1}e_{0,1}$ to $(zx_{0,0}+x_{0,1})\otimes t_0$; 
dualizing, $du_0$ can be described in coordinates by the map 
which takes the canonical generator 
$t_0$ of the $\mathcal{O}_{\mathcal{H}_p^\unr}$-module
$(T^\unr_0)^\vee$ (satisfying the relation $dt_0(t_0)=1$) 
to the map $x_{0,0}e_{0,0}+x_{0,1}e_{0,1}\mapsto 
zx_{0,0}+x_{0,1}$. If we denote $de_{0,i}$ the dual basis of $e_{0,i}$ 
(satisfying the condition $de_{0,i}(e_{0,j})=\delta_{i,j}$), we may write 
this map as $zde_{0,0}+de_{0,1}$. To simplify the notation, we put 
from now on $\tau=t_0$, $d\tau=dt_0$, 
$x=e_{0,0}$, $y=e_{0,1}$, $dx=de_{0,0}$ and $dy=de_{0,1}$, 
so that the above map reads simply as \[d\tau=zdx+dy.\] 

Let $\mathcal{C}=\mathcal{C}^0(\mathcal{H}_p(\hat{\Q}_{p}^\unr),\C_p)$ denote the $\C_p$-vector space of continuous (for the standard $p$-adic topology on both spaces) $\C_p$-valued functions on ${\mathcal{H}_p(\hat{\Q}_{p}^\unr)}$. Denote 
$\mathcal{A}=H^0(\mathcal{H}_p^\unr,\mathcal{O}_{\mathcal{H}_p^\unr})$ the $\Q_p^\unr$-vector space of global sections of $\mathcal{O}_{\mathcal{H}_p^\unr}$. Each $f\in\mathcal{A}$ is, in particular, continuous on $\mathcal{H}_p^\unr$ for the standard $p$-adic topology of $\Q_p^\unr-\Q_p$, 
and therefore restriction induces a map of $\Q_p^\unr$-vector spaces 
$r:\mathcal{A}\rightarrow\mathcal{C}$. 
Denote $\mathcal{A}^*$ the image of the morphism of $\mathcal{A}$-algebras 
$\mathcal{A}[X,Y]\rightarrow\mathcal{C}$ defined by sending $X$ to the 
function $z\mapsto 1/(z-\sigma(z))$ and $Y$ to the function 
$z\mapsto\sigma(z)$ (note that the function $z\mapsto z-\sigma(z)$ is invertible on 
$\mathcal{H}_p^\unr(\hat{\Q}_p^\unr)$). To simplify the notation, we put from now on 
\[z^\ast=\sigma(z).\]
 
Set up the following notation
(here $n\geq 1$ is an integer)
\begin{itemize}
\item $\Lambda_\mathcal{G}=H^0(\mathcal{H}_p^\unr,\mathcal{L}_\mathcal{G}^0)$ and 
$\Lambda_{\mathcal{G},n}=\Lambda_\mathcal{G}^{\otimes{n}}$,
\item $\Lambda_{\mathcal{G}}^\ast=\Lambda_{\mathcal{G}}\otimes_\mathcal{A}\mathcal{A}^\ast$ and $\Lambda_{\mathcal{G},n}^{\ast}=\left(\Lambda_{\mathcal{G},n}^{\ast}\right)^{\otimes{n}}$,
\item $w_\mathcal{G}=H^0(\mathcal{H}_p^\unr,\underline\omega_\mathcal{G}^0)$ and $w_{\mathcal{G},n}=w_\mathcal{G}^{\otimes{n}}$,
\item ${w}_{\mathcal{G}}^\ast=w_{\mathcal{G}}\otimes_\mathcal{A}\mathcal{A}^\ast$ and $w_{\mathcal{G},n}^{\ast}=\left(w_{\mathcal{G},n}^{\ast}\right)^{\otimes{n}}$.
\end{itemize} 

The $\hat{\Q}_p^\unr$-algebra $\mathcal{A}$ is equipped with the standard 
derivation $\frac{d}{dz}$ on power series. The $\mathcal{A}$-module  
$\Omega^1_\mathcal{A}=H^0(\mathcal{H}_{p}^\unr,\Omega^1_{\mathcal{H}_p^\unr})$
is then one dimensional and generated by $dz$ satisfying $dz\left(\frac{d}{dz}\right)=1$. 
We extend differential operator $\frac{d}{dz}$ 
to a differential operator 
$\frac{\partial}{\partial{z}}:\mathcal{A}^\ast\rightarrow\mathcal{A}^\ast$
by  $\hat{\Q}_p^\unr$-linearity using the product formula 
and setting $\frac{\partial}{\partial{z}}(z^\ast )=0$ and $\frac{\partial}{\partial{z}}\left(\frac{1}{z-z^\ast }\right)=\frac{-1}{(z-z^\ast )^2}$. Similarly, we define a differential operator 
$\frac{\partial}{\partial{z^\ast }}:\mathcal{A}^\ast\rightarrow\mathcal{A}^\ast$
setting $\frac{\partial}{\partial{z^\ast }}(z)=0$, 
$\frac{\partial}{\partial{z^\ast }}(z^\ast )=1$  
and $\frac{\partial}{\partial{z}}\left(\frac{1}{z-z^\ast }\right)=\frac{1}{(z-z^\ast )^2}$. 
Define $\Omega^{1}_{\mathcal{A}^\ast}$ to be the
$\mathcal{A}^\ast$-subalgebra of the algebra of derivations  
generated by $dz$ and $dz^\ast$ satisfying the usual rules 
$dz\left(\frac{\partial}{\partial z}\right)=1$, $dz\left(\frac{\partial}{\partial z^\ast}\right)=0$,
$dz^\ast\left(\frac{\partial}{\partial z^\ast}\right)=0$, $dz^\ast\left(\frac{\partial}{\partial z^\ast}\right)=1$.

\subsection{Splitting of the rigid analytic Hodge filtration}\label{splitting section}
Recall the notation fixed before for the differential form  
$d\tau=zdx+dy$. Define 
\[d\tau^\ast=z^\ast dx+dy.\] Then $d\tau^\ast$ belongs to $w_\mathcal{G}^\ast$. 
Taking global sections, restricting to $\hat{\Q}_{p}^\unr$, 
and extending linearly with $\mathcal{A}^\ast$ we obtain 
a short exact sequence of $\mathcal{A}^\ast$-algebras 
\begin{equation}\label{Hodge p-adic} 
0\longrightarrow{w}_{\mathcal{G}}^\ast\longrightarrow
\Lambda_{\mathcal{G}}^\ast.\end{equation}

\begin{theorem}\label{splitting theorem}
The exact sequence \eqref{Hodge p-adic} 
admits a canonical splitting $\Psi_p:\Lambda_{\mathcal{G}}^\ast\rightarrow
{w}_{\mathcal{G}}^\ast$.\end{theorem}
 
\begin{proof} 
We have
\[dx=\frac{d\tau-d\tau^\ast}{z-z^\ast },\qquad
dy= \frac{zd\tau^\ast-z^\ast d\tau}{z-z^\ast}.\] 
We may therefore write any differential form $\omega=f(z)dx+g(z)dy$ with $f,g\in\mathcal{A}^*$ as 
\[\omega=\left(\frac{f(z)-g(z)z^\ast }{z-z^\ast }\right)d\tau+d\tau^\ast\left(\frac{zg(z)-f(z)}{z-z^\ast }\right).\]
One then defines the sough-for splitting sending 
$\omega\mapsto\left(\frac{f(z)-g(z)z^\ast }{z-z^\ast }\right)d\tau$. 
\end{proof}


\subsection{The $p$-adic Shimura-Maass operator} 
Taking global section, the Gauss-Manin connection gives rise to a map 
$\nabla^\rig_\mathcal{G}:\Lambda_\mathcal{G}\rightarrow \Lambda_\mathcal{G}\otimes\Omega^1_\mathcal{A}$. We extend $\nabla_\mathcal{G}^\rig$ to a map 
$\nabla_\mathcal{G}^\ast:\Lambda_\mathcal{G}^\ast\rightarrow \Lambda_\mathcal{G}^\ast\otimes\Omega^1_{\mathcal{A}^\ast}$
as follows. First define   
$
\nabla_\mathcal{G}^{1,0}:\Lambda_\mathcal{G}^\ast\rightarrow \Lambda_\mathcal{G}^\ast\otimes\Omega^1_{\mathcal{A}^\ast}
$
to be the derivation satisfying the rules 
$\nabla^{1,0}_\mathcal{G}(d\tau)=dx\otimes dz$, 
$\nabla^{1,0}_\mathcal{G}(d\tau^\ast)=0$,
$\nabla^{1,0}_\mathcal{G}(z)=1$,
$\nabla^{1,0}_\mathcal{G}(z^\ast)=0$. 
Define similarly the derivation  
$
\nabla_\mathcal{G}^{0,1}:\Lambda_\mathcal{G}^\ast\rightarrow \Lambda_\mathcal{G}^\ast\otimes\Omega^1_{\mathcal{A}^\ast}
$
by the rules
$\nabla^{0,1}_\mathcal{G}(d\tau)=0$,
$\nabla^{0,1}_\mathcal{G}(d\tau^\ast)=dx\otimes dz^\ast$,
$\nabla^{0,1}_\mathcal{G}(z)=0$,
$\nabla^{0,1}_\mathcal{G}(\bar z)=dz^\ast$. 
We finally define 
\[\nabla^\ast_\mathcal{G}=\nabla^{1,0}_\mathcal{G}+\nabla^{0,1}_\mathcal{G}:\Lambda_\mathcal{G}^\ast\longrightarrow \Lambda_\mathcal{G}^\ast\otimes\Omega^1_{\mathcal{A}^\ast}
.\]
%
%

Taking global sections, the Kodaira-Spencer map gives rise to a map 
$\mathrm{KS}_\mathcal{G}:w_\mathcal{G}^{\otimes{2}}\rightarrow\Omega^1_\mathcal{A}$, which we extend $\mathcal{A}^\ast$-linearly 
to a map
\[\mathrm{KS}_\mathcal{G}^\ast:(w_\mathcal{G}^\ast)^{\otimes{2}}
\longrightarrow\Omega^{1}_{\mathcal{A}^\ast}.\] 

Note that \begin{equation}\label{GM}
\nabla_\mathcal{G}^\rig(d\tau)=\nabla_\mathcal{G}^\ast(d\tau)=
\frac{d\tau-d{\tau}^\ast}{z-z^\ast}\otimes dz.\end{equation}
and, since $\nabla^\ast_\mathcal{G}(z^\ast)=0$, we have 
\begin{equation}\label{horizontal} 
\nabla^\ast_\mathcal{G}(d\tau^\ast)=0.\end{equation}
In particular, if $f(z)\otimes d\tau\in w_\mathcal{G}^\ast$ 
we have 
\begin{equation}\label{formula 1}
\nabla_\mathcal{G}^\rig\left(f(z)\otimes d\tau\right)=
\left(\frac{\partial}{\partial z}f(z)\otimes d\tau+f(z)\otimes\frac{d\tau-d\tau^*}{z-z^\ast}\right)\otimes dz.\end{equation} 

Taking global sections, we can form the pairing
$\langle,\rangle_\mathcal{G}^\rig:\Lambda_\mathcal{G}\otimes_{\mathcal{A}}
\Lambda_\mathcal{G}\rightarrow\mathcal{A}$. 
Extending linearly by $\mathcal{A}^\ast$, we obtain a new pairing 
\[\langle,\rangle_\mathcal{G}^\ast:\Lambda_\mathcal{G}^\ast\otimes_{\mathcal{A}^\ast}
\Lambda_\mathcal{G}^\ast\longrightarrow\mathcal{A}^\ast.\]
Using the description of the Kodaira-Spencer map in the end of
\S\ref{universal SFD subsection}, we see that 
\[
\langle d\tau,\nabla_\mathcal{G}^\rig(d\tau)\rangle_\mathcal{G}^\rig
= \langle d\tau,\nabla_\mathcal{G}^\ast(d\tau)\rangle_\mathcal{G}^\ast
=
\frac{-\langle d\tau,d\tau^\ast\rangle_{\mathcal{G}}^\ast}{z-z^\ast} dz
=-\langle dx,dy\rangle_\mathcal{G}^\ast dz
\]where for the second equality we use \eqref{GM}, while the last 
equality easily from the equality  
$\langle zdx+dy,z^\ast dx+dy\rangle_{\mathcal{G}}^
\ast=(z-z^\ast)\langle dx,dy\rangle_\mathcal{G}^\ast$. Therefore 
\[\mathrm{KS}_\mathcal{G}^\rig(d\tau\otimes d\tau)=-\langle dx,dy\rangle_\mathcal{G}^\rig dz.\]
So, to compute 
$\mathrm{KS}_\mathcal{G}^\ast(d\tau\otimes d\tau)=\mathrm{KS}_\mathcal{G}^\rig(d\tau\otimes d\tau)$ we are reduced to compute 
$\langle dx,dy\rangle_\mathcal{G}^\rig$. 
For this, we switch to de Rham homology and follow the computations in \cite{Mori}, \cite{Brooks}. 

To begin with, let $W$ denote the order $\mathcal{O}_D$ viewed as free left $\mathcal{O}_D$-module of rank 1; then $W\simeq\mathcal{R}_\mathrm{max}\otimes_{\Z}\Z_p$.  
By \cite[Ch. III, Lemma 1.9]{BC}, the collection of bilinear skew-symmetric maps $\psi:W\times W\rightarrow\Z_p$ which satisfy $\psi(dx,y)=\psi(x,d^\dagger y)$ (for all $x,y\in W$ and $d\in\mathcal{O}_D$) is a free $\Z_p$-module of rank $1$,  
and every generator $\psi_0$ of this $\Z_p$-module is a perfect duality on $W$; the pairing 
\[\psi_0(x,y)=\frac{\tr(\mathbf{i} y^\dagger x)}{p}\] is such a generator, which we fix once and for all
(recall the notation introduced in \S\ref{quaternion algebras} and \S\ref{algebraic de Rham} for $\mathbf{i}$ and $d^\dagger$). 

Recall that $H^1_\mathrm{cris}(\Phi)$ is 
a free $D\otimes_{\Q_p}\hat{\Q}_p^\unr$-module of rank $1$ 
(\emph{cf.} \cite[page 354]{IS}); the structure of $D\otimes_{\Q_p}\hat{\Q}_p^\unr$-module 
is induced from the $D$-module structure of $(\eta(\Phi)\otimes_{\Z_p}\Q_p)^\vee$ via
the isomorphisms $(\eta(\Phi)\otimes_{\Z_p}\Q_p)^\vee\simeq V_\mathrm{cris}(\Phi)$ 
and $H^1_\mathrm{cris}(\Phi)\simeq V_\mathrm{cris}(\Phi)\otimes_{\Q_p}\hat{\Q}_p^\unr$ 
in Lemma \ref{prop1}. We have then from Lemma \ref{prop1} 
then canonical isomorphisms
of convergent $F$-isocrystals:  
\begin{equation}\label{isocristals} 
\begin{split}
\mathcal{H}^1_\mathrm{dR}(\mathcal{G}/\hat{\mathcal{H}}_p)&\simeq 
H^1_\mathrm{cris}(\Phi)\otimes_{\hat{\Q}_{p}^\unr}\mathcal{E}(\mathcal{O}_{\hat{\mathcal{H}}_p^\unr})\\ 
&\simeq \left(D\otimes_{\Q_p}\hat{\Q}_p^\unr\right)\otimes_{\hat{\Q}_{p}^\unr}\mathcal{E}(\mathcal{O}_{\hat{\mathcal{H}}_p^\unr})\\ 
&\simeq \left(D\otimes_{\Q_p}{\Q}_{p^2}\right)\otimes_{{\Q}_{p^2}}\mathcal{E}(\mathcal{O}_{\hat{\mathcal{H}}_p^\unr})\\ 
&\simeq \left(D\otimes_{\Q_p}\hat{\Q}_p^\unr\right)\otimes_{\hat{\Q}_{p}^\unr}\mathcal{E}(\mathcal{O}_{\hat{\mathcal{H}}_p^\unr})\\ 
&\simeq \M_2(\Q_{p^2})\otimes_{{\Q}_{p^2}}\mathcal{E}(\mathcal{O}_{\hat{\mathcal{H}}_p^\unr}).
\end{split}
\end{equation}
Let $\psi_0$ denote the $\hat{\Q}_{p}^\unr$-linear extension of $\psi_0$; under the isomorphism \eqref{isocristals}, $\psi_0$ defines a pairing $H^1_\mathrm{cris}(\Phi)\times H^1_\mathrm{cris}(\Phi)\rightarrow\hat{\Q}_p^\unr$ still denoted by $\psi_0$.  
%
If we still denote $\langle,\rangle_{\mathcal{G}}^\rig$ the restriction of 
$\langle,\rangle_\mathcal{G}^\rig$ to $H^1_\mathrm{cris}(\Phi)$, 
it follows from the unicity of $\psi_0$ up to constant that 
there exists an element $t_p\in \C_p^\times$ such that 
\begin{equation}\label{pairings}
\langle,\rangle_{\mathcal{G},W}^\rig=\frac{1}{t_p}\cdot\psi_0.\end{equation}
%
Moreover, under the isomorphism \eqref{isocristals}, the element 
$d\tau=zdx+dy$ of $\mathcal{H}^1_\mathrm{dR}(\mathcal{G})$ 
corresponds to the element 
$e_1\otimes{z}+e_2\otimes1$, where $e_1=\smallmat 1001$, $e_2=\smallmat 0100$, $e_3=\smallmat 0010$, 
$e_4=\smallmat 0001$ is the standard basis of $\M_2(\Q_{p^2})$. We therefore obtain the sought-for  recipe to compute the Kodaira-Spencer image of $d\tau\otimes d\tau$ in 
terms of $\psi_0$:  
\begin{equation}\label{KS}\mathrm{KS}_\mathcal{G}^\rig(d\tau\otimes d\tau)= 
\frac{1}{t_p}\cdot\psi_0(e_1\otimes z,e_2\otimes 1).\end{equation}

\begin{remark}
The number $t_p$ 
may be viewed as the $p$-adic analogue of the complex 
period $2\pi i$, relating de Rham cohomology with homology 
(\cite[(2.7)]{Mori}, \cite[p. 4197]{Brooks}). 
This explains why we prefer to keep $t_p$ at the denominator in \eqref{pairings}. 
\end{remark} 

We now make more explicit the equations \eqref{GM} and \eqref{KS} 
using Hashimoto basis. For this part, we follow closely the nice calculations in 
\cite[Prop. 2.3]{Mori}, to which 
the reader is referred to for details. 
Recall the Haschimoto basis $\{\mathbf{1},\mathbf{i},\mathbf{j},\mathbf{k}\}$ in 
\S\ref{quaternion algebras}.  
As in \cite[(2)]{Hashimoto}, define 
$\epsilon_1=\mathbf{1}$, $\epsilon_2=(\mathbf{1}+\mathbf{j})/2$, 
$\epsilon_3=(\mathbf{i}+\mathbf{i}\mathbf{j})/2$, $\epsilon_4=(apN^-\mathbf{j}+\mathbf{i}\mathbf{j})/{p_0}$ and use these elements to define 
a symplectic basis of $W$ 
with respect to the pairing $\psi_0$ as in \cite[(5)]{Hashimoto} by 
$\eta_1=\epsilon_3-\frac{p_0-1}{2}\epsilon_4$, 
$\eta_2=-aD\epsilon_1-\epsilon_4$, 
$\eta_3=\epsilon_1$, 
$\eta_4=\epsilon_2$ (note that $\psi_0$ 
we consider above is equal to the pairing $(x,y)\mapsto \tr(xiy^\dagger)$ in 
\cite[(3)]{Hashimoto}). 
Denote $\eta_1^\vee,\eta_2^\vee,\eta_3^\vee,\eta_4^\vee$ 
the dual basis of $W^\vee$, and let $\underline{\eta}^\vee$ be the column vector 
with entries  $\eta_1^\vee,\eta_2^\vee,\eta_3^\vee,\eta_4^\vee$. 
The elements $\eta_i^\vee$ give rise to elements of 
$\mathcal{H}^1_{\mathrm{dR}}(\mathcal{G})$, 
denoted with the same symbol, which are horizontal 
with respect to $\nabla^\rig_\mathcal{G}$, namely $\nabla_\mathcal{G}^\rig(\eta_i^\vee)=0$. 
Write 
$d\tau=\Pi(z)\cdot\underline{\eta}^\vee$. 
A simple calculation 
shows that 
\begin{equation}\label{Pi}
\Pi(z)= \left({\frac{\alpha^-}{2\sqrt{p_0}}(\alpha^+a\Delta z+1)}, 
{\frac{-1}{\sqrt{p_0}}(\alpha^+a\Delta z+1)},z,{\frac{1}{2}\alpha^+z}\right).
\end{equation}
Since $\eta_i^\vee$ are horizontal sections 
of $\nabla_\mathcal{G}^\rig$, using 
\eqref{Pi} to calcolate $\frac{d\Pi(z)}{dz}$ shows that \eqref{GM} becomes  
\begin{equation}\label{GM explicit}
\nabla_\mathcal{G}^\rig(d\tau)= 
\left({\frac{\alpha^-\alpha^+a\Delta}{2\sqrt{p_0}}}, 
{\frac{-\alpha^+a\Delta z}{\sqrt{p_0}}},1,{\frac{1}{2}\alpha^+}\right)\cdot\underline\eta^\vee\otimes dz.\end{equation}
The recipe \eqref{KSdual} to compute the Kodaira-Spencer map
combined with \eqref{pairings} and \eqref{GM explicit} gives then  
\begin{equation}\label{formula 2}
\mathrm{KS}_\mathcal{G}^\rig(d\tau\otimes d\tau)=
\frac{1}{t_p}\frac{d\Pi(z)}{dz}\begin{pmatrix}0&I_2\\-I_2&0
\end{pmatrix}
\Pi(z)^Tdz = \frac{1}{t_p}dz. 
\end{equation}
In particular, \eqref{formula 2} shows that $\mathrm{KS}_\mathcal{G}^\rig$ is an isomorphism, and therefore the same is true for
$\mathrm{KS}_\mathcal{G}^\ast$. This allows us to define the operator 
\[\xymatrix{
\tilde\nabla_{\mathcal{G},n}^\ast: \Lambda_{\mathcal{G},n}^\ast
\ar[r]^-{\nabla_{\mathcal{G},n}^\ast}&
\Lambda_{\mathcal{G},n}^\ast\otimes_{\mathcal{A}^\ast}\Omega^1_{\mathcal{A}^\ast}
\ar[r]^-{(\mathrm{KS}_\mathcal{G}^\ast)^{-1}}&
\Lambda_{\mathcal{G},n}^\ast\otimes_{\mathcal{A}^\ast}w_{\mathcal{G},2}^\ast\ar[r]^-{} & \Lambda_{\mathcal{G},n}^\ast\otimes_{\mathcal{A}^\ast}
\Lambda_{\mathcal{G},2}^\ast\ar[r]^-{} & \Lambda_{\mathcal{G},n+2}^\ast
}
\] where $\nabla_{\mathcal{G},n}^\ast$ is obtained from 
$\nabla_\mathcal{G}^\ast$ using the Leibniz rule (as before, see 
for example \cite[\S 3.2]{Brooks}). The splitting in Theorem \ref{splitting theorem} induces 
a morphism of $\mathcal{A}^\ast$-modules
$\Psi_{p,n}^\ast:\Lambda_{\mathcal{G},n}^{\ast}\rightarrow
w^{\ast}_{\mathcal{G},n}$. 
The composition 
\[
\xymatrix{\Theta_{p,n}^\ast:
w_{\mathcal{G},n}^\ast\ar[r]&
\Lambda_{\mathcal{G},n}^\ast\ar[r]^-{\tilde\nabla_{\mathcal{G},n}^\ast}&
\Lambda_{\mathcal{G},n+2}^\ast\ar[r]^-{\Psi_{p,n+2}^\ast}&
w^{\ast}_{\mathcal{G},n+2}
}
\]
is the \emph{$p$-adic Shimura-Maass operator}. 
We also need to consider iterates of this operator. Define 
\[\tilde\nabla_n^{j,\ast}=\tilde\nabla_{n+2j}^\ast\circ\dots\circ\tilde\nabla_{n+2}^\ast\circ\tilde\nabla_n^\ast.
\] Define then 
\begin{equation}\label{definition Theta j}
\xymatrix{\Theta_{p,n}^{j,\ast}:
w_{\mathcal{G},n}^\ast\ar[r]&
\Lambda_{\mathcal{G},n}^\ast\ar[r]^-{\tilde\nabla_n^{j,\ast}}&
\Lambda_{\mathcal{G},n+2j}^\ast\ar[r]^-{\Psi_{p,n+2j}^\ast}&
w^{\ast}_{\mathcal{G},n+2j}
}\end{equation} 
where the morphism of $\mathcal{A}^\ast$-modules
$\Psi_{p,n+2j}^\ast:\Lambda_{\mathcal{G},n+2j}^{\ast}\rightarrow
w^{\ast}_{\mathcal{G},n+2j}$ is induced as above by 
the splitting in Theorem \ref{splitting theorem}. We call  
$\Theta_{p,n}^{\ast,j}$ the  
$j$-th iterate of the $p$-adic Shimura-Maass operator. 

The work accomplished so far allows us to explicitly describe  
$\Theta_{p,k}^{\ast,j}$. We first introduces some differential operators, similar 
in shape to the Shimura-Maass operator in the real analytic setting.  
For each integer $k\geq 0$, we may then define the $\hat{\Q}_p^\unr$-linear 
function 
\begin{equation}\label{p-adic Shimura-Maass}
\delta_{p,k}=\frac{\partial}{\partial{z}}+\frac{k}{z-z^\ast}:
\mathcal{A}^\ast\longrightarrow \mathcal{A}^\ast.\end{equation}
For each integer $j\geq 0$ we get a map 
$\delta_{p,k}^j:\mathcal{A}^\ast\rightarrow\mathcal{A}^\ast$ defined by 
\[\delta_{p,k}^j=\delta_{p,k+2(j-1)}\circ\dots\circ\delta_{p,k+2}\circ\delta_{p,k}.\]
We call $\delta_{p,k}$ the \emph{Shimura-Maass} operator, 
and $\delta_{p,k}^j$ its $j$-th iteration. 
%
%
Applying \eqref{formula 2} to compute the inverse of the Kodaira-Spencer map
to \eqref{formula 1}, we obtain 
\[\tilde\nabla_k^\ast\left(f(z)\otimes d\tau^{\otimes{k}}\right)=
\frac{1}{t_p}\cdot \left(\frac{\partial}{\partial z}f(z)\otimes d\tau^{\otimes{k+2}}+
kf(z)\otimes\frac{d\tau-d\sigma(\tau)}{z-\sigma(z)}\otimes d\tau^{\otimes{k+1}}\right).\] 
Applying the splitting ${\Psi}_{p,k+2}^\ast$ of the Hodge filtration which annihilate $d\tau^\ast$, we finally obtain 
\begin{equation}\label{Shimura-Maass explicit formula}
\Theta_{p,k}^\ast\left(f(z)\otimes d\tau^{\otimes{k}}\right)=
\frac{1}{t_p}\cdot \left(\frac{\partial}{\partial z}f(z) +
\frac{kf(z)}{z-\sigma(z)}\right)\otimes d\tau^{\otimes{k+2}}=
\frac{1}{t_p}\cdot\delta_{p,k}(f)\otimes d\tau^{\otimes{k+2}}.\end{equation}
Iterating \eqref{Shimura-Maass explicit formula} we obtain  
\begin{equation}\label{Shimura-Maass explicit formula iterates}
\Theta_{p,k}^{j,\ast}\left(f(z)\otimes d\tau^{\otimes{k}}\right)=\left(\frac{1}{t_p}\right)^j\delta_{p,k}^j\left(f(z)\right)\otimes d\tau^{\otimes{k+2j}}.\end{equation}

\section{The $p$-adic Shimura-Maass operator on Shimura curves} 

The aim of this Section is to apply the results in Section \ref{Drinfelderia} 
to define a $p$-adic Shimura-Maass operator on Shimura curves. 

\subsection{$p$-adic uniformization of Shimura curves}  

In this subsection we review the Cerednik-Drinfel'd Theorem.
Let $B/\Q$ be the quaternion algebra obtained from $\mathcal{B}$ by 
interchanging the invariants at $\infty$ and $p$; so $B$ is the definite 
quaternion algebra over $\Q$ of discriminant $N^-$. For a subgroup $U\subseteq\hat{B}^\times$, let
$U^{(p)}$ the elements outside of the place $p$. 
Fix isomorphisms $\mathcal{B}_\ell\simeq B_\ell$ for all primes $\ell\neq p$, so that we can view  $V_1(N^+)^{(p)}$ as a subgroup of $(\hat{B}^\times)^{(p)}$.
Define $\tilde\Gamma_p=B^\times\cap V_1(N^+)^{(p)}$. We still denote $\tilde\Gamma_p$ the image of $\tilde{\Gamma}_p$ in 
$\GL_2(\Q_p)$ via a fixed isomorphism $i_p:B\otimes_\Q\Q_p\simeq\M_2(\Q_p)$, 
and we let ${\Gamma}_p$ denote the subgroup of $\tilde\Gamma_p$ 
consisting of elements whose determinant has even $p$-power order. 

Base changing from $\Z_p$ to the valuation ring $\Z_{p^2}$ of $\Q_{p^2}$ 
gives a $\Z_{p^2}$-formal scheme $\hat{\mathcal{H}}_{p^2}$, whose 
generic fiber $\mathcal{H}_{p^2}$ is the base change of the $\Q_p$-rigid analytic 
space $\mathcal{H}_p$ to $\Q_{p^2}$. 
The group $\GL_2(\Q_p)$ acts on the $\Z_p$-formal scheme $\hatHp$ 
(\cite[Chapitre I, \S 6]{BC}) and acts on $\Spf({\Z}_{p^2})$ via the inverse of the arithmetic Frobenius 
raised to the determinant map (\cite[Chapitre II, \S 9]{BC}). 
Therefore, the group $\GL_2(\Q_p)$ also acts on the 
$\Z_{p^2}$-formal scheme $\hat{\mathcal{H}}_{p^2}$ and 
the
$\hat{\Z}_{p}^\unr$-formal scheme $\mathcal{\mathcal{H}}_{p}^\unr$, 
and the associated rigid analytic spaces.  
We may then form the quotient $\Gamma_p\backslash{\hat{\mathcal{H}}}_{p^2}$, 
in the category of $\Z_{p^2}$-formal schemes, and the quotient
$\Gamma_p\backslash{\hat{\mathcal{H}}}_{p}^\unr$, 
in the category of $\hat{\Z}_{p}^\unr$-formal schemes, and similarly 
for the associated rigid analytic spaces. 
The formal completion $\hat{\mathcal{A}}_{\Z_{p^2}}$ of the universal abelian variety 
$\mathcal{A}_{\Z_{p^2}}$ over $\mathcal{C}_{\Z_{p^2}}$ along its special fiber 
is a $\SFD$-module over the formal completion $\hat{\mathcal{C}}_{\Z_{p^2}}$ of
$\mathcal{C}_{\Z_{p^2}}$ along its special fiber. We may base change $\hat{\mathcal{A}}_{\Z_{p^2}}$ and $\hat{\mathcal{C}}_{\Z_{p^2}}$ to 
$\hat{\Z}_{p^2}$ obtaining a $\SFD$-module  $\hat{\mathcal{A}}_{\hat{\Z}_{p}^\unr}$
over the formal scheme $\mathcal{C}_{\hat{\Z}_{p}^\unr}$; of course,  $\hat{\mathcal{A}}_{\hat{\Z}_{p}^\unr}$ is the completion of $\mathcal{C}_{\hat{\Z}_{p}^\unr}$ along its special fiber. 
The Cerednik-Drinfel'd Theorem (\cite{Dr}, \cite[Th\'eor\`eme 5.3]{BC}) 
states the existence of an isomorphism of $\hat{\Z}_{p}^\unr$-formal schemes 
$\Gamma_p\backslash{\mathcal{H}}^\unr_{p}
\simeq
{\mathcal{C}}_{\hat{\Z}_{p}^\unr}^\rig$
which induced an isomorphism of $\hat{\Z}_p^\unr$-formal schemes
$\Gamma_p\backslash{\mathcal{G}}
\simeq
{\mathcal{A}}_{\hat{\Z}_{p}^\unr}^\rig$ on the universal objects.  
Under our assumptions, there is an isomorphism of $\Z_{p^2}$-formal schemes 
$\Gamma_p\backslash{\mathcal{H}}_{p}^\unr\simeq
\Gamma_p\backslash{\mathcal{H}}_{p^2}$ (\cite[\S3.5.3]{BC}, \cite[Theorem 4.3']{JL})
from which we deduce an isomorphism of $\Z_{p^2}$-formal schemes 
$\Gamma_p\backslash{\hat{\mathcal{H}}}_{p^2}
\simeq
{\hat{\mathcal{C}}}_{ {\Z}_{p^2}}
$
which induces an isomorphism of $\Z_{p^2}$-formal schemes, equivariant for 
the quaternionic actions on both sides, 
$
\Gamma_p\backslash{\mathcal{G}}
\simeq
{\mathcal{A}}_{ {\Z}_{p^2}}.$
%
%
%

Denote $(X,\mathcal{O}_X)\leadsto (X^\rig,\mathcal{O}_X^\rig)$ 
the rigidification functor which takes a proper scheme over a complete extension 
$F$ of $\Q_p$ to  
to its associated rigid analytic space over $F$ (\cite[\S5.4]{Bosch}).
For each coherent sheaf $\mathcal{F}$ of $\mathcal{O}_X$-modules on $X$, we also denote $\mathcal{F}^\rig$ 
the rigidification of $\mathcal{F}$, and for each morphism 
$\varphi:\mathcal{F}\rightarrow\mathcal{G}$ of $\mathcal{O}_X$-modules, we let $\varphi^\rig:\mathcal{F}^\rig\rightarrow\mathcal{G}^\rig$ the corresponding morphism of rigid analytic sheaves (\cite[\S6]{Bosch}).
We have a rigid version of GAGA stating that $\mathcal{F}\leadsto\mathcal{F}^\rig$ is an equivalence of categories between coherent $\mathcal{O}_X$-modules 
and coherent $\mathcal{O}_X^\rig$-modules; we refer to 
\cite[\S 6.3, Theorems 11, 12, 13]{Bosch}, or \cite{Ursula-Kopf} for details. Moreover, if 
$X$ is a proper $\mathcal{O}_F$-scheme, where $\mathcal{O}_F$ is the valuation ring of $F$, 
the generic fiber of the formal completion $\hat{X}$ of $X$ along its special fiber 
coincides with $X_F^\rig$, where $X_F=X\otimes_{\mathcal{O}_F}F$. 
Passing to the generic fiber, the Cerednik-Drinflel'd theorem then implies that there are isomorphisms of $\Q_{p^2}$-rigid analytic spaces 
\begin{equation}\label{Rigid CD}
\Gamma_p\backslash{{\mathcal{H}}}_{p^2}
\simeq{{\mathcal{C}}}_{\Q_{p^2}}^\rig\end{equation} 
and an isomorphism of $\Q_{p^2}$-rigid analytic spaces which is equivariant with 
respect to the quaternionic actions on both sides: 
\begin{equation}\label{Rigid CD SFD} 
\Gamma_p\backslash{\mathcal{G}^\rig}
\simeq{\mathcal{A}}_{{\Q}_{p^2}}^\rig.\end{equation} 

\subsection{Rigid analytic modular forms} 
A rigid analytic function $f:\mathcal{H}_p(\C_p)\rightarrow\C_p$ is 
said to be a \emph{rigid analytic modular form} of weight $k$ and level $\Gamma_p$ if \[f(\gamma z)=(c z+d)^kf( z)\] for all 
$ z\in\mathcal{H}_p(\C_p)$ and 
$\gamma\in\Gamma_p$, where $\gamma( z)=(a z+b)/(c z+d)$. 
Denote $S_k^\rig(\Gamma_p)$ the $\C_p$-vector space of 
rigid analytic modular forms of weight $k$ and level $\Gamma_p$. 
See \cite[\S5.2]{Darmon-Book} for details. 

Given a $\Z[\Gamma_p]$-module $M$, 
we denote $M^{\Gamma_p}$ the submodule consisting of $\Gamma_p$-invariant 
elements of $M$. 
With notation as in \S\ref{splitting section}, 
 define $\hat{\Q}_p^\unr$-subspace 
$w_\mathcal{G}^{\Gamma_p}$ of 
$w_\mathcal{G}$ 
consisting of global sections
which are invariant for the $\Gamma_p$-action.  In particular, 
$w_\mathcal{G}^{\Gamma_p}$ is a 
$\mathcal{A}^{\Gamma_p}$-module. 
Given $f\in S_2^\rig(\Gamma_p)$, define
$\omega_f=f(z)\otimes d\tau^{\otimes{k}}$ in 
$w_{\mathcal{G},\C_p}=w_\mathcal{G}\otimes_{\hat{\Q}_p^\unr}{\C_p}$. 

\begin{lemma}\label{lemma modular forms 1}
The correspondence $f\mapsto\omega_f=f(z)\otimes d\tau^{\otimes{k}}$ sets up a $\C_p$-linear 
isomorphism between $S_k^\rig(\Gamma_p)$ and 
$w_{\mathcal{G},\C_p}^{\Gamma_p}$.
\end{lemma}

\begin{proof}
That $\omega_f$ belongs to 
$w_{\mathcal{G},{\C_p}}^{\Gamma_p}$ is because of \eqref{moebius}. The map $f\mapsto \omega_f$ 
has clearly an inverse because $\underline\omega_\mathcal{G}^0$ is an invertible 
$\mathcal{O}_{\mathcal{H}_p^\unr}$-module, and the result follows. 
\end{proof}

For any sheaf $\mathcal{F}$ on $\mathcal{H}_p^\unr$, denote $\mathcal{F}^{\Gamma_p}$ 
the sheaf on $\Gamma_p\backslash\mathcal{H}_p^\unr$ defined by 
taking $\Gamma_p$-invariant sections. 
Also, recall the sheaves $\underline{\omega}_{\hat{\Q}_p^\unr}$ and 
$\mathcal{L}_{\hat{\Q}_p^\unr}$ introduced in \eqref{def omega} and 
\eqref{def L}. 
\begin{lemma}\label{lemma modular forms 2} The isomorphisms \eqref{Rigid CD} and \eqref{Rigid CD SFD} induce isomorphisms
$(\underline{\omega}_{\mathcal{G}}^0)^{\Gamma_p}\simeq 
\underline{\omega}_{\hat{\Q}_p^\unr}^\rig$ and $(\mathcal{L}_\mathcal{G}^0)^{\Gamma_p}\simeq\mathcal{L}_{\hat{\Q}_p^\unr}^\rig$ of sheaves. 
%
%
\end{lemma}

\begin{proof} Recall that a basis of affinoid subsets of $\Gamma_p\backslash\mathcal{H}_p^\unr$ is given by $\Sp(A^{\Gamma_p})$ where 
$\Sp(A)$ ranges over the affinoid subsets of $\mathcal{H}_p^\unr$ 
such that $\Gamma_p$ acts on $A$ by a finite group (\cite[\S 6]{Drinfeld-Elliptic}).  
It follows that the structural sheaf $\mathcal{O}_{\mathcal{C}^\rig_{\hat{\Q}_p^\unr}}=\mathcal{O}_{\mathcal{C}_{\hat{\Q}_p^\unr}}^\rig$ 
of $\mathcal{C}^\rig_{\hat{\Q}_p^\unr}$ is identified with the sheaf of 
$\Gamma_p$-invariant sections of $\mathcal{H}_p^\unr$. 
The result follows from this, in light of \eqref{Rigid CD}, \eqref{Rigid CD SFD}
and the construction of differentials and de Rham cohomology. 
\end{proof}

\begin{proposition}\label{prop5.3} 
There are canonical isomorphisms of $\C_p$-vector spaces: 
\[S_2^\rig(\Gamma_p)\simeq 
w_{\mathcal{G},{\C_p}}^{\Gamma_p}=
H^0(\mathcal{H}_{p}^\unr,\underline{\omega}_{\mathcal{G}}^0)^{\Gamma_p}_{\C_p}
\simeq 
H^0(\mathcal{C}_{\hat{\Q}_{p}^\unr}^\rig,\underline{\omega}_{\hat{\Q}_{p}^\unr}^\rig)_{\C_p}.\]
\end{proposition}
\begin{proof}
Put together Lemmas \ref{lemma modular forms 1} and \ref{lemma modular forms 2}. 
\end{proof}

\subsection{The $p$-adic Shimura-Maass operator}\label{computations} 
Taking $\Gamma_p$-invariants defines a map, for integers $k\geq 0$ and $j\geq 0$ and understanding that $\Theta^{0,\ast}_{p,k}=\Theta_{p,k}^\ast$, 
\[\Theta_{p,k}^{j,\ast}: (w_{\mathcal{G},n}^\ast)^{\Gamma_p}\longrightarrow 
(w_{\mathcal{G},n+2}^\ast)^{\Gamma_p}
\]
where recall that $\Theta_{p,k}^{j,\ast}$ was introduced in \eqref{definition Theta j}.

An alternative way to introduce $\Theta_{p,k}^{j,\ast}$ is the following. Recall the operator $\tilde\nabla_n$ in \eqref{nabla n} and, for any integer $j\geq 0$, 
define $\tilde\nabla_n^j:
\mathcal{L}_{\hat{\Q}_p^\unr,n}\rightarrow \mathcal{L}_{\hat{\Q}_p^\unr,n+2j}$ by 
the formula 
\[\tilde\nabla_n^j=\tilde\nabla_{n+2j}\circ\dots\circ\tilde\nabla_{n+2}\circ\tilde\nabla_n.
\] Considering the associated rigid analytic sheaves, and taking global sections, we obtain a map of $\mathcal{A}$-modules 
$\tilde\nabla_n^{j,\rig}:\Lambda_{\mathcal{G},n}^{\Gamma_p}\rightarrow\Lambda_{\mathcal{G},n+2j}^{\Gamma_p}$.
One may define the operator  
\[
\xymatrix{\Theta_{p,n}^{j}:
w_{\mathcal{G},n}^{\Gamma_p}\ar[r]^-{i_{\mathcal{G},n}}&
\Lambda_{\mathcal{G},n}^{\Gamma_p}\ar[r]^-{\tilde\nabla_n^{j,\rig}}&
\Lambda_{\mathcal{G},n+2j}^{\Gamma_p}\ar[r]&
\left(\Lambda_{\mathcal{G},n+2j}^\ast\right)^{\Gamma_p}\ar[r]^-{\Psi_{p,n+2j}^\ast}&
\left(w^{\ast}_{\mathcal{G},n+2j}\right)^{\Gamma_p}. 
}
\] By \eqref{horizontal}, $d\tau^\ast$ is horizontal for $\nabla_\mathcal{G}^\ast$, and therefore $\Theta_{p,n}^{j}$ coincides with the restriction of $\Theta_{p,k}^{j,\ast}$ to $w_{\mathcal{G},n}^{\Gamma_p}$.

\subsection{Comparison of Shimura-Maass operators at CM points}

Identify the set of $\Q_{p^2}$-points in the rigid space $\mathcal{H}_p$ 
with the set of $\Q_p$-algebra homomorphisms $\Hom(\Q_{p^2},\M_2(\Q_p))$ as follows: 
any $\Psi\in\Hom(\Q_{p^2},\M_2(\Q_p))$ defines an action of $\Q_{p^2}^\times$ on $\mathcal{H}_p(\Q_{p^2})=\Q_{p^2}-\Q_p$ by fractional linear transformations, and the point 
$z\in\mathcal{H}_p(\Q_{p^2})$ associated with $\Psi$ is characterised 
by the property $\Psi(a)\binom{z}{1}=a\binom{z}{1}$, for all $a\in\Q_{p^2}^\times$. 

Given a representation $\rho=(\rho_1,\rho_2):\GL_2\times\GL_2\rightarrow\GL(V)$, 
the stalk $\mathcal{E}(V)_\Psi$ of $\mathcal{E}(V)$ at a point $\Psi\in\Hom(\Q_{p^2},\M_2(\Q_p))$ can be described explicitly.  
One first observes that the structure of filtered convergent $F$-isocrystal of $\mathcal{E}(V)$ induces a structure of filtered Frobenius module (\cite[\S2]{IS})
on the fiber $\mathcal{E}(V)_\Psi$. On the other hand, one attaches to such a pair $(V,\Psi)$ a filtered Frobenius module $V_\Psi$ in a natural way as follows.  The underlying 
vector space of $V_\Psi$ is 
$V_{\Q_p^\unr}=V\otimes_{\Q_p}\Q_{p}^\unr$. The Frobenius is given by 
$\phi_V\otimes\sigma$, where $\phi_V=\rho_2\left(\smallmat 0p10\right)$ as 
before (thus, only depending on $\rho_2$). The filtration, only depending 
on $\rho_1$, has a more involved definition. 
Recall that any representation $V$ 
can be split into the direct sum of sub-representations $(V^{(n)},\rho^{(n)})$ 
which are pure of weight $n$, and therefore it is enough to define the filtration for a representation $\rho:\GL_2\rightarrow\GL(V)$ which is pure of weight $n$, since in the general case, the filtration $F^iV_{\Q_p^\unr}$ is by definition the direct sum of the filtrations $F^iV^{(n)}_{\Q_p^\unr}$ for all $n\in\Z$. 
If $V$ is pure of weight $n$, define $V_j$ to be the subspace of $V_{\Q_p^\unr}$ consisting of elements 
$v\in V$ satisfying the property $\rho(\psi(a))(v)=a^j\sigma(a)^{n-j}v$ for all 
$a\in\Q_{p^2}$. Define the filtration $F^iV_{\Q_p^\unr}$ of $V_{\Q_p^\unr}$ 
as the direct sum of $V_j$ for $j\geq i$. This equips $V_{\Q_p^\unr}$ 
with a structure 
of filtered Frobenius module, denoted $V_\Psi$. By \cite[Lemma 4.2]{IS}, 
$V_\Psi\simeq\mathcal{E}(V)_\Psi$ as filtered Frobenius modules. 

To stress the dependence on $\Psi$, 
we denote $F^\bullet V_\Psi$ the filtration of the filtered Frobenius module 
$V_\Psi$; this is then a filtration on $V_{\Q_p^\unr}$ which depends on $\Psi$. 
Let $\mathrm{gr}^i(F^\bullet V_\Psi)=F^iV_\Psi/F^{i-1}V_\Psi$ be the graded pieces of the filtration. If $V$ is pure of weight $n$, 
we have a canonical 
isomorphism $\mathrm{gr}^i(F^\bullet V_\Psi)\simeq V_i$ as well as a decomposition 
$V_{\Q_p^\unr}=\bigoplus_{i\in\Z}\mathrm{gr}^i\left(F^\bullet V_\Psi\right)$. 

For $\Psi\in\Hom(\Q_{p^2},\M_2(\Q_p))$, denote $\bar\Psi$ the 
morphism of $\Q_p$-algebras obtained by composition $\Psi$ with the main 
involution of $\M_2(\Q_p)$; therefore, if $\Psi(x)=\smallmat abcd$ then $\bar\Psi(x)=\smallmat d{-b}{-c}a$. 
If $V$ is pure of weight $n$, then the graduate pieces 
$\mathrm{gr}^i(F^\bullet V_\Psi)$ and 
$\mathrm{gr}^{n-i}(F^\bullet V_{\bar\Psi})$ are equal, for all $i\in\Z$. 
In particular, for $V=\M_2$ we have 
\begin{equation}\label{graded pieces}
\mathrm{gr}^1\left(F^\bullet (\M_2)_\Psi\right)\simeq
\mathrm{gr}^2\left(F^\bullet (\M_2)_{\bar{\Psi}}\right).
\end{equation}
and therefore there is an exact sequence:
\[
0\longrightarrow\mathrm{gr}^1\left(F^\bullet (\M_2)_\Psi\right)\longrightarrow
(\M_2)_\Psi\longrightarrow \mathrm{gr}^1\left(F^\bullet (\M_2)_{\bar{\Psi}}\right)\longrightarrow0\]and a canonical decomposition 
\begin{equation}\label{splitting}(\M_2)_{\Psi}\simeq\mathrm{gr}^1\left(F^\bullet (\M_2)_\Psi\right)\bigoplus
\mathrm{gr}^1\left(F^\bullet (\M_2)_{\bar{\Psi}}\right).\end{equation}
One can choose generators $\omega_1,\omega_2$ 
of the $\hat{\Q}_{p}^\unr$-vector space 
$\mathrm{gr}^1\left(F^\bullet (\M_2)_\Psi\right)$ so that 
$\omega_1$ and $\omega_2$ are defined over $\Q_{p^2}$. 
Then $\overline{\omega}_1$ and $\overline\omega_2$ are generators 
of the $\hat{\Q}_p^\unr$-vector space 
$\mathrm{gr}^1\left(F^\bullet (\M_2)_{\bar\Psi}\right)$, where $\omega_i\mapsto\overline\omega_i$ for $i=1,2$ 
denotes the action of $\Gal(\Q_{p^2}/\Q_p)$ on $\omega_i$. 
If therefore follows that the Hodge splitting coincides on quadratic points with the projection $(\M_2)_{\Psi}\rightarrow\mathrm{gr}^1\left(F^\bullet (\M_2)_\Psi\right)$ to the first factor in the decomposition \eqref{splitting}.

We now apply the above results to the situation of the previous sections. Recall that $K$ is a imaginary quadratic field and $f\in H^0(\mathcal{C}_\Q,\underline{\omega}_\Q^{\otimes{k}})$ is an algebraic modular form of weight $k$ and level $N^+N^-p$ with 
$p\nmid N=N^+N^-$ and $(N^+,N^-)=1$. We write $f_\infty:\mathcal{H}_p\rightarrow\C$ 
and $f_p:\mathcal{H}_p\rightarrow\C_p$ for the holomorphic and the rigid analytic 
modular forms corresponding to $f$, respectively. 
Assume $N^-p$ is a product of an 
even number of distinct primes, each of them inert in $K$, and 
that all primes dividing $N^+$ are split in $K$. 
Let $P\in\mathcal{C}_\Q(K)$ be a Heegner point, and assume that 
$P\in\mathcal{C}_{\C}(\C)$ represented by the point $\tau_\infty\in\mathcal{H}_\infty$ modulo $\Gamma_\infty$, and $P\in\mathcal{C}_{\C_p}(\C_p)$is represented by the point $\tau_p\in\mathcal{H}_p$ modulo $\Gamma_p$. 
Fix embeddings $\bar\Q\hookrightarrow\bar\Q_p$ and $\bar\Q\hookrightarrow\C_p$ which allows us to view algebraic numbers as complex and 
$p$-adic numbers.  
\begin{theorem}\label{comparison theorem} For any positive integer $j$ 
we have the equality 
\[\Theta^j_{\infty,k}(f_\infty)(\tau_\infty)=\Theta^j_{p,k}(f_p)(\tau_p).\]
\end{theorem} 

\begin{proof}
We mimic a well known argument of Katz when $p$ is split in $K$ 
(\cite[Theorems 2.4.5, 2.4.7]{KatzCM}; see also \cite[Proposition 1.12]{BDP}, \cite[Theorem 3.5]{Brooks}, \cite[Proposition 2.12]{Mori}). 
Let $A_P$ be the 
false elliptic curve corresponding to the Heegner point $P$. 
The algebraic CM splitting of $A_P$ coincides 
both with the Hodge splitting and the $p$-adic setting, and therefore 
the values of $\Psi_{\infty,n}$ and $\Psi_{p,n}$ at CM points are the same. Since the 
construction of the Shimura-Maass operators is algebraic, we see that 
$\nabla^{\mathrm{r}\text{-}\mathrm{an}}_n(f_\infty)$ coincides with $\nabla^\rig_n(f_p)$, 
and the same still holds for the iterates of the Shimura-Maass operator, which also admit an algebraic construction. The result follows. 
\end{proof}

\subsection{Nearly rigid analytic modular forms} 
In this subsection we make explicit the relation between the results of this paper 
and those of Franc's thesis \cite{Franc}; it is independent from the rest of the paper. 

We first  introduce a $\C_p$-subspace of $\mathcal{C}$, 
which plays a role analogue to that of nearly holomorphic functions in 
the real analytic setting. For this part, we closely follow \cite{Franc}.  
The assignment $X\mapsto 1/( z-z^\ast)$
defines an injective 
homomorphism $\mathcal{A}[X]\hookrightarrow\mathcal{C}$ 
(\cite[Proposition 4.3.3]{Franc}).  
Define the $\mathcal{A}$-algebra $\mathcal{N}$ of \emph{nearly rigid analytic functions} on 
to be the image of this map (\emph{cf.} \cite[Definition 4.3.5]{Franc}). 
By definition, $\mathcal{N}$ is a sub-$\mathcal{A}$-algebra 
of $\mathcal{A}^*$.  The $\mathcal{A}$-algebra $\mathcal{N}$ is equipped with 
a canonical graduation $\mathcal{N}=\bigoplus_{j\geq 0}\mathcal{N}^{(j)}$ 
where for each integer $j\geq0$, we denote $\mathcal{N}^{(j)}$ the sub-$\mathcal{A}$-algebra 
of 
$\mathcal{N}$ consisting of functions $f$ which can be written 
in the form 
\[f( z)=\sum_{i=0}^j\frac{f_i( z)}{(z-\sigma(z))^i}\] with $f_i\in\mathcal{A}$.
The Shimura-Maass operator $\delta_{p,k}$ restricts to an operator (denoted with 
the same symbol) $\delta_{p,k}:\mathcal{N}\rightarrow\mathcal{N}$ 
which takes $\mathcal{N}^{(j)}$ to $\mathcal{N}^{(j+2)}$. 

Define now $\mathcal{N}_k(\Gamma_p)=\mathcal{N}_{k}^{\Gamma_p}$ 
to be the $\C_p$-subalgebra of $\mathcal{N}$ 
consisting of functions which are invariant under the weight $k$ action of $\Gamma_p$ 
on $\mathcal{N}$, namely, those functions satisfying the transformation property 
$f(\gamma z)=(c z+d)^kf(z)$ for all 
$ z\in\hat{\Q}_p^\unr-\Q_p$ and 
$\gamma\in\Gamma_p$. Note that $S_k^\rig(\Gamma_p)\subseteq\mathcal{N}_{k}({\Gamma_p})$. We call $\mathcal{N}_{k}({\Gamma_p})$ the $\C_p$-vector space of \emph{nearly rigid analytic modular forms} of weight $k$ and level $\Gamma_p$. Define also $\mathcal{N}_{k}^{(j)}(\Gamma_p)=\mathcal{N}_{k}({\Gamma_p})\cap\mathcal{N}^{(j)}$.
The operator $\delta_{p,k}^j$ introduced in \S \ref{section differential operators} 
restricts to 
a map $\delta_{p,k}:\mathcal{N}_{k}({\Gamma_p})\rightarrow\mathcal{N}_{k+2j}({\Gamma_p})$ (\cite[Lemma 4.3.8]{Franc}). 
By \cite[Theorem 4.3.11]{Franc}, 
for each integer $r\geq 0$  
we have an isomorphism of $\C_p$-vector spaces 
\[\bigoplus_{j=0}^rS^\rig_{k+2(r-j)}(\Gamma_p)\simeq\mathcal{N}_{p,k+2r}^{(r)}(\Gamma_p)\]
which maps $(h_j)_{j=0}^{k+2(r-j)}$ to $\sum_{j=0}^{k+2(r-j)}\delta_{p,k}^j(h_j)$. 

\begin{corollary}[Franc]\label{franc result}
Let $\tau_p\in\mathcal{H}_p$ corresponds to a Heegner point. 
The values 
$\frac{\delta^j_{p,k}(f)(\tau_p)}{t_p^j}$ are algebraic for each integer $j\geq 0$.
\end{corollary} 

\begin{proof}
The result is clear from 
Theorem \ref{comparison theorem} since this is known for  
$\Theta_{\infty,k}^j(f)(\tau_\infty)$. 
\end{proof}

\begin{remark} Equation \eqref{Shimura-Maass explicit formula iterates}
answers affirmatively one of the questions 
left in \cite[\S 6.1]{Franc} whether if it was possible to describe the $p$-adic Shimura-Maass operator $\delta^j_{p,k}$ introduced in \cite{Franc} in 
a more conceptual way, similar to that in the complex case. 
Corollary \ref{franc result} is the main result of \cite{Franc}, which was obtained 
via a completely different method, following more closely 
the complex analytic approach of Shimura.
\end{remark}

\section{The Coleman primitive} \label{section Coleman primitive}
Write
$\nabla=\nabla_{\mathcal{G},n}^\ast$, $\nabla^{1,0}=\nabla^{1,0}_\mathcal{G}$, 
$\nabla^{0,1}=\nabla^{0,1}_\mathcal{G}$ and 
$\langle,\rangle=\langle,\rangle_{\mathcal{G},n}^\ast$ to simplify the notation. 
For any $n$ and any $j$, whenever there is not possible confusion,  
we write  
$\Theta_p=\Theta_{p,n}^\ast$ and   
$\Theta_{p,n}^{\ast,j}=\Theta_p^j$ for the $p$-adic Shimura-Maass operator, and 
$\Psi_p=\Psi_{p,n}^\ast$ for the splitting of the Hodge filtration.  

We set up the notation $\omega_\can=d\tau$ and 
$\eta_\can=\frac{d\tau^\ast}{z^\ast-{z}}$. 
Since $\langle dx,dy\rangle=-1$, we have 
$\langle\omega_\can,\eta_\can\rangle=1$.  
We also write $\omega_\can^j\eta_\can^{n-j}=\omega_\can^{\otimes{j}}\otimes\eta_\can^{\otimes{n-j}}$.

The computation of the Gauss-Manin connection gives
\[\nabla(\omega_\mathrm{can})=\left(-\frac{\omega_\can}{{z}^\ast-z}+\eta_\mathrm{can}\right)\otimes dz,\]
\[\nabla(\eta_\mathrm{can})=-\frac{\omega_\can \otimes dz^\ast}{(z^\ast-z)^2}+\frac{\eta_\can\otimes dz}{z^\ast-z}.\] 

Let $f:\mathcal{H}_p\rightarrow\C_p$ be a rigid modular form 
giving rise to a 
section $\omega_f=f(z)\otimes d\tau^{\otimes{k}}.$ Put $n=k-2$. 
Using the Kodaira-Spencer map, we identify this with 
$\omega_f=f(z)dz\otimes d\tau^{\otimes n}.$
Let $F_f$ be the Coleman primitive of the differential form $\omega_f$, satisfying the differential equation 
\[\nabla(F_f)=\omega_f.\] 
Define for $j=n/2,\dots, n$ an integer 
\begin{equation}\label{def G_j}
G_j(z)=\langle F_{f}(z),\omega_\mathrm{can}^j\eta_\mathrm{can}^{n-j}\rangle\otimes\omega_\can^{n-2j}.\end{equation}

\begin{theorem}\label{lemma8.1}
$\Theta^{j+1}_p(G_j)=j!\omega_f$. 
\end{theorem}
\begin{proof}  This result, which is proved by means of a simple and 
explicit computation, is the analogue of 
\cite[Proposition 3.24]{BDP} (and also of \cite[Theorem 7.3]{Brooks}), but we provide a complete proof since our formalism is quite different from that in \cite{BDP}, where one can use the Tate curve and the $q$-expansion principle. 
As in \emph{loc. cit.} we show that $\Theta_pG_0(z))=\omega_f$ 
and $\Theta_p(G_j(z))=jG_{j-1}(z)$.

We first compute $\nabla(G_0(z))$. 
We have: 
\[\begin{split} 
\nabla(G_0(z))&=
\nabla\left(\langle F_{f}(z),\eta_\mathrm{can}^{n}
\rangle\otimes\omega_\can^n\right)
\\
&=
\left\langle \nabla(F_{f}(z)),\eta_\mathrm{can}^{n}\right\rangle\otimes \omega_\can^n+
\langle F_{f}(z),\nabla\left(\eta_\mathrm{can}^{n}\right)\rangle\otimes \omega_\can^n +
\left\langle F_f(z),\eta_\mathrm{can}^{n}\right\rangle\otimes\nabla(\omega_\can^n)
\\
&=\left\langle f(z)dz\otimes\omega_\can^n,\eta_\mathrm{can}^{n}\right\rangle\otimes\omega_\can^n
+\langle F_{f}(z),\nabla\left(\eta_\mathrm{can}^{n}\right)\rangle\otimes \omega_\can^n +
\left\langle F_f(z),\eta_\mathrm{can}^{n}\right\rangle\otimes\nabla(\omega_\can^n).
\end{split}\]
We now compute the last two pieces: 
\[\begin{split}
\langle F_{f}(z),\nabla\left(\eta_\mathrm{can}^{n}\right)\rangle\otimes\omega_\can^n&=
\left\langle F_{f}(z),n\eta^{n-1}_\can\nabla(\eta_\can)\right\rangle\otimes\omega_\can^n\\
&=\left\langle F_{f}(z),n\eta^{n-1}_\can\left(\frac{-\omega_\can \otimes dz^\ast}{({z}^\ast-z)^2}+\frac{\eta_\can\otimes dz}{{z}^\ast-z}\right)\right\rangle\otimes\omega_\can^n
\\
&=-\left\langle F_{f}(z),\frac{n\eta^{n-1}_\can\omega_\can \otimes dz^\ast}{({z}^\ast-z)^2}\right\rangle\otimes\omega_\can^n
+\left\langle F_{f}(z),\frac{n\eta^{n}_\can\otimes dz}{{z}^\ast-z}\right\rangle\otimes\omega_\can^n\\
&=-\left\langle F_{f}(z),{\eta^{n-1}_\can\omega_\can }\right\rangle\otimes \frac{n \omega_\can^n\otimes dz^\ast}{({z}^\ast-z)^2}
+\left\langle F_{f}(z),{\eta^{n}_\can}\right\rangle\otimes \frac{n \omega_\can^n\otimes dz}{{z}^\ast-z}
\end{split}\]
and
\[\begin{split}\left\langle F_f(z),\eta_\mathrm{can}^{n}\right\rangle\otimes\nabla(\omega_\can^n)&=\left\langle F_f(z),\eta_\mathrm{can}^{n}\right\rangle\otimes\nabla(\omega_\can^n)\\
&=\left\langle F_f(z),\eta_\mathrm{can}^{n}\right\rangle\otimes
n\omega_\can^{n-1}\nabla(\omega_\can)\\
&=\left\langle F_f(z),\eta_\mathrm{can}^{n}\right\rangle\otimes
n\omega_\can^{n-1}\left(-\frac{\omega_\can}{z^\ast-z}+\eta_\mathrm{can}\right)\otimes dz 
\\
&=
-\left\langle F_f(z),\eta_\mathrm{can}^{n}\right\rangle\otimes
\frac{n\omega_\can^{n}\otimes dz }{z^\ast-z}+
\left\langle F_f(z),\eta_\mathrm{can}^{n}\right\rangle\otimes
n\omega_\can^{n-1}\eta_\mathrm{can}\otimes dz . 
\end{split}\]
Therefore the sum of these two pieces gives: 
\[-\left\langle F_{f}(z),{\eta^{n-1}_\can\omega_\can }\right\rangle\otimes \frac{n \omega_\can^n\otimes dz^\ast}{(z-z^\ast)^2}+
\left\langle F_f(z),\eta_\mathrm{can}^{n}\right\rangle\otimes
n\omega_\can^{n-1}\eta_\mathrm{can}\otimes dz.\]
Recall now that $\Psi(\eta_\can)=0$ and $\Psi(dz^\ast)=0$. 
Therefore, using the Kodaira-Spencer map to replace $dz$ with $\omega_\can^2$, 
and applying $\Psi$  
we have
\[\Theta_p(G_0(z))=\omega_f\langle\omega_\can^n,\eta_\can^n\rangle=\omega_f.\]

We now compute  $\nabla(G_j(z))$ for $j\geq 1$.
The Gauss-Manin connection 
\[
\nabla(G_j(z))=
\nabla\left(\langle F_{f}(z),\omega_\mathrm{can}^j\eta_\mathrm{can}^{n-j}\rangle\otimes\omega_\can^{n-2j}\right) \]
is the sum of three terms \begin{equation}\label{Gj1}
\langle \nabla(F_{f}(z)),\omega_\mathrm{can}^j\eta_\mathrm{can}^{n-j}\rangle\otimes\omega_\can^{n-2j}+
\langle F_{f}(z),\nabla(\omega_\mathrm{can}^j\eta_\mathrm{can}^{n-j})\rangle\otimes\omega_\can^{n-2j} +
\langle F_{f}(z),\omega_\mathrm{can}^j\eta_\mathrm{can}^{n-j}\rangle\otimes\nabla(\omega_\can^{n-2j})
\end{equation}
which we calculate separately as before. 
First, since $j>0$, we have 
\[\langle \nabla(F_{f}(z)),\omega_\mathrm{can}^j\eta_\mathrm{can}^{n-j}\rangle\otimes\omega_\can^{n-2j}=\langle f(z)dz\otimes\omega_\can^n,\omega_\mathrm{can}^j\eta_\mathrm{can}^{n-j}\rangle=0.\]
Next, a simple computation shows that 
\[\nabla(\omega_\mathrm{can}^j\eta_\mathrm{can}^{n-j})=(n-2j)\omega_\can^{j}\eta_\mathrm{can}^{n-j}\otimes \frac{dz}{{{z}^\ast-z}}+j\omega_\can^{j-1}\eta_\mathrm{can}^{n-j+1}\otimes dz
-(n-j)\omega_\mathrm{can}^{j+1}\eta_\mathrm{can}^{n-j-1}
\otimes \frac{dz^\ast}{({z}^\ast-z)^2}\]
and therefore the second summand in \eqref{Gj1} is 
\[\begin{split}\langle F_{f}(z),\nabla(\omega_\mathrm{can}^j\eta_\mathrm{can}^{n-j})\rangle\otimes\omega_\can^{n-2j}=&
(n-2j)\left\langle F_{f}(z),\omega_\can^{j}\eta_\mathrm{can}^{n-j}
\right\rangle\otimes\omega_\can^{n-2j}\otimes \frac{dz}{{z-{z}^\ast}}-\\
&j\left\langle F_{f}(z),\omega_\can^{j-1}\eta_\mathrm{can}^{n-j+1}\right\rangle\otimes\omega_\can^{n-2j}\otimes dz+\\
&-(n-j)\left\langle F_{f}(z),\omega_\mathrm{can}^{j+1}\eta_\mathrm{can}^{n-j-1}
\right\rangle\otimes\omega_\can^{n-2j} \otimes\frac{dz^\ast}{(z-z^\ast)^2}\end{split}
\]
Finally, the third summand is 
\[\begin{split}
\langle F_{f}(z),\omega_\mathrm{can}^j\eta_\mathrm{can}^{n-j}\rangle\otimes\nabla(\omega_\can^{n-2j})=&
(n-2j)\langle F_{f}(z),\omega_\mathrm{can}^j\eta_\mathrm{can}^{n-j}\rangle\otimes\omega_\can^{n-2j-1}\nabla(\omega_\can)\\
=&(n-2j)\langle F_{f}(z),\omega_\mathrm{can}^j\eta_\mathrm{can}^{n-j}\rangle\otimes\omega_\can^{n-2j-1}\left(-\frac{\omega_\can}{z-{z}^\ast}+\eta_\mathrm{can}\right)\otimes dz\\
=&-(n-2j)\langle F_{f}(z),\omega_\mathrm{can}^j\eta_\mathrm{can}^{n-j}\rangle\otimes\omega_\can^{n-2j}\otimes \frac{dz}{{z-{z}^\ast}}+\\
&+(n-2j)\langle F_{f}(z),\omega_\mathrm{can}^j\eta_\mathrm{can}^{n-j}\rangle\otimes\omega_\can^{n-2j-1}\eta_\can\otimes dz.
\end{split}\]
Summing up the pieces in \eqref{Gj1},  
using the Kodaira-Spencer map to replace $dz$ with $\omega_\can^2$, 
and applying the splitting of the Hodge filtration $\Psi$ which kills $\eta_\can$ and $dz^\ast$, 
we have 
\[\Theta_p(G_j(z))=j\langle F_{f}(z),\omega_\can^{j-1}\eta_\mathrm{can}^{n-(j-1)}\rangle\otimes\omega_\can^{n-2(j-1)}=jG_{j-1}(z).\]
The result follows.\end{proof}

\section{The generalised Kuga-Sato motive}\label{section: motive}
Fix an even integer $k\geq 2$ 
and put $n=k-2$, $m=n/2$. 
Let $A_0$ be a false elliptic curve with quaternionic multiplication and full level-$M$ structure, defined over $H$ (the Hilbert class field of $K$) and with complex multiplication by $\mathcal{O}_K$; the action of $\mathcal{O}_K$ is required to commute with the quaternionic action, and this implies that $A_0$ is isogenous to $E\times E$ for an elliptic curve $E$ with CM by $\mathcal{O}_K$.  Fix a field $F\supset H$ and consider the $(2n+1)$-dimensional variety $X_m$ over $F$ given by
\[X_m:=\mathcal{A}^m\times A_0^m.\]Here and in the following we simplify the notation and simply write $\mathcal{A}$, $\mathcal{C}$  and $A_0$ for $\mathcal{A}_F$, $\mathcal{C}_F$ and $(A_0)_F$, unless we need to stress the field of definition in which case we keep the full notation. 
The variety $X_m$ is equipped with a proper morphism $\pi\colon X_m\rightarrow \mathcal{C}$ with $2n$-dimensional fibers. The fibers above points of $\mathcal{C}$ are products of the form $A^m\times A^m_0$.

The de Rham cohomology of $\mathcal{C}$ attached to $\mathcal{L}_n$, denoted $H^1_\mathrm{dR}(\mathcal{C},\mathcal{L}_n,\nabla)$, is defined to be the $1$-st hypercohomology of the complex
$$0\longrightarrow\mathcal{L}_n\overset{\nabla}{\longrightarrow}\mathcal{L}_n\otimes \Omega^1_{\mathcal{C}}\longrightarrow 0.$$
As shown in \cite[Corollary 6.3]{Brooks}, one can define a projector 
$\epsilon_{\mathcal{A}}$ (denoted $P$ in \emph{loc. cit.}) 
in the ring of correspondences $\mathrm{Corr}_{\mathcal{C}}(\mathcal{A}^m,\mathcal{A}^m)$, such that 
\begin{equation}\label{filtration A bis}\epsilon_\mathcal{A}H^*_\mathrm{dR}(\mathcal{A}_m/F)\subseteq H^{n+1}_\mathrm{dR}(\mathcal{A}_m/F),\end{equation}
\begin{equation}\label{filtration A}\epsilon_\mathcal{A}H^*_\mathrm{dR}(\mathcal{A}_m/F)\simeq H^1_\mathrm{dR}(\mathcal{C},\mathcal{L}_n,\nabla).\end{equation} 
On the other hand, by \cite[Proposition 6.4]{Brooks}, we can define a projector $\epsilon_A\in \mathrm{Corr}(A^m_0,A^m_0)$
(which is defined by means of $\epsilon_\mathcal{A}$)
such that 
\begin{equation}\label{filtration A_0}
\epsilon_{A_0} H_\mathrm{dR}^*(A_0^m/F)=\mathrm{Sym}^neH_\mathrm{dR}^1(A_0/F).\end{equation}
The projectors $\epsilon_\mathcal{A}$ and $\epsilon_A$ are commuting idempotents when viewed in the ring $\mathrm{Corr}_\mathcal{C}(X_m,X_m)$. We define $\epsilon=\epsilon_{\mathcal{A}}\epsilon_{A_0}$ and denote $\mathcal{D}$ the motive 
$(X_m,\epsilon)$. 
By \cite[Proposition 6.5]{Brooks} and \eqref{filtration A bis}, \eqref{filtration A}, \eqref{filtration A_0} we see that
\begin{equation}\label{de Rham realization}
\epsilon H^i_\mathrm{dR}(X_m/F)=\begin{cases}H^1_\mathrm{dR}(\mathcal{C},\mathcal{L}_n,\nabla)\otimes \mathrm{Sym}^neH^1_\mathrm{dR}(A_0^m/F), \text{if $i=2n+1$},\\
0, \text{if $i\neq2n+1$}.
\end{cases}\end{equation}
Thus the Hodge filtration of the de Rham cohomology is given by 
\begin{equation}\label{equation Hodge filtration}
F^{n+1}\left(\epsilon H^{2n+1}_\mathrm{dR}(X_m/F)\right)=
F^{n+1}\left(H^1_\mathrm{dR}(\mathcal{C},\mathcal{L}_n,\nabla)\right)\otimes \mathrm{Sym}^neH^1_\mathrm{dR}(A_0^m/F).
\end{equation}
Finally we have a map (\cite[page 4221]{Brooks}) 
\begin{equation}\label{modform} 
H^0\left(\mathcal{C},\underline\omega^{\otimes{n+2}}\right)\longrightarrow F^{n+1}\left(H^1_\mathrm{dR}(\mathcal{C},\mathcal{L}_n,\nabla)\right).
\end{equation}

\section{The Abel-Jacobi map}

Let $[\Delta]$ be the class a null-homologous codimension-$(n+1)$ cycle $\Delta$ in $\CH^{n+1}(\mathcal{D})(F)$, where $F$ is as in Section \ref{section: generalized Heegner cycles} a field containing the Hilbert class field of $K$. One may associate to $[\Delta]$ the isomorphism class of the extension 
$$0\longrightarrow \epsilon H^{2n+1}_{\text{\'et}}\left(\overline{X}_m,\Q_p(n+1)\right)\longrightarrow E\longrightarrow \Q_p\longrightarrow 0$$
in
$$\mathrm{Ext}_{G_F}^1\left(\Q_p,\epsilon H^{2n+1}_{\text{\'et}}(\overline{X}_m,\Q_p(n+1))\right)$$
(where $\mathrm{Ext}_{G_F}^1$ denotes the first $\mathrm{Ext}$ group 
in the category of $G_F=\Gal(\bar{F}/F)$-modules) 
given by the pull-back of
\begin{multline*}
0\longrightarrow \epsilon H^{2n+1}_{\text{\'et}}\left(\overline{X}_m,\Q_p(n+1)\right)\longrightarrow \epsilon H^{2n+1}_{\text{\'et}}\left(\overline{X}_m-|\overline{\Delta}|,\Q_p(n+1)\right)\longrightarrow\\
\longrightarrow\mathrm{Ker}\left(\epsilon H^{2n+2}_{{\text{\'et}}|\overline{\Delta}|}\left(\overline{X}_m,\Q_p(n+1)\right)\rightarrow \epsilon H^{2n+2}_{\text{\'et}}\left(\overline{X}_m,\Q_p(n+1)\right)\right)\longrightarrow 0
\end{multline*}
via the map $\Q_p\rightarrow\epsilon H^{2n+2}_{{\text{\'et}}|\overline{\Delta}|}(\overline{X}_m,\Q_p(n+1))$ sending $1$ to the cycle class $\mathrm{cl}^{\overline{X}_m}_{|\overline{\Delta}|}(\overline{\Delta})$ of $\Delta$ in $H^{2n+2}_{|\overline{\Delta}|}(\overline{X}_m,\Q_p(n+1))$. This association defines a map,  
called \emph{$p$-adic \'etale Abel-Jacobi map}
$$\AJ_p\colon \CH^{n+1}(\mathcal{D})(F)\longrightarrow \mathrm{Ext}_{G_F}^1\left(\Q_p,\epsilon H^{2n+1}_{\text{\'et}}(\overline{X}_m,\Q_p(n+1))\right).$$

Let $v$ be the place of $F$ above $p$ induced by the inclusion $F\subseteq\overline{\Q}\hookrightarrow\C_p$, which for simplicity we assume to be unramified over $p$. We now describe the restriction of $\AJ_p$ to $\mathrm{CH}^{n+1}(\mathcal{D})(F_v)$. Consider the base change of $X_m$ and $\mathcal{C}$ to $F_v$ that we still denote by $X_m$ and $\mathcal{C}$ in this section. 
Since the motive $X_m$ has semistable reduction at $v$,  
the image of the Abel-Jacobi map is contained in the first $\mathrm{Ext}$ group in the category of semistable representations; using \cite[Lemma 2.1]{IS}, and following the argument in \cite[page 362]{IS} (see also \cite[\S 4.2]{LP2}) the Abel-Jacobi map gives a map, denoted with the same symbol by a slight abuse of notation, 
\[\AJ_p\colon \CH^{n+1}(\mathcal{D})(F_v)\longrightarrow
\frac{D_{\mathrm{st},F_v}\left(\epsilon H^{2n+1}_{\text{\'et}}(\overline{X}_m,\Q_p(n+1))\right)}
{F^{n+1}\left(D_{\mathrm{st},F_v}(\epsilon H^{2n+1}_{\text{\'et}}(\overline{X}_m,\Q_p(n+1)))\right)}
\] where 
$D_{\mathrm{st},F_v}$ is the Fontaine's semistable functor from the category of $G_{F_v}=\Gal(\bar{F}_v/F_v)$-representations to the category $M^{\phi,N}_{F_v}$ of filtered Frobenius monodromy modules over $F_v$, and we denote as usual by 
$F^{i}\left(D\right)$ the $i$-step filtration of  a filtered Frobenius monodromy module $D$. 

 By \cite{Tsuji1}, \cite{FaltingsAEE} (see also \cite{Tsuji2}) we know that  
$D_{\mathrm{st},F_v}\left(H^{2n+1}_\text{\'et}(X_m,\Q_p)\right)$ is isomorphic to the de Rham cohomology group  
$H^{2n+1}_\mathrm{dR}(X_m/F_v)$ as filtered Frobenius monodromy modules. Therefore, applying the idempotent
$\epsilon$, we obtain the isomorphism 
\begin{equation}\label{equation39}
\frac{D_{\mathrm{st},F_v}\left(\epsilon H^{2n+1}_{\text{\'et}}(\overline{X}_m,\Q_p(n+1))\right)}
{F^{n+1}\left(D_{\mathrm{st},F_v}(\epsilon H^{2n+1}_{\text{\'et}}(\overline{X}_m,\Q_p(n+1)))\right)}
\simeq \frac{\epsilon H^{2n+1}_\mathrm{dR}(X_m/F_v)(n+1)}{F^{n+1}\left(\epsilon H^{2n+1}_\mathrm{dR}(X_m/F_v)(n+1)\right)}.\end{equation}


By Poincar\'e duality, 
\[\frac{\epsilon H^{2n+1}_\mathrm{dR}(X_m/F_v)(n+1)}{F^{n+1}\left(\epsilon H^{2n+1}_\mathrm{dR}(X_m/F_v)(n+1)\right)}
\simeq
({F^{n+1}\left(\epsilon H^{2n+1}_\mathrm{dR}(X_m/F_v)(n+1)\right)})^\vee\]
where $V^\vee$ denotes dual of $F_v$-vector spaces. 
Combining 
\eqref{equation Hodge filtration} and dualizing 
\eqref{modform} we obtain a map 
\[({F^{n+1}\left(\epsilon H^{2n+1}_\mathrm{dR}(X_m/F_v)(n+1)\right)})^\vee \longrightarrow 
\left(M_k(\mathcal{C},F_v)\otimes \mathrm{Sym}^neH^1_\mathrm{dR}(A_0^m/F)\right)^\vee.\]
%
%
%
%
The $p$-adic Abel-Jacobi map for the nullhomologous $(n+1)$-th Chow cycles of the motive $\mathcal{D}$ can thus be viewed as a map, denoted with the same symbol by an abuse of notation:
\[\AJ_p\colon \mathrm{CH}_0^{n+1}(\mathcal{D})(F_v)\longrightarrow \left(M_{k_0}(X,F_v)\otimes \mathrm{Sym}^n eH_\mathrm{dR}^1(A/F_v)\right)^\vee.\]

\section{Generalized Heegner cycles}\label{section: generalized Heegner cycles} 
\subsection{Definition}
Let $\varphi\colon A_0\rightarrow A$ be an isogeny (defined over $\bar{K}$) of false elliptic curves, of degree prime to $N^+$, \emph{i.e.} whose kernel intersects the level structures of $A_0$ trivially. Let $P_{A}$ be the point on $\mathcal{C}$ corresponding to $A$ with level structure given by composing $\varphi$ with the level structure of $A_0$. We associate to any pair $(\varphi,A)$ a codimension $n+1$ cycle $\Upsilon_\varphi$ on $X_m$ by defining
$$\Upsilon_\varphi:=(\Gamma_\varphi)^m\subset (A\times A_0)^m$$
where $\Gamma_\varphi=\{(\varphi(x),x):x\in A_0\}\subset A\times A_0$ is the graph of $\varphi$. We then set
$$\Delta_\varphi:=\epsilon\Upsilon_\varphi.$$
The cycle $\Delta_\varphi$ of $\mathcal{D}$ is supported on the fiber above $P_{A}$ and has codimension $n+1$ in $\mathcal{A}^m\times A^m_0$, thus $\Delta_\varphi\in \mathrm{CH}^{n+1}(\mathcal{D})$. By \eqref{de Rham realization}, the cycle $\Delta_\varphi$ is homologous to zero.

We now compute the image of $\Delta_\varphi$ under the Abel-Jacobi map. 
The de Rham cohomology group 
$H^i_{\mathrm{dR}}(A/F)$ of a false elliptic curve $A$ defined over a field $F$ is equipped with the Poincar\'e pairing
 $\langle,\rangle_{H^i_\mathrm{dR}(A/F)}$, 
which we simply denote $\langle,\rangle_{A}$.  
Fix a nonvanishing differential $\omega_{A_0}$ in $e\Omega^1_{A_0/F}$. This fixed differential 
determines a class $\eta_{A_0}\in eH^1(A_0,\mathcal{O}_{A_0})$ dual to $\omega_{A_0}$ under the Poincar\'e duality pairing $\langle\ ,\ \rangle_{A_0}$, normalised so that 
$\langle\omega_{A_0},\eta_{A_0}\rangle_{A_0}=1$. 
We can view $\{\omega_{A_0},\eta_{A_0}\}$ as a basis of $eH^1_\mathrm{dR}(A_0/F)$ since the Hodge exact sequence
$$0\longrightarrow \Omega^1_{A_0/F}\longrightarrow H^1_\mathrm{dR}(A_0/F)\longrightarrow H^1(A_0,\mathcal{O}_{A_0})\longrightarrow 0$$
splits, because $A_0$ has $\mathrm{CM}$. This yields a basis for $\mathrm{Sym}^n eH^1_\mathrm{dR}(A_0/F)$ given by 
the elements $\omega^{j}_{A_0}\otimes\eta^{n-j}_{A_0}$ for $j$ an integer such that $0\leq j\leq n$. 

Let $\omega_f$ be the global section of the sheaf $\underline{\omega}^n\otimes\Omega^1_{\mathcal{C}}$
associated to the modular form over the Shimura curve $\mathcal{C}$ which corresponds to $f$ under the Jacquet-Langlands correspondence. The aim of this section is to compute 
\[\AJ_p(\Delta_\varphi)(\omega_f\otimes\omega^j\eta^{n-j})\] for
$j=0,\dots,n$, following \cite{BDP}, \cite{Brooks} and \cite{IS}.  

Define 
\[\mathcal{L}_{n,n}=\mathcal{L}_n\otimes\mathrm{Sym}^n eH_\mathrm{dR}^1(A_0/F).\]
The Gauss-Manin connection on $\mathcal{L}_n$ combined with the trivial connection on $H^1_\mathrm{dR}(A_0/F)$, gives rise to the connection 
\[\nabla\colon \mathcal{L}_{n,n}\longrightarrow\mathcal{L}_{n,n}\otimes\Omega^1_{\mathcal{C}}.\] 
The de Rham cohomology groups attached to $(\mathcal{L}_{n,n},\nabla)$ are defined to be the hypercohomology of the complex
$$0\longrightarrow\mathcal{L}_{n,n}\overset{\nabla}{\longrightarrow}\mathcal{L}_{n,n}\otimes \Omega^1_{\mathcal{C}}\longrightarrow 0.$$
We have 
\[H^1_\mathrm{dR}(\mathcal{C},\mathcal{L}_{n,n},\nabla)=H^1_\mathrm{dR}(\mathcal{C},\mathcal{L}_n,\nabla)\otimes\mathrm{Sym}^n \epsilon H^1_\mathrm{dR}(A_0/F).\]

Let $z_A$ be the point in $\mathcal{C}$ corresponding to $A$, 
and denote $X_{m,z_A}$ the fiber of $X_m$ over $z_A$. 
Recall the cycle class map 
\[\mathrm{cl}_{z_A}:\mathrm{CH}^n(X_{m,z_A})_\Q(F)\longrightarrow 
H^{2n}_\text{\'et}(\overline{X}_{m,z_A},\Q_p)(n).\]
Since $\Delta_\varphi=\epsilon\Upsilon_\varphi$, the 
image $\mathrm{cl}_{z_A}(\Delta_\varphi)$ of $\Delta_\varphi$ 
under the cycle class map $\mathrm{cl}_{z_A}$ 
belongs to $\epsilon H^{2n}_\text{\'et}(\overline{X}_{m,z_A},\Q_p)(n)$. 
Since $D_{\mathrm{st},F_v}\left(H^{2n}_\text{\'et}(X_m,\Q_p)\right)$ is isomorphic to the de Rham cohomology group  
$H^{2n}_\mathrm{dR}(X_m/F_v)$ as filtered Frobenius monodromy modules, we still denote $\mathrm{cl}_{z_A}(\Delta_\varphi)$, with a slight abuse of notation, the image 
of $\mathrm{cl}_{z_A}(\Delta_\varphi)$ under the functor $D_{\mathrm{st},F_v}$ and the isomorphism with Rham cohomology group; so we finally end up with the element 
\[\mathrm{cl}_{z_A}(\Delta_\varphi) \in \epsilon H^{2n}_\mathrm{dR}(X_m/F_v)=
H^1_\mathrm{dR}(\mathcal{C},\mathcal{L}_n,\nabla)\otimes \mathrm{Sym}^neH^1_\mathrm{dR}(A_0^m/F).\]

\subsection{Rigid analysis on Mumford curves} 
In this subsection we work over $\hat{\Q}_p^\unr$. To simplify the notation, we suppress the symbol $\hat{\Q}_p^\unr$; thus we write 
$\mathcal{C}=\mathcal{C}_{\hat{\Q}_p^\unr}$, $\mathcal{C}^\rig=\mathcal{C}_{\hat{\Q}_p^\unr}^\rig$, 
$\mathcal{L}_n=\mathcal{L}_{n,\hat{\Q}_p^\unr}$, $\mathcal{L}_n^\rig=\mathcal{L}_{n,\hat{\Q}_p^\unr}^\rig$ and $V_n=V_n\otimes_{\Q_p}\hat{\Q}_p^\unr$. 

\subsubsection{Structure of filtered Frobenius monodromy modules}
We first derive a description of 
$H^1_\mathrm{dR}(\mathcal{C}^\rig,\mathcal{L}_{n,n}^\rig,\nabla)$ by means of $V_n$-valued differential forms on $\mathcal{H}_p^\unr$. 

\begin{lemma}\label{AJprop1} 
$eH^1_\mathrm{dR}(\mathcal{G}/\hat{\HH})\simeq\mathcal{E}(V_1)$ as filtered convergent $F$-isocrystals on $\HH_p^\unr$. 
\end{lemma}

\begin{proof}
The representation $(\M_2,\rho_1,\rho_2)$ is isomorphic to $V_1\odot V_1=(V_1\otimes V_1,\sigma_1,\sigma_2)$, where $\sigma_1(A)(R_1\otimes R_2)=(A\cdot R_1)\otimes R_2$ and $\sigma_2(A)(R_1\otimes R_2)=R_1\otimes (\overline{A}^t\cdot R_2)$.
Recall that the isomorphism $\iota_p$ satisfies $\iota_p(e)=
\smallmat 1000$. The result then follows from \eqref{IS2}.
\end{proof}

Write $\mathcal{V}_n=\mathcal{E}(V_n)$ to simplify the notation. It follows from Lemma \ref{AJprop1} that 
\[H^1_\mathrm{dR}(\mathcal{C}^\rig,\mathcal{L}_{n}^\rig,\nabla)\simeq 
H^1_\mathrm{dR}(\mathcal{C}^\rig,\mathcal{V}_n)\] as filtered Frobenius monodromy modules, 
and the right hand side is isomorphic to the $\hat{\Q}_p^\unr$-vector space of $V_n$-valued, 
$\Gamma_p$-invariant differential forms of the second kind on $\HH_p^\unr$ modulo 
forms $\omega$ such that $\nabla(\omega)=0$. 
Define 
\[\mathcal{V}_{n,n}=\mathcal{V}_n\otimes\mathrm{Sym}^neH^1_\mathrm{dR}(A_0^m).\] 
Then we have 
\[H^1_\mathrm{dR}(\mathcal{C}^\rig,\mathcal{L}_{n,n}^\rig,\nabla)\simeq 
H^1_\mathrm{dR}(\mathcal{C}^\rig,\mathcal{V}_{n,n})\] as filtered Frobenius monodromy modules. 

We now describe the monodromy operator. 
Let $\mathcal{T}$ 
denote the Bruhat-Tits tree of $\PGL_2(\Q_p)$, and denote $\overrightarrow{\mathcal{E}}$ and $\mathcal{V}$ the set of oriented edges and vertices of $\mathcal{T}$, respectively. If $e=(v_1,v_2)\in\overrightarrow{\mathcal{E}}$, we denote by $\overline{e}$ the oriented edge $(v_2,v_1)$. Let $C^0(V_n)$ be the set of maps $\mathcal{V}\rightarrow V_n$ 
and 
$C^1( V_n)$ the set of maps 
$\overrightarrow{\mathcal{E}}\rightarrow V_n$ such that $f(\overline{e})=-f(e)$ for all $e\in\overrightarrow{\mathcal{E}}$. 
The group $\Gamma_p$ acts on $f\in C^i( V_n)$ by $\gamma(f)=\gamma\circ f\circ \gamma^{-1}$. Let 
\[\epsilon\colon C^1( V_n)^\Gamma\longrightarrow H^1(\Gamma,M)\] be the connecting homomorphism arising from the short exact sequence
$$0\longrightarrow  V_n\longrightarrow C^0( V_n)\overset{\delta}{\longrightarrow} C^1(M)\longrightarrow 0,$$
where $\delta$ is the homomorphism defined by $\delta(f)(e)=f(V_n)-f(v_2)$ for $e=(V_n,v_2)$.
The map $\epsilon$ induces the following isomorphism that we also denote by $\epsilon$
$$\epsilon\colon C^1( V_n)^\Gamma \big/ C^0( V_n)^\Gamma\longrightarrow H^1( V_n).$$

Let $A_e\subset\HH_p^\unr$ be the oriented annulus in $\HH_p$ corresponding to $e$ and $U_v\subset\HH_p^\unr$ be the affinoid corresponding to $v\in\mathcal{V}$, which are obtained as inverse images of the reduction map (see \cite[page 342]{IS}). 
Let $\omega$ be a $V_n$-valued $\Gamma$-invariant  differential of the second kind on $\HH_p$. We define $I(\omega)$ to be the map which assigns to an oriented edge $e\in\overrightarrow{\mathcal{E}}$ the value $I(\omega)(e)=\mathrm{Res}_e(\omega)$, where $\mathrm{Res}_e$ denotes the annular residue along $A_e$. 
If $\omega$ is exact, $I(\omega)=0$. Thus $I$ gives a well-defined map
\begin{equation}\label{I}
I\colon H^1_\mathrm{dR}(\mathcal{C}^\rig,\mathcal{V}_n)\longrightarrow C^1(V_n)^{\Gamma_p}.\end{equation}
Since the set $\{U_v\}_{v\in\mathcal{V}}$ is an admissible covering of $\HH_p$, the Mayer-Vietoris sequence yields an embedding
$$C^1(V_n)^{\Gamma_p}\big/C^0(V_n)^{\Gamma_p}\longmono H^1_\mathrm{dR}(\mathcal{C}^\rig,\mathcal{V}_n).$$
Precomposing with $\epsilon$, we obtain an embedding
\begin{equation}\label{iota}
\iota\colon H^1\left(\Gamma_p,V_n\right)\longmono H^1_\mathrm{dR}(\mathcal{C}^\rig,\mathcal{V}_n)\end{equation}
This map admits a natural left inverse
\begin{equation}\label{P}
P\colon H^1_\mathrm{dR}(\mathcal{C}^\rig,\mathcal{V}_n)\longrightarrow H^1\left(\Gamma_p,V_n\right),\end{equation} which takes  
$\omega$ to the class of the cocycle $\gamma\mapsto\gamma(F_\omega)-F_\omega$, where $F_\omega$ is a Coleman primitive of $\omega$ satisfying $\nabla(F_\omega)=\omega$.
 
Define now the monodromy operator $N_n$ on $H^1_\mathrm{dR}(\mathcal{C}^\rig,\mathcal{V}_n)$ as the composite $\iota\circ(-\epsilon)\circ I$.
The monodromy operator $N_{S_n}$ on the filtered $(\phi,N)$-module 
\[S_n=\mathrm{Sym}^neH^1_\mathrm{dR}(A_0^m)\] is trivial. Therefore, the monodromy operator on $D_{\mathrm{st},\hat{\Q}_p^\mathrm{unr}}(H^{2n+1}_p(\mathcal{D}))$ is given by
\begin{equation}\label{monodromy}
N=\mathrm{id}_n\otimes N_{S_n}+N_n\otimes\mathrm{id}_{S_n}.\end{equation}

We now describe the Frobenius operator on $D_{\mathrm{st},\hat{\Q}_p^\mathrm{unr}}(H^{2n+1}_p(\mathcal{D}))$. First, 
$H^1(\Gamma_p,V_n)$ has a Frobenius endomorphism induced by the map $p^\frac{n}{2}\otimes \sigma$ on $V_n$, where $\sigma$ denotes the absolute Frobenius automorphism on $\hat{\Q}_p^\mathrm{unr}$.
By \cite[page 348]{IS}, there exists a unique operartor 
$\Phi_n$ on $H^1_\mathrm{dR}(\mathcal{C}^\rig,\mathcal{V}_n)$ satisying $N_n\Phi_n=p\Phi_n N_n$ and which is compatible, with respect to $\iota$, 
with the Frobenius on $H^1(\Gamma_p,V_n)$.
On the other hand, the Frobenius on the filtered $(\phi,N)$-module $S_n$ is given by $\Phi_{S_n}=p^{\frac{n}{2}}\otimes\sigma$ acting on the underlying vector space $S_n$.
The Frobenius operator on $D_{\mathrm{st},\hat{\Q}_p^\mathrm{unr}}(H^{2n+1}_p(\mathcal{D}))$ is then given by
$$\Phi=\Phi_n\otimes \Phi_{S_n}.$$
Note that $N$ and $\Phi$ satisfy the relation $N\Phi=p\Phi N$.

For any $D\in\mathrm{MF}_{\hat{\Q}_p^\unr}^{\phi,N}$, write 
$D=\oplus_{\lambda}D_\lambda$ for its slope decomposition, 
where $\lambda\in\Q$
(\cite[(2)]{IS}). 
%
We now note that $N$ induces an isomorphism 
\begin{equation}\label{lemmaslope}
N:H^1_\mathrm{dR}(\mathcal{C}^\rig,\mathcal{V}_{n,n})_{n+1}\simeq 
H^1_\mathrm{dR}(\mathcal{C}^\rig,\mathcal{V}_{n,n})_{n}.\end{equation} 
To see this, note that since the monodromy operator $N$ and the Frobenius $\Phi$ on $H^1_\mathrm{dR}(\mathcal{C}^\rig,\mathcal{V}_{n,n})$ satisfy the relation $N\Phi=p\Phi N$, we have 
$N\left(H^1_\mathrm{dR}(\mathcal{C}^\rig,\mathcal{V}_{n,n})_{n+1}\right)\subseteq
H^1_\mathrm{dR}(\mathcal{C}^\rig,\mathcal{V}_{n,n})_{n}.$ Since 
$S_n$ is isotypical of slope $n/2$, we have 
$$H^1_\mathrm{dR}(\mathcal{C}^\rig,\mathcal{V}_{n,n})_{n+1}=H^1_\mathrm{dR}(\mathcal{C}^\rig,\mathcal{V}_{n,n})_{n/2+1}\otimes S_n$$
and 
$$H^1_\mathrm{dR}(\mathcal{C}^\rig,\mathcal{V}_{n,n})_{n}=H^1_\mathrm{dR}(\mathcal{C}^\rig,\mathcal{V}_{n,n})_{n/2}\otimes S_n.$$  
By \cite[Lemma 6.1]{IS}, the operator $N_n$ on $H^1_\mathrm{dR}(\mathcal{C}^\rig,\mathcal{V}_n)_{n+1}$ is an isomorphism, the therefore the same is true in \eqref{lemmaslope}
by the definition of the monodromy operator $N$ given in \eqref{monodromy}.

\subsubsection{Calculation of Ext groups by the Gysin sequence}For $f\in M_k(\Gamma)$ and $v\in S_n$, 
\eqref{lemmaslope} allows us to use the description in \cite[Lemma 2.1]{IS} to calculate some Ext groups. 

Define $U_{z_A}=\mathcal{C}^\rig-\{z_A\}$, and  put 
\begin{equation}\label{iso2}
H^1_\mathrm{dR}(U_{z_A},\mathcal{V}_{n,n})=
H^1_\mathrm{dR}(U_{z_A},\mathcal{V}_{n})\otimes S_n.\end{equation}
Let $\mathrm{Res}_{z_A}:H^1_\mathrm{dR}(U_{z_A},\mathcal{V}_{n}^\rig)\rightarrow (\mathcal{V}_n^\rig)_{z_A}$ be the residue map at a point $z_A$.  
We have the Gysin sequence (\cite[Theorem 5.13]{IS})
in $\mathrm{MF}^{\phi,N}_{\hat{\Q}_p^\unr}$: 
\begin{equation}\label{5.13IS}
0\longrightarrow H^1_\mathrm{dR}(\mathcal{C}^\rig,\mathcal{V}_{n,n})[-(n+1)]\longrightarrow
H^1_\mathrm{dR}(U_{z_A},\mathcal{V}_{n,n})[-(n+1)]\overset{\mathrm{Res}_{z_A}}\longrightarrow
\left(V_n\otimes S_n\right)[-n]\longrightarrow 0\end{equation}
where we write $V_n$ for the stalk $(\mathcal{V}_n)_{z_A}$ of $\mathcal{V}_n$ at $z_A$. 

We have the cycle class map 
{\[
\cl_A=\cl_{(A^m\times A_0^m,\epsilon_M)}^{(n)}: \CH^{n}((A^m\times A_0^m,\epsilon_{A\times A_0}))\longrightarrow \Gamma\left((V_n\otimes S_n)[-n]\right)\]}
where for a filtered Frobenius monodromy module $M$ with filtration $F^\bullet(M)$, Frobenius $\phi$ and monodromy $N$, we put $\Gamma(M)=F^0(M)\cap M^{\phi=\mathrm{id},N=0}$. 
Next, from \eqref{5.13IS} we obtain a connecting homomorphism in the sequence of $\mathrm{Ext}$ groups  
\[\begin{split}
\Gamma\left((V_n\otimes S_n)[-n]\right)\overset\partial\longrightarrow &
\mathrm{Ext}^1_{\mathrm{MF}_{\hat{\Q}_p^\unr}^{\phi,N}}\left(\hat{\Q}_p^\unr,H^1_\mathrm{dR}(\mathcal{C}^\rig,\mathcal{V}_{n,n})_{n}[-(n+1)]\right)\\
&\simeq \mathrm{Ext}^1_{\mathrm{MF}_{\hat{\Q}_p^\unr}^{\phi,N}}\left(\hat{\Q}_p^\unr[n+1],H^1_\mathrm{dR}(\mathcal{C}^\rig,\mathcal{V}_{n,n}))_{n}\right)\\
&\simeq \left(M_k(\Gamma)\otimes S_n\right)^\vee\end{split}\]
where the last isomorphism comes from \cite[Lemma 2.1]{IS}.   
On the other hand, we have a canonical map 
\[
i:\CH^{n}((A^m\times A_0^m,\epsilon_{A\times A_0}))\longrightarrow \CH^{n+1}(\mathcal{D}) .\]
The definition of the Abel-Jacobi map shows 
that the following diagram is commutative: 
\begin{equation}\label{IS7.3}\xymatrix{
i^{-1}\left(\CH^{n+1}_0(\mathcal{D})\right)\ar[r]^-{\cl_A}\ar[d]^i  & 
\Gamma\left((V_n\otimes S_n)[-n]\right)\ar[d]^\partial\\
\CH^{n+1}_0(\mathcal{D})\ar[r]^-{\AJ_p}&\left(M_k(\Gamma)\otimes S_n\right)^\vee.}
\end{equation}
Then $\AJ_p(\Delta_\varphi)$ is the 
extension class determined by the following diagram (in which the right square is cartesian) 
\begin{equation}\label{diagramAJ}
\xymatrix{0\ar[r] & 
H^1_\mathrm{dR}(\mathcal{C}^\rig,\mathcal{V}_{n,n})\ar[r]^{j_*} & 
H^1_\mathrm{dR}(U_{z_A},\mathcal{V}_{n,n})\ar[r]^-{{\mathrm{Res}_{z_A}}}& 
\left((V_n\otimes S_n)\right)[1]\ar[r]& 0\\
0\ar[r] & 
H^1_\mathrm{dR}(\mathcal{C}^\rig,\mathcal{V}_{n,n})\ar[r] \ar@{=}[u]& 
E\ar[u]\ar[r]&\hat{\Q}_p^\unr[n+1]\ar[u] \ar[r]&0}\end{equation}
where the vertical left map sends $1\longmapsto\cl_A(\Delta_\varphi)[n+1]$.

\subsubsection{Computation of the Abel-Jacobi map}
Choose $\alpha\in H^1_\mathrm{dR}(U_{z_A},\mathcal{V}_{n,n})_{n+1}$ such that 
\begin{equation}\label{res and class}
\mathrm{Res}_{z_A}(\alpha)=\cl_A(\Delta_\varphi)\end{equation} 
and $N(\alpha)=0$.  
Choose $\beta$ in $H^1_\mathrm{dR}(\mathcal{C}^\rig,\mathcal{V}_{n,n})$ such that 
\[j_*(\beta)\equiv\alpha\mod F^{n+1}\left(H^1_\mathrm{dR}(U_{z_A},\mathcal{V}_{n,n})\right).\]
Then the image of the extension $\cl_{A}(\Delta_\varphi)$ in 
\[H^1_\mathrm{dR}(\mathcal{C}^\rig,\mathcal{V}_{n,n})
/F^{n+1}\left(H^1_\mathrm{dR}(\mathcal{C}^\rig,\mathcal{V}_{n,n})\right)\simeq (M_k(\Gamma)\otimes S_n)^\vee\] is the class of $\beta$ (which we denote by the same symbol $\beta$) in this quotient
(for the isomorphism, see \cite[Proposition 6.1]{IS}).  

We have the Poincar\'e duality pairing $\langle,\rangle_{n,n}:\mathcal{V}_{n,n}\otimes_{\mathcal{O}_{\mathcal{C}^\rig}} \mathcal{V}_{n,n}\rightarrow\mathcal{O}_{\mathcal{C}^\rig}$ arising from Poincar\'e duality on the fibers $A\times A_0$. We therefore obtain a Poincar\'e pairing, still denoted $\langle,\rangle_{n,n}$,  
\[\langle,\rangle_{n,n}:H^1_\mathrm{dR}(\mathcal{C}^\rig,\mathcal{V}_{n,n})\times 
H^1_\mathrm{dR}(\mathcal{C}^\rig,\mathcal{V}_{n,n})
\overset\cup\longrightarrow H^1_\mathrm{dR}(\mathcal{C}^\rig,\mathcal{V}_{n,n}\otimes\mathcal{V}_{n,n})
\overset{\langle,\rangle_{n,n}}\longrightarrow
H^2_\mathrm{dR}(\mathcal{C}^\rig,\mathcal{O}_{\mathcal{C}^\rig})\simeq\hat{\Q}_p^\unr.\] 
Let $\omega_f$ be the class in 
$F^{n+1}\left(H^1_\mathrm{dR}(\mathcal{C}^\rig,\mathcal{V}_{n})\right)$ corresponding to 
$f\in M_k(\Gamma)$. Then by definition 
\begin{equation}\label{first equality AJ}
\AJ_p(\Delta_\varphi)(f\otimes v)=\langle\omega_f\otimes v,\beta\rangle_{n,n}.\end{equation}
By \cite[(39)]{IS} and the fact that  
$S_n$ is isotypical of slope $n/2$, we obtain a decomposition 
\[H^1_\mathrm{dR}(\mathcal{C}^\rig,\mathcal{V}_{n,n})\simeq H^1_\mathrm{dR}(\mathcal{C}^\rig,\mathcal{V}_{n,n})_n\oplus F^{n+1}\left(H^1_\mathrm{dR}(\mathcal{C}^\rig,\mathcal{V}_{n,n})\right).\]
We may therefore assume that the element $\beta$ considered above belongs to $H^1_\mathrm{dR}(\mathcal{C}^\rig,\mathcal{V}_{n,n})_n$. 

We now compute $\langle\omega_f\otimes v,\beta\rangle_{n,n}$. 
By \cite[Theorem 6.4]{IS}, 
\begin{equation}\label{eq25}
\ker(N_n)=\iota\left(H^1(\Gamma,V_n))\right)=H^1_\mathrm{dR}(\mathcal{C}^\rig,\mathcal{V}_{n})_{n/2}.\end{equation} 
To simplify the notation we put 
\[H^1(\Gamma,V_{n,n})=H^1\left(\Gamma,V_n\right)\otimes S_n.\]
We now extend $\iota$ to a map, still denoted by the same symbol, 
\[\iota=\iota\otimes\mathrm{id}_{(\mathcal{V}_n)_{z_{A_0}}}\colon H^1(\Gamma,V_{n,n})\longmono H^1_\mathrm{dR}(\mathcal{C}^\rig,\mathcal{V}_{n,n})\]
and \eqref{eq25} shows that there exists an isomorphisms 
 $\ker(N)=\iota\left(H^1(\Gamma,V_{n,n})\right)$, 
so we may also assume $\beta=\iota(c)$ for some $c\in  H^1(\Gamma,V_{n,n})$. 
Let $C_\mathrm{har}(V_n)^\Gamma$ denote the 
$\Q_p$-vector space of $\Gamma$-invariant $V_n$-valued 
harmonic cocycles and denote 
\[\langle,\rangle_{\Gamma}': 
C_\mathrm{har}(V_{n})^\Gamma\otimes H^1(\Gamma,V_n)\longrightarrow \Q_p\]
the pairing introduced in \cite[(75)]{IS}. 
To simplify the notation, we set 
\[C_\mathrm{har}(V_{n,n})^\Gamma=C_\mathrm{har}(V_{n})^\Gamma\otimes S_n.\]
We then define the pairing 
\[\langle,\rangle_{\Gamma}: 
C_\mathrm{har}(V_{n,n})^\Gamma\otimes H^1(\Gamma,V_{n,n})\longrightarrow \Q_p\]
by $\langle,\rangle_{\Gamma}=\langle,\rangle_{\Gamma}'\otimes\langle,\rangle_{A_0}$, where $\langle,\rangle_{A_0}$ is the Poincar\'e pairing on $A_0$. 

\begin{lemma} \label{lemma5.2}
$\langle\omega_f\otimes v,\beta\rangle_{n,n}=-\langle I(\omega_f)\otimes v,c\rangle_\Gamma.$\end{lemma}

\begin{proof}
Write $\beta=\sum_i\beta_i\otimes v_i$ and $c=\sum_jc_j\otimes w_j$. 
Recall that $\iota(c)=\beta$ with $\iota$ injective, so that $\iota(\beta_i)=c_i$ and $v_i=w_i$ for all $i$. 
By \cite[Theorem 10.2]{IS} we know that for each $i$ we have 
$\langle \omega_f,\beta_i\rangle_{\mathcal{V}_n}=-\langle I(\omega_f),c_i\rangle'_\Gamma$. 
The definitions of $\langle,\rangle_{\mathcal{V}_{n,n}}$ and $\langle,\rangle_\Gamma$ imply the result. 
\end{proof}

Write $\alpha-j_*(\beta)=\sum_i\gamma_i\otimes v_i$. 
For each $i$, let $\chi_i$ be a $\Gamma$-invariant 
${V}_{n}$-valued meromorphic differential 
form on $\mathcal{H}_p$ which is holomorphic outside 
$\pi^{-1}(U_{z_A})$, with a simple pole at $z_A$, 
and whose class $[\chi_i]$ in $F^{\frac{n}{2}+1}\left(H^1_\mathrm{dR}(U_{z_A},\mathcal{V}_{n})\right)$
represents $\gamma_i$. Then the class of $\chi=\sum_i\chi_i\otimes v_i$ represents $\alpha-j_*(\beta)$. 

Having identified 
$H^1_\mathrm{dR}(\mathcal{C}^\rig,\mathcal{V}_{n})$ with the $\hat{\Q}_p^\unr$-vector space of $\Gamma$-invariant $V_n$-valued differential forms of the second kind on $\mathcal{H}_p$ modulo horizontal forms for $\nabla$, denote $F_{\omega_f}\in H^0_\mathrm{dR}(\mathcal{C}^\rig,\mathcal{V}_{n})$ the Coleman primitive of $\omega_f$ (\cite[\S 2.3]{deshalit}).  

\begin{lemma}\label{lemma5.3} 
$-\langle I(\omega_f)\otimes v,c\rangle_\Gamma=\langle F_{\omega_f}(z_A)\otimes v,\mathrm{Res}_{z_A}(\chi)\rangle_{A\times A_0}$, where $\langle,\rangle_{A\times A_0}$ is the Poincar\'e pairing on $A\times A_0$.\end{lemma}

\begin{proof}
As in the proof of Lemma \ref{lemma5.2} write $c=\sum_jc_j\otimes w_j$. By definition, 
\[\langle I(\omega_f)\otimes v,c\rangle_\Gamma=\sum_j\langle I(\omega_f),c_j\rangle_\Gamma'\cdot\langle v,w_j\rangle_{A_0}.\]
By \cite[Corollary 10.7]{IS},  
\[\langle I(\omega_f),c_j\rangle_\Gamma'=\langle F_{\omega_f}(z_A),\mathrm{Res}_{z_A}(\chi_j)\rangle_{A}\]
where in the last pairing is the Poincar\'e pairing on $A$; note that we can apply 
\cite[Corollary 10.7]{IS} because the proof of \cite[Theorem 10.6]{IS} still holds in our setting taking because 
the analogues of \cite[(87), (88)]{IS} are true. 
The result 
follows now from the definition of the pairing $\langle,\rangle_{\mathcal{V}_{n,n}}$.
\end{proof}

\begin{lemma}\label{lemma9.2} $
\AJ_p(\Delta_\varphi)(\omega_f\otimes \omega_{A_0}^{j}\eta_{A_0}^{n-j})=\langle F_f(z_A)\otimes\omega_{A_0}^{j}\eta_{A_0}^{n-j},\mathrm{cl}_{z_A}(\Delta_\varphi)\rangle_{A\times A_0}.$\end{lemma} 

\begin{proof} Taking into account \eqref{res and class} and \eqref{first equality AJ}, this follows from combining Lemma \ref{lemma5.2} and Lemma \ref{lemma5.3}. 
\end{proof}

\begin{lemma}\label{lemma9.3} $\langle F_f(z_A)\otimes\omega_{A_0}^{j}\eta_{A_0}^{n-j},\mathrm{cl}_{z_A}(\Delta_\varphi)\rangle=\langle\varphi^*(F_f(z_A)),\omega_{A_0}^{j}\eta_{A_0}^{n-j}\rangle_{A_0}$.
\end{lemma}

\begin{proof}
This follows from \ref{lemma9.2} and the functoriality property of the Poincar\'e pairing, as in \cite[Proposition 3.21]{BDP}. 
\end{proof}

Let $\omega_A\in eH^1_\mathrm{dR}(A/F)$ be such that 
$\omega_{A_0} =\varphi^*\omega_{A}$ and let $\eta_A\in eH^1(A,\mathcal{O}_{A})$ be as before the dual class to $\omega_{A}$ under the Poincar\'e duality pairing $\langle\ ,\ \rangle_{A}$, normalised so that 
$\langle\omega_{A},\eta_{A}\rangle_{A}=1$. 

\begin{proposition}\label{proposition9.4}
$
\AJ_p(\Delta_\varphi)(\omega_f\otimes\omega_{A_0}^{j}\eta_{A_0}^{n-j})=d_\varphi^{j}\langle F_f(z_A),\omega_{A}^{j}\eta_{A}^{n-j}\rangle_A
$, where $d_\varphi$ is the degree of $\varphi$. 
\end{proposition}
\begin{proof}
Observe that $\varphi^*\eta_{A}=d_\varphi\eta_{A_0}$. 
It follows that 
\begin{equation}\label{eq45}\langle\varphi^*(F_f(z_A)),\omega_{A_0}^{j}\eta_{A_0}^{n-j}\rangle_{A_0} =d_\varphi^{j-n}\langle\varphi^*(F_f(z_A)),\varphi^*(\omega_{A}^{j}\eta_{A}^{n-j})\rangle_{A_0}.\end{equation}
The functoriality properties of Poincar\'e pairing show that 
$\langle\varphi^*\omega,\varphi^*\eta\rangle_{A_0}=d_\varphi\langle\omega,\eta\rangle_{A}$ for 
$\omega$ and $\eta$ in $H^1_\mathrm{dR}(A/F)$. 
It follows that 
\begin{equation}\label{eq46}\langle\varphi^*(F_f(z_A)),\varphi^*(\omega_{A}^{j}\eta_{A}^{n-j})\rangle_{A_0}= d_\varphi^n\langle F_f(z_A),\omega_{A}^{j}\eta_{A}^{n-j}\rangle_A.
\end{equation}
The result follows combining Lemma \ref{lemma9.2} and Lemma \ref{lemma9.3} with equations \eqref{eq45} and \eqref{eq46}. \end{proof}

Recall the canonical differentials $\omega_\can$, $\eta_\can$ introduced in Section \ref{section Coleman primitive}. Since $A_0$ and $A$ have CM by $\mathcal{O}_K$, $\omega_\can$ and $\eta_\can$ defined pair of differentials $\omega_{A_0}$, $\eta_{A_0}$ and $\omega_{A}$, $\eta_A$. 
Recall now the definition of the function $G_j$ in \eqref{def G_j}, and for integers $j=n/2,\dots,n$ define 
the function \[H_j(z)=\langle F_{f}(z),\omega_\mathrm{can}^j(z)\eta_\mathrm{can}^{n-j}(z)\rangle.\]

\begin{theorem}\label{AJ theorem} Let $\omega_{A_0}$, $\eta_{A_0}$, $\omega_A$ and $\eta_A$ be defined by means of $\omega_\can$ and $\eta_\can$. Then for each $j=n/2,\dots,n$ we have 
\[H_j(z_A)=\AJ(\Delta_\varphi)(\omega_f\otimes\omega_{A_0}^{j}\eta_{A_0}^{n-j}).\] 
\end{theorem}

\begin{proof}
The differentials on $A_0$ and $A$ thus defined 
satisfy the conditions $\langle\omega_{A_0},\eta_{A_0}\rangle_{A_0}=1$, $\langle\omega_{A},\eta_A\rangle_A=1$ and $\varphi^*\omega_A=\omega_{A_0}$. Therefore we may 
apply Proposition \ref{proposition9.4} and the result follows. 
\end{proof}

\begin{corollary}\label{coro9.6}
Let $\omega_{A_0}$, $\eta_{A_0}$, $\omega_A$ and $\eta_A$ be defined by means of $\omega_\can$ and $\eta_\can$. Then for each $j=n/2,\dots,n$ we have 
\[\delta^{n-j}_p\left(H_n\right)(z_A)=\frac{n!}{j!}\AJ(\Delta_\varphi)(\omega_f\wedge \omega_{A_0}^{j}\eta_{A_0}^{n-j}).\] 
\end{corollary}

\begin{proof}
From the proof of Theorem \ref{lemma8.1} we see that  
$\Theta_p(G_j(z))=jG_{j-1}(z)$, and therefore we have  
$\Theta_p^{n-j}(G_n(z))= \frac{n!}{j!}G_{j}(z)$. Therefore, 
$\delta_p^{n-j}(H_n(z))= \frac{n!}{j!}H_{j}(z)$
The result follows from Theorem \ref{AJ theorem}. 
\end{proof}

\bibliographystyle{amsalpha} 
\bibliography{references}
\end{document}